\documentclass[final,12pt]{arxiv2025} % 

\usepackage{times}

\usepackage{bigints}
\usepackage{enumitem}
\usepackage{nicefrac}
\usepackage{bm}
\usepackage{mathtools}

\newtheorem{theorem}{Theorem}

\newtheorem{lemma}{Lemma}

\newtheorem{proposition}{Proposition}

\newtheorem{corollary}{Corollary}

\newtheorem{fact}{Fact}

\newtheorem{remark}{Remark}

\newtheorem{definition}{Definition}

\newcommand{\op}{\mathsf{op}}

\newcommand{\dvert}{\mathbin{\|}}

\newcommand{\bP}{\bm{\mathrm{P}}}

\renewcommand{\P}{\mathcal{P}}
\newcommand{\R}{\mathbb{R}} 
\newcommand{\N}{\mathcal{N}}
\newcommand{\M}{\mathsf{M}}
\newcommand{\E}{\mathop{\mathbb{E}}}

\newcommand{\cC}{\mathcal{C}}

\newcommand{\cK}{\mathcal{K}}

\newcommand{\cO}{\mathcal{O}}
\newcommand{\cW}{\mathcal{W}}

\newcommand{\nn}{\nonumber}

\newcommand{\bbN}{\mathbb{N}}
\newcommand{\HS}{\mathrm{HS}}
\newcommand{\Cov}{\mathrm{Cov}}

\newcommand{\Var}{\mathop{\mathrm{Var}}}
\newcommand{\Ent}{\mathsf{Ent}}
\newcommand{\KL}{\mathsf{KL}}

\newcommand{\FI}{\mathsf{FI}}
\newcommand{\D}{\,\mathrm{d}}

\newcommand{\law}{\mathsf{law}}

\newcommand{\Lip}{\mathsf{Lip}}

\newcommand{\sfD}{\mathsf{D}}
\newcommand{\sfH}{\mathsf{H}}

\newcommand{\prox}{\textnormal{\!prox}}
\newcommand{\MI}{\mathsf{MI}}

\newcommand{\PSI}{\Phi \mathsf{SI}}

\newcommand{\ULA}{\mathsf{ULA}}
\newcommand{\Dir}{\mathcal{E}}

\usepackage{centernot}
\newcommand{\nll}{\centernot{\ll}}

\title[Contraction of $\Phi$-Mutual Information along the Langevin Dynamics and Algorithms]{Characterizing Dependence of Samples along the Langevin Dynamics and Algorithms via Contraction of $\Phi$-Mutual Information}

\coltauthor{\Name{Jiaming Liang} \Email{jiaming.liang@rochester.edu}\\
\addr{Goergen Institute for Data Science \& Department of Computer Science, University of Rochester}
\AND
 \Name{Siddharth Mitra} \Email{siddharth.mitra@yale.edu}\\
 \Name{Andre Wibisono} \Email{andre.wibisono@yale.edu}\\
 \addr Department of Computer Science, Yale University
 }

\begin{document}

\maketitle

\begin{abstract}%

The mixing time of a Markov chain determines how fast the iterates of the Markov chain converge to the stationary distribution; however, it does not control the dependencies between samples along the Markov chain.
In this paper, we study the question of how fast the samples 
become approximately independent
along popular Markov chains for continuous-space sampling: the Langevin dynamics in continuous time, and the Unadjusted Langevin Algorithm and the Proximal Sampler in discrete time.
We measure the dependence between samples via $\Phi$-mutual information, which is a broad generalization of the standard mutual information, and which is equal to $0$ if and only if the the samples are independent.
We show that along these Markov chains, the $\Phi$-mutual information between the first and the $k$-th iterate decreases to $0$ exponentially fast in $k$ when the target distribution is strongly log-concave. 
Our proof technique is based on showing the Strong Data Processing Inequalities (SDPIs) hold along the Markov chains.
To prove fast mixing of the Markov chains, we only need to show the SDPIs hold for the stationary distribution.
In contrast, 
to prove the contraction of $\Phi$-mutual information, we need to show the SDPIs hold along the entire trajectories of the Markov chains;
we prove this when the iterates along the Markov chains satisfy the corresponding $\Phi$-Sobolev inequality, which is implied by the strong log-concavity of the target distribution.
\end{abstract}

\begin{keywords}%
    Markov chain, Langevin dynamics, unadjusted Langevin algorithm, proximal sampler, $\Phi$-mutual information, $\Phi$-Sobolev inequality, strong data processing inequality
\end{keywords}

\footnotetext[1]{Accepted for presentation at the Conference on Learning Theory (COLT) 2025}

\setcounter{tocdepth}{2}
\tableofcontents

\section{Introduction}\label{sec:Introduction}

Consider the task of sampling from a target probability distribution $\nu \propto \exp{(-f)}$ supported on $\R^d$.
This is a fundamental algorithmic question appearing in many fields including machine learning, statistics, and Bayesian inference~\citep{gelman1995bayesian, von2011bayesian, johannes2003mcmc}. 
In settings where exact sampling from $\nu$ is not possible, a common approach is to construct a Markov chain with $\nu$ as its stationary distribution, and output the samples after an initial waiting (``burn-in'') period when the chain has approximately mixed. 
When implementing a Markov chain to draw samples, there are many considerations to make in order to obtain strong statistical guarantees, for instance, how many chains to run, how long to wait before outputting a sample, and where to start the chains from~\citep{gilks1995markov, kass1998markov}.

The mixing time of a Markov chain tracks how fast the iterate along the Markov chain converges to the stationary distribution, and thus it controls the burn-in period, i.e., how long to wait before obtaining a useful sample~\citep{gilks1995introducing, geyer2011introduction,levin2017markov}.
The mixing time can be defined with respect to a statistical divergence between probability distributions that we use to measure the error; this includes for example the Total Variation (TV) distance, Kullback-Leibler (KL) divergence, and chi-squared divergence, all of which are instances of a general family of $\Phi$-divergences induced by convex function $\Phi$ (see Definition~\ref{def:PhiDivergence}, and~\citep{chafai2004entropies, raginsky2016strong}).
Existing results in the literature have established good mixing time guarantees for Markov chains in various divergences, see e.g.,~\citep{levin2017markov, montenegro2006mathematical} for discrete-space Markov chains, and~\citep{chewi2023log} for continuous-space Markov chains.

The mixing time of a Markov chain only tracks how close the last iterate is from the stationary distribution;
however, it does not characterize the dependency between iterates along the Markov chain.
Even if each iterate along the Markov chain has the correct distribution (e.g., when we initialize the Markov chain from the stationary distribution), the dependencies between successive iterates can be large.
In many applications, we may want to generate multiple samples from the target distribution which are \textit{approximately independent}.
If we have the ability to run multiple independent chains, then we can produce multiple independent samples.
However, if we can only run \textit{one} Markov chain, then a natural strategy is to output a subsequence of the iterates along the Markov chain, by dropping several successive iterates until the next sample that we output is approximately independent from the current sample.
For this strategy to work, we need an estimate of how fast the dependence between the iterates is decreasing, or equivalently, how fast the iterates along a Markov chain become approximately independent.
This is the central question that we study in this paper.

There are various ways to quantify the dependence between samples.
A simple way is to control the covariance or the correlation between the samples; this may be sufficient for applications where we want to approximate the expectation of some function via the empirical average from the samples.
It is well-known that to control the correlation decay along a Markov chain, a spectral gap or a Poincar\'e inequality is sufficient; see e.g.,~\citep{madrasslade93,madras02markov}, see also Appendix~\ref{app:CovarianceDecay} for a review.
However, in other applications where we need to estimate more complicated properties of the distribution, such as for entropy or density estimation, we may need to have a stronger control on the dependence between the samples.

A natural measure of dependence in information theory is \textit{mutual information}, which is the KL divergence between the joint distribution of the samples and the product of the marginal distributions; thus, mutual information is non-negative, and it is $0$ when the samples are independent.
In some applications such as density estimation, we need a control not in the standard mutual information (induced by the KL divergence), but in a generalized notion including the Hellinger mutual information (induced by the Hellinger divergence)~\citep[Section~4.3]{zhang2021convergence} or the chi-squared mutual information (induced by the chi-squared divergence)~\citep{goldfeld2019estimating}.
This motivates our study in this paper where we measure the dependence between samples via the \textit{$\Phi$-mutual information} induced by the $\Phi$-divergence functional arising from any (twice-differentiable) strictly convex function $\Phi$; see Definition~\ref{def:PhiMutualInformation}.
The $\Phi$-mutual information is a stronger measure of dependence than covariance, in that a small $\Phi$-mutual information between samples implies that their covariance is small; see Lemma~\ref{lem:CovPhiMIBound} in Appendix~\ref{app:CovPhiMI}.

We study three Markov chains which are popular for continuous-space sampling: the (overdamped) Langevin dynamics~\eqref{eq:LangevinDynamics} in continuous time, and its time discretizations -- the Unadjusted Langevin Algorithm (ULA)~\eqref{eq:ULA} and the Proximal Sampler~\eqref{eqs:ProximalSampler} algorithm; see Section~\ref{Sec:MCs} for the definitions.
Mixing time guarantees for these Markov chains in various divergences have been well-studied, for the continuous-time Langevin dynamics in e.g.,~\citep{chafai2004entropies, villani2009optimal, dolbeault2018varphi, cao2019exponential}; for the ULA in e.g.,~\citep{dalalyan2017further, cheng2018convergence, durmus2019analysis, VW19, mitra2024fast}; and for the proximal sampler in e.g.,~\citep{lee2021structured,chen2022improved,yuan2023class,fan2023improved,mitra2024fast}. 
We provide additional references and discussion on related works in Appendix~\ref{sec:Intro_RelatedWork}.

In this paper, we study the question of when the iterates become approximately independent for these Markov chains, as measured in $\Phi$-mutual information.
We note that in the special case of the standard mutual information, \textit{if} we have a mixing time guarantee for the Markov chain from \textit{any} point mass initialization, then we can deduce a bound on the dependence between the samples; see Lemma~\ref{lem:IndependenceTimeleqMixingTime} in Appendix~\ref{app:IndependenceViaMixing}.
However, this requires mixing from any point mass initialization, which (for continuous-space sampling) is a strong assumption and only known to hold for a few results; furthermore, this argument does not seem to hold for more general $\Phi$-mutual information, see Appendix~\ref{app:IndependenceViaMixing} for further discussion.
We provide guarantees on general $\Phi$-mutual information.

\paragraph{Contributions.}
We show that when the target distribution is strongly log-concave, the $\Phi$-mutual information converges exponentially fast to $0$ along all the Markov chains we study: along the Langevin dynamics~\eqref{eq:LangevinDynamics} in continuous time (see Theorem~\ref{thm:LangevinDynamicsMain}), along the ULA~\eqref{eq:ULA} in discrete time under an additional smoothness assumption (see Theorem~\ref{thm:ULA_Main}), and along the Proximal Sampler~\eqref{eqs:ProximalSampler} in discrete time (see Theorem~\ref{thm:ProximalMain}).
Our proof technique proceeds via establishing the strong data processing inequalities along the iterates of the Markov chains, which we describe in Section~\ref{sec:Intro_SDPI}.

\subsection{Strong Data Processing Inequalities}\label{sec:Intro_SDPI}

We briefly discuss our main technique via strong data processing inequalities (SDPIs); we provide a more detailed discussion on SDPIs in Appendix~\ref{app:SDPI}.

\textit{Data processing inequality (DPI)} is a fundamental concept in information theory,
which states that information cannot increase along a noisy channel or a Markov chain. 
Concretely, let $\bP$ be a Markov kernel or transition operator representing the noisy channel, and let $\sfD_{\Phi}$ be the $\Phi$-divergence (Definition~\ref{def:PhiDivergence}) induced by any convex function $\Phi$.
Then the DPI in $\Phi$-divergence states that for any probability distributions $\mu$ and $\pi$, when we apply the same Markov kernel $\bP$ to both $\mu$ and $\pi$, the $\Phi$-divergence between them cannot increase: $\sfD_\Phi(\mu \bP \dvert \pi \bP) \leq \sfD_\Phi (\mu \dvert \pi)$. 

\textit{Strong data processing inequality (SDPI)}~\citep{raginsky2016strong, polyanskiy2025information, polyanskiy2017strong} is a strengthening of DPI which quantifies the rate at which information is decreasing, and it is typically a function of both the Markov chain and one of the input distributions.
For a Markov kernel $\bP$ and a probability distribution $\pi$, we can define the \textit{contraction coefficient} of $(\bP,\pi)$ in $\Phi$-divergence as:
\begin{equation}\label{eq:ContractionCoefficientMainText}
\varepsilon_{\sfD_{\Phi}}(\bP, \pi) \coloneqq \sup_{\mu} \frac{\sfD_\Phi(\mu \bP \dvert \pi \bP)}{\sfD_\Phi (\mu \dvert \pi)}\,
\end{equation}
where the supremum is over all probability distributions $\mu$ such that $0 < \sfD_{\Phi}(\mu\dvert \pi) < \infty$.
Note $\varepsilon_{\sfD_{\Phi}}(\bP, \pi) \le 1$ by the (weak) DPI.
We say that $(\bP, \pi)$ satisfies an \textit{SDPI in $\Phi$-divergence} if $\varepsilon_{\sfD_{\Phi}}(\bP, \pi) < 1$.
If we can bound $\varepsilon_{\sfD_{\Phi}}(\bP, \nu)$ when $\nu$ is stationary for $\bP$ (i.e., $\nu \bP = \nu$), then we immediately obtain a bound on the mixing time in $\Phi$-divergence along the Markov chain defined by $\bP$; see eq.~\eqref{eq:MixingViaSDPI} in Appendix~\ref{app:SDPI}.
In practice, bounding the contraction coefficient $\varepsilon_{\sfD_{\Phi}}(\bP, \nu)$ is typically the challenging step to prove mixing guarantees via SDPI; see e.g.,~\citep{raginsky2016strong} for the discrete state space setting, and~\citep{mitra2024fast} for the continuous state space setting.

We can equivalently describe DPIs and SDPIs in terms of the decay of mutual information of the iterates along a Markov chain.
Let $\bP$ be a Markov chain and suppose that the iterates along the chain are $X_i \sim \rho_i$ for $i \geq 0$. The data processing inequality in terms of mutual information states that for any $i \geq 0$, $\MI_{\Phi}(X_i; X_{i+2}) \leq \MI_{\Phi}(X_i; X_{i+1})$, i.e. that the $\Phi$-mutual information cannot increase along the chain.
SDPIs in $\Phi$-mutual information follow similarly by defining appropriate contraction coefficients (Definition~\ref{def:ContractionCoefficient_PhiMutualInformation}). After minor calculations, one can relate the drop in $\Phi$-mutual information between $X_0$ and $X_k$ by the product of contraction coefficients in $\Phi$-divergence along the chain (Lemma~\ref{lem:IndependenceContractionBySDPI}). That is, $\MI_{\Phi}(X_0; X_k) \leq \prod_{i=\ell}^k \varepsilon_{\sfD_{\Phi}}(\bP, \rho_i) \, \MI_{\Phi}(X_0; X_\ell)$ for any $\ell \geq 1$ and $k \geq \ell$. 
Thus, controlling the contraction coefficients also helps us control the decay of mutual information; 
however, note that now we need to control the contraction coefficients for distributions \textit{along the trajectory} of the Markov chain (whereas for mixing time, we only need to control the contraction coefficient for the stationary distribution).
This task of controlling the contraction coefficients along the trajectory makes studying the information contraction more subtle and challenging than the task for bounding the mixing time. 
We bound the contraction coefficients along the trajectory for the Langevin dynamics, ULA, and Proximal Sampler in Lemmas~\ref{lem:ContractionCoefficientForLangevinDynamics},~\ref{lem:ULAContractionCoefficientAndPSIEvolution}(a), and~\ref{lem:proximal_ContractionCoefficientBoundAndPSIEvolution}(a) respectively. 
These lemmas show that the contraction coefficients are strictly less than $1$ so long as the distributions along the trajectory satisfy a $\Phi$-Sobolev inequality (Definition~\ref{def:PhiSobolevInequality}). 
Ensuring that the distributions along the trajectory satisfy a $\Phi$-Sobolev inequality is the crucial part where we need the strong log-concavity of the target distribution.

%%%%%%%%%%%%%%%%%%%%%
\section{Preliminaries}\label{sec:Preliminaries}

We say a probability distribution $\nu \propto \exp{(-f)}$ on $\R^d$ with a twice continuously differentiable potential function $f \colon \R^d \to \R$ is \textit{$\alpha$-strongly log-concave} ($\alpha$-SLC) for some $\alpha > 0$ if $\nabla^2 f(x) \succeq \alpha I$ for all $x \in \R^d$; when $\alpha = 0$, we call $\nu$ \textit{weakly log-concave}.
We say $\nu \propto \exp(-f)$ is \textit{$L$-smooth} for some $0 < L < \infty$ if $-LI \preceq \nabla^2 f(x) \preceq LI$ for all $x \in \R^d$.

Throughout, we take $\Phi \colon \R_{\geq 0} \to \R$ to be a twice-differentiable strictly convex function with $\Phi(1) = 0$. The twice-differentiability ensures that the $\Phi$-Fisher information~\eqref{eq:PhiFisherInformation} is well defined, and the strict convexity ensures that the $\Phi$-divergence (Definition~\ref{def:PhiDivergence}) is $0$ if and only if both arguments are the same. Consequently, this ensures that the $\Phi$-mutual information (Definition~\ref{def:PhiMutualInformation}) is $0$ if and only if its arguments are independent.

Let $\P(\R^d)$ denote the space of probability distributions supported on $\R^d$.
Let $\P_{2,\text{ac}}(\R^d)$ denote the space of distributions $\rho \in \P(\R^d)$ with finite second moment and which are absolutely continuous with respect to the Lebesgue measure. 
In this paper, we will consider distributions in $\P_{2,\text{ac}}(\R^d)$, with a minor exception for Dirac distributions. 
We let $\delta_x$ denote the Dirac distribution (point mass) at point $x \in \R^d$; note $\delta_x \in \P(\R^d)$ but $\delta_x \notin \P_{2,\text{ac}}(\R^d)$. 
When $\rho \in \P_{2,\text{ac}}(\R^d)$, we will identify $\rho$ via its density function with respect to the Lebesgue measure, which we also denote by $\rho \colon \R^d \to \R$.
The notation $\mu \ll \nu$ denotes that $\mu$ is absolutely continuous with respect to $\nu$.

\subsection{$\Phi$-Divergence}

The $\Phi$-divergence or $f$-divergence~\citep{sason2016f, polyanskiy2025information} functional is a generalization of many popular statistical divergences such as KL divergence $(\KL(\mu \dvert \nu) = \E_{\mu}[\log \frac{\mu}{\nu}])$ and chi-squared divergence $(\chi^2(\mu \dvert \nu) = \E_{\nu}[(\frac{\mu}{\nu}-1)^2])$. The family of $\Phi$-divergences are defined via a convex function $\Phi: \R_{\geq 0} \to \R$ with $\Phi(1) = 0$ as follows. We further assume that $\Phi$ is twice-differentiable and strictly convex.

\begin{definition}\label{def:PhiDivergence}
    The $\Phi$-divergence between probability distributions $\mu$ and $\nu$ with $\mu \ll \nu$ is defined by
    \begin{equation}\label{eq:defPhiDiv}
        \sfD_{\Phi}(\mu \dvert \nu) \coloneqq  \E_\nu \left[\Phi\left(\frac{\mu}{\nu}\right)\right] = \int_{\R^d} \Phi\left(\frac{\mu(x)}{\nu(x)}\right)\nu(x) \D x.
    \end{equation}
    If $\mu \nll \nu$, then $\sfD_{\Phi}(\mu \dvert \nu) \coloneqq +\infty$\,.
\end{definition}
Since we assume $\Phi$ is strictly convex, the $\Phi$-divergence functional is non-negative, and it is equal to $0$ if and only if both arguments are the same~\citep[Theorem~7.5]{polyanskiy2025information}.
The KL divergence corresponds to the case $\Phi(x) = x \log x$, and the chi-squared divergence corresponds to the case $\Phi(x) = (x-1)^2$; see Table~\ref{table:PhiExamples} in Appendix~\ref{app:ExamplesOfPhiDiv} for further examples. 
The family of $\Phi$-divergences share many common properties such as satisfying data processing inequalities, and are therefore natural to study in a unified manner.

\subsection{$\Phi$-Mutual Information}

Each $\Phi$-divergence induces a $\Phi$-mutual information, in the same way that the KL divergence induces the standard mutual information. 
Recall for a joint random variable $(X,Y) \sim \rho^{XY}$, the (classical) mutual information functional is defined as $\MI(X; Y) \equiv \MI(\rho^{XY}) = \KL(\rho^{XY} \dvert \rho^X \otimes \rho^Y)$. 
We can generalize this to define $\Phi$-mutual information in terms of $\Phi$-divergence as follows.

\begin{definition}\label{def:PhiMutualInformation}
    Given two random variables $X$ and $Y$ on $\R^d$ with joint law $\rho^{XY}$, the $\Phi$-mutual information functional is given by
    \begin{equation}\label{eq:MIPhi}
        \MI_{\Phi}(X; Y) \equiv \MI_\Phi (\rho^{XY}) \coloneqq \sfD_\Phi \left(\rho^{XY} \dvert \rho^X \otimes \rho^Y \right) = \E_{x \sim \rho^X} \left[  \sfD_\Phi \left(\rho^{Y \mid X = x} \dvert \rho^Y \right)  \right].
    \end{equation}
\end{definition}
By the property of $\Phi$-divergence, we see that the $\Phi$-mutual information functional is always non-negative, and it is equal to $0$ if and only if $X$ and $Y$ are independent, i.e., $\rho^{XY} = \rho^X \otimes \rho^Y$. 
Note $\Phi$-mutual information is symmetric, i.e., $\MI_{\Phi}(X; Y) = \MI_{\Phi}(Y; X)$.

\subsection{$\Phi$-Sobolev Inequalities}

We now describe $\Phi$-Sobolev inequalities~\citep{chafai2004entropies, raginsky2016strong}, a family of isoperimetric inequalities which include as special cases popular inequalities such as the log-Sobolev inequality and the Poincaré inequality. 
The family of $\Phi$-Sobolev inequalities also form a natural condition for the exponential convergence of $\Phi$-divergence along the Langevin dynamics, as we review in Lemma~\ref{lem:PhiConvergenceAlongLangevin} in Appendix~\ref{app:LangevinDynamicsBackground}; this generalizes the convergence guarantees of KL divergence along the dynamics under a log-Sobolev inequality, and of chi-squared divergence under a Poincaré inequality.

As stated in Section~\ref{sec:Intro_SDPI}, we bound the contraction coefficients arising in the SDPIs as long as the distributions along the trajectory of the Markov chain satisfy $\Phi$-Sobolev inequalities. Hence, studying these families of inequalities and how they relate to the distributions along the Markov chain is crucial in our approach.

\begin{definition}\label{def:PhiSobolevInequality}
    A probability distribution $\nu$ satisfies a $\Phi$-Sobolev inequality $(\PSI)$ with constant $\alpha > 0$ if for all probability distributions $\mu \ll \nu$\,, we have
    \begin{equation}\label{eq:PhiSI_DistributionBased}
        2 \alpha \, \sfD_\Phi (\mu \dvert \nu) \leq \FI_\Phi (\mu \dvert \nu),
    \end{equation}
    where $\sfD_\Phi (\mu \dvert \nu)$ is defined in~\eqref{eq:defPhiDiv} and $\FI_\Phi (\mu \dvert \nu)$ is the $\Phi$-Fisher information functional defined as
    \begin{equation}\label{eq:PhiFisherInformation}
        \FI_\Phi (\mu \dvert \nu) \coloneqq  \E_{\nu} \left[ \left\|\nabla \frac{\mu}{\nu}\right\|^2 \Phi''\left( \frac{\mu}{\nu} \right) \right] = \int_{\R^d} \left\|\nabla \frac{\mu (x)}{\nu(x)}\right\|^2 \Phi''\left( \frac{\mu(x)}{\nu(x)} \right) \nu(x) \D x.
    \end{equation}    
    We define the $\Phi$-Sobolev constant of $\nu$ to be the optimal (largest) constant $\alpha$ such that~\eqref{eq:PhiSI_DistributionBased} holds, i.e.,
    \begin{equation}\label{eq:PhiSI_Constant}
        \alpha_{\PSI}(\nu) := \inf_{\mu} \frac{\FI_\Phi (\mu \dvert \nu)}{2\sfD_\Phi (\mu \dvert \nu)}
    \end{equation}
    where the infimum is taken over all probability distributions $\mu$ with $0 < \sfD_\Phi(\mu \dvert \nu) < \infty$.
\end{definition}

The $\Phi$-Sobolev inequality recovers the log-Sobolev inequality when $\Phi(x) = x \log x$, and the Poincaré inequality when $\Phi(x) = (x-1)^2$. 
The Poincaré inequality is the weakest $\Phi$-Sobolev inequality in that it is implied by any other $\Phi$-Sobolev inequality~\citep[Section~2.2]{chafai2004entropies}.

    To study how the $\Phi$-Sobolev constant~\eqref{eq:PhiSI_Constant} evolves along the Markov chain, we require further properties of these constants such as how they change under convolutions and pushforwards. We discuss these properties in Appendix~\ref{app:PhiSobolevProperties}.

\subsection{Langevin Dynamics, Unadjusted Langevin Algorithm, and the Proximal Sampler}
\label{Sec:MCs}

Here we formally introduce the Markov chains we will study. References for mixing times of these Markov chains can be found in Section~\ref{sec:Introduction} and Appendix~\ref{sec:Intro_RelatedWork}.

\paragraph{Langevin dynamics} The overdamped Langevin dynamics to sample from $\nu \propto \exp{(-f)}$ is given by the following stochastic differential equation (SDE)
\begin{equation}\label{eq:LangevinDynamics}
    \D X_t = -\nabla f(X_t) \D t + \sqrt{2} \D W_t\,,
\end{equation}
where $W_t$ is standard Brownian motion on $\R^d$.
The Langevin dynamics has $\nu$ as the stationary distribution~\citep{BGL14}, and hence it is a natural process to study for sampling.
However, this dynamics needs to be discretized in time to implement in practice.
We will focus on two discrete-time algorithms.

\paragraph{Unadjusted Langevin Algorithm} A forward Euler discretization of these dynamics gives rise to the Unadjusted Langevin Algorithm (ULA), explicitly given as
\begin{equation}\label{eq:ULA}
    X_{k+1} = X_k - \eta \nabla f(X_k) + \sqrt{2 \eta} Z_k\,,
\end{equation}
where $\eta >0$ is the step-size and $Z_k \sim \N(0,I)$.
Note that as $\eta \to 0$ and $\eta k \to t$, the ULA update~\eqref{eq:ULA} recovers the Langevin dynamics~\eqref{eq:LangevinDynamics}.
However, for each fixed $\eta > 0$, the ULA is biased, i.e., its stationary distribution is $\nu^{\eta} \neq \nu$.
For mixing time analysis, this results in a low-accuracy guarantee where the iteration complexity to reach an error $\epsilon$ in any specified divergence scales polynomially in $\epsilon^{-1}$.
Here, we show that the ULA still decreases the $\Phi$-mutual information exponentially fast, despite the biased limit; see Theorem~\ref{thm:ULA_Main}.

\paragraph{Proximal Sampler} The Proximal Sampler is a discrete-time Gibbs sampling-based algorithm. The Proximal Sampler considers an augmented $(X, Y)$ space $\R^d \times \R^d$ and alternatively samples from the conditional distributions. When referring to the Proximal Sampler, we denote the target distribution as $\nu^X \propto \exp{(-f)}$ on $\R^d$ and in general, use superscripts to denote the space supporting the distribution. The Proximal Sampler then considers the following joint target distribution 
\begin{equation}\label{eq:JointTargetForProximal}
\nu^{XY}(x,y) \propto \exp \left(-f(x) - \frac{\|x-y\|^2}{2\eta} \right),
\end{equation}
for step-size $\eta >0$.
The Proximal Sampler, initialized from $X_0 \sim \rho_0^X$, is the following two-step algorithm
\begin{align}\label{eqs:ProximalSampler}
    \begin{split}
        &\textbf{Step 1 (forward step)\,:} \text{\,\, Sample \,\,} Y_k \mid X_k \sim \nu^{Y \mid X = X_k} = \N(X_k, \eta I),\hspace{3cm}\\
        &\textbf{Step 2 (backward step)\,:} \text{\,\, Sample \,\,} X_{k+1} \mid Y_k \sim \nu^{X \mid Y = Y_k}.
    \end{split}
\end{align}
Observe that $\nu^{XY}$~\eqref{eq:JointTargetForProximal} has the desired target distribution $\nu^X$ as the $X$-marginal.
The forward step is easy to implement as it corresponds to drawing a Gaussian random variable. The backward step can be implemented given access to a Restricted Gaussian Oracle (RGO). 
A RGO is an oracle that, given any $y \in \R^d$, outputs a sample from $\nu^{X \mid Y=y}$, i.e. from
\begin{equation}\label{eq:RGOdistribution}
\nu^{X \mid Y}(x \mid y) \propto_x \exp \left( -f(x)-\frac{\|x-y\|^2}{2\eta}   \right).
\end{equation}
Similar to~\citet{lee2021structured, chen2022improved, mitra2024fast}, we consider the ideal Proximal Sampler which assumes access to a perfect RGO for our main result. We mention further details and background on the Proximal Sampler including a rejection sampling-based RGO implementation in Appendix~\ref{app:ProximalBackground}.

\section{Convergence of $\Phi$-Mutual Information along Langevin Dynamics}

We show that along the continuous-time Langevin dynamics~\eqref{eq:LangevinDynamics} for strongly log-concave target distributions, the $\Phi$-mutual information converges exponentially fast to $0$ as soon as we have an iterate which satisfies a $\Phi$-Sobolev inequality.

\begin{theorem}\label{thm:LangevinDynamicsMain}
    Assume $\nu$ is $\alpha$-SLC for some $\alpha > 0$. Let $X_t \sim \rho_t$ evolve following the Langevin dynamics~\eqref{eq:LangevinDynamics} to $\nu$ from $X_0 \sim \rho_0$, and let $\rho_{0,t}$ be the joint law of $(X_0, X_t)$. If for some $s > 0$ we know that $\rho_s$ satisfies a $\Phi$-Sobolev inequality with constant $\alpha_{\PSI}(\rho_s)$, then for all $t \geq s$:
    \begin{equation}\label{eq:LDMainConvergence}
        \MI_{\Phi}(\rho_{0,t}) \leq e^{-2\alpha (t-s)} \max\left\{1, \frac{\alpha}{\alpha_{\PSI}(\rho_s)}\right\} \MI_{\Phi}(\rho_{0,s}).
    \end{equation}
\end{theorem}

We provide two proofs of Theorem~\ref{thm:LangevinDynamicsMain}. The first proof via our primary strategy of SDPIs follows by taking the appropriate limits of the discrete-time ULA analysis. We present the SDPI-based proof of Theorem~\ref{thm:LangevinDynamicsMain} in Appendix~\ref{app:SDPIPfOfLangevinDynamicsMain}. The second proof of Theorem~\ref{thm:LangevinDynamicsMain} following a direct time derivative approach is described in Section~\ref{sec:MainSectionLD_DirectProofSketch}, and the full proof is presented in Appendix~\ref{app:DirectPfOfLangevinDynamicsMain}.

This rate of convergence of $\Phi$-mutual information matches the rate of convergence of $\Phi$-divergence along the dynamics (Lemma~\ref{lem:PhiConvergenceAlongLangevin}).
The condition that $\rho_s$ satisfy a $\Phi$-Sobolev inequality (Definition~\ref{def:PhiSobolevInequality}) can be ensured by initializing the process from $\rho_0$ which satisfies a $\Phi$-Sobolev inequality, such as from a strongly log-concave distribution, for example a Gaussian distribution. 
The fact that strongly log-concave distributions satisfy a $\Phi$-Sobolev inequality is mentioned in Lemma~\ref{lem:PhiSISLC}, and the fact that if $\rho_0$ satisfies a $\Phi$-Sobolev inequality then $\rho_s$ (for $s > 0$) does too is a consequence of Lemma~\ref{lem:PSIEvolutionAlongLangevinDynamics}.

Observe how the right-hand side of~\eqref{eq:LDMainConvergence} has dependence on time $s$. For the special case of $\Phi(x) = x \log x$, we obtain bounds on the convergence of mutual information along the Langevin dynamics which do not possess this dependence and which do not require $\rho_s$ to satisfy a $\Phi$-Sobolev inequality.
This result follows by exploiting the regularity properties of the dynamics and we mention it in Theorem~\ref{thm:MI_LD_regularize} in Appendix~\ref{app:RegularityLD}.
We mention the tightness of Theorems~\ref{thm:LangevinDynamicsMain} and~\ref{thm:MI_LD_regularize} for the special case of $\Phi(x) = x \log x$ by explicitly computing the mutual information along the Langevin dynamics for $\nu$ being a Gaussian distribution; we describe this in Appendix~\ref{app:Tightness_LD}.

\subsection{Proof Sketch Based on Time Derivative}\label{sec:MainSectionLD_DirectProofSketch}

We present the direct proof of Theorem~\ref{thm:LangevinDynamicsMain} which is based on taking the time derivative of the $\Phi$-mutual information functional. 
The classical de Bruijn's identity~\citep{stam1959some} shows the time derivative of the relative entropy along the the Langevin dynamics is the relative Fisher information; here we show that the time derivative of the $\Phi$-mutual information functional along the Langevin dynamics is given by the $\Phi$-mutual Fisher information, which we define below.
We prove Lemma~\ref{lem:PhiMutualDeBruijn} in Appendix~\ref{app:PfOfPhiMutualDeBruijn}.

\begin{lemma}\label{lem:PhiMutualDeBruijn}
    Suppose $X_t \sim \rho_t$ evolves following the Langevin dynamics~\eqref{eq:LangevinDynamics} and let the joint law of $(X_0,X_t)$ be $\rho_{0,t}$\,.
    Then for $t > 0$
    \begin{align}\label{eq:deBruijnPhiMI}
        \frac{\D }{\D t} \MI_\Phi(\rho_{0,t}) = - \FI_\Phi^\M(\rho_{0,t}),
    \end{align}
    where $\FI^\M_\Phi$ is the $\Phi$-mutual Fisher information defined for a joint distribution $\rho^{XY}$ as
    \begin{align}\label{eq:MutualPhiFisherInfo}
        \FI^\M_\Phi (\rho^{XY}) &\coloneqq \E_{x \sim \rho^X} \left[\FI_\Phi \left(\rho^{Y \mid X = x} \dvert \rho^Y \right)  \right].
    \end{align}
\end{lemma}

Next, we show that the $\Phi$-mutual Fisher information can be lower bounded in terms of the $\Phi$-mutual information, under a $\Phi$-Sobolev inequality assumption on one of the marginal distributions. 
We prove Lemma~\ref{lem:mutual_PhiSI} in Appendix~\ref{app:PfOfmutual_PhiSI}.

    \begin{lemma}\label{lem:mutual_PhiSI}
    Let the joint law of $(X,Y)$ be $\rho^{XY}$ and suppose $\rho^Y$ satisfies $\Phi$-Sobolev inequality with optimal constant $\alpha_{\PSI}(\rho^Y)$\,. Then
    \[
    2\alpha_{\PSI}(\rho^Y)\, \MI_{\Phi}(\rho^{XY}) \leq \FI_{\Phi}^\M (\rho^{XY})\,.
    \]
    \end{lemma}

Theorem~\ref{thm:LangevinDynamicsMain} follows by combining Lemma~\ref{lem:mutual_PhiSI} with~\eqref{eq:deBruijnPhiMI} and integrating; we provide this proof in Appendix~\ref{app:DirectPfOfLangevinDynamicsMain}. However, note that combining Lemma~\ref{lem:mutual_PhiSI} with~\eqref{eq:deBruijnPhiMI} requires $\rho_t$ along the Langevin dynamics to satisfy a $\Phi$-Sobolev inequality. This challenge is consistent with the SDPI approach (discussed in Section~\ref{sec:Intro_SDPI}), and this is where we crucially require the strong log-concavity of $\nu$.

\section{Convergence of $\Phi$-Mutual Information along ULA}

We show the following result on the exponential convergence of $\Phi$-mutual information along the ULA~\eqref{eq:ULA} for smooth and strongly log-concave target distribution, provided an iterate along ULA satisfies a $\Phi$-Sobolev inequality. 
We present a proof sketch of Theorem~\ref{thm:ULA_Main} in Section~\ref{sec:MainSectionULA_SDPIProofSketch}, and we provide the full proof in Appendix~\ref{app:PfOfThmULAMain}.

\begin{theorem}\label{thm:ULA_Main}
    Suppose $\nu$ is $\alpha$-strongly log-concave and $L$-smooth for some $0 < \alpha \leq L < \infty$. 
    Let $X_k \sim \rho_k$ evolve following ULA~\eqref{eq:ULA} with step-size $\eta \leq 1/L$ from $X_0 \sim \rho_0$, and let the joint law of $(X_i, X_j)$ be $\rho_{i,j}$\,. If $\rho_\ell$ satisfies a $\Phi$-Sobolev inequality with optimal constant $\alpha_{\PSI}(\rho_\ell)$ for some $\ell \geq 1$\,.
    Then for all $k \ge \ell$, we have
    \begin{equation}\label{eq:ULA_MainTheorem}
        \MI_{\Phi}(\rho_{0,k}) \leq (1-\alpha\eta)^{2(k-\ell)}\max \left\{1, \frac{\alpha}{\alpha_{\PSI}(\rho_\ell)} \right\} \MI_{\Phi}(\rho_{0,\ell}).
    \end{equation}
\end{theorem}

Theorem~\ref{thm:ULA_Main} implies the following corollary regarding the iteration complexity of ULA to output approximately independent samples.
We provide the proof of Corollary~\ref{cor:ULA} in Appendix~\ref{app:ULAIterationComplexity}.

\begin{corollary}\label{cor:ULA}
    Under the same assumptions as Theorem~\ref{thm:ULA_Main} and given any error threshold $\epsilon >0$, ULA~\eqref{eq:ULA} outputs a sample $X_k$ such that $\MI_{\Phi}(X_0; X_k) \leq \epsilon$ as long as 
    \[
k \geq \ell + \frac{1}{2\alpha \eta}\log\left(\epsilon^{-1}\max\left\{1, \frac{\alpha}{\alpha_{\PSI}(\rho_\ell)}\right\} \MI_{\Phi}(\rho_{0,\ell})\right).
\]
\end{corollary}

The condition in Theorem~\ref{thm:ULA_Main} that $\rho_\ell$ satisfy a $\Phi$-Sobolev inequality can be met by choosing $\rho_0$ which satisfies a $\Phi$-Sobolev inequality, for example a strongly log-concave distribution. In this case, $\rho_\ell$ will satisfy a $\Phi$-Sobolev inequality for all $\ell \geq 1$ (consequence of Lemma~\ref{lem:ULAContractionCoefficientAndPSIEvolution}(b)). The fact that strongly log-concave distribution satisfy a $\Phi$-Sobolev inequality is mentioned in Lemma~\ref{lem:PhiSISLC}.
The exponential rate of convergence matches the rate of convergence of $\Phi$-divergence along the ULA~\citep[Theorem~1]{mitra2024fast}.
Also note that the additional smoothness assumption on $\nu$ in Theorem~\ref{thm:ULA_Main} is standard in the analysis of ULA.

For the special case of $\Phi(x) = x \log x$, Theorem~\ref{thm:MI_ULA_regularize} in Appendix~\ref{app:RegularityULA} studies the convergence of mutual information along ULA as a consequence of the regularity properties of the algorithm. Theorem~\ref{thm:MI_ULA_regularize} does not require the distributions along the trajectory to satisfy a $\Phi$-Sobolev inequality and the resulting bound does not depend on $\alpha_{\PSI}(\rho_\ell)$.
The tightness of Theorems~\ref{thm:ULA_Main} and~\ref{thm:MI_ULA_regularize} for the special case of $\Phi(x) = x \log x$ follows by explicitly computing $\MI(\rho_{0,k})$ along the ULA when $\nu$ is a Gaussian distribution; we describe this in Appendix~\ref{app:OU_ULA}.

\subsection{Proof Sketch of Theorem~\ref{thm:ULA_Main}}\label{sec:MainSectionULA_SDPIProofSketch}

In order to analyze the ULA~\eqref{eq:ULA} via SDPIs, it will be helpful to view the update in the space of distributions.
Letting $X_k \sim \rho_k$\,, the update of $\rho_k$ as $X_k$ evolves following~\eqref{eq:ULA} is
\begin{equation}\label{eq:ULA_TwoStepUpdate}
    \rho_{k+1} = \rho_k \bP_{\ULA} = F_\# \rho_k * \N(0, 2\eta I)\,,
\end{equation}
where $F(x) = x - \eta \nabla f (x)$ and $\bP_{\ULA}$ denotes the Markov kernel of ULA.
This two-step interpretation of a pushforward followed by a Gaussian convolution will be crucial in the SDPI analysis. We will denote $\bP_{\ULA}$ by $\bP$ when the Markov chain is clear.

As described briefly in Section~\ref{sec:Intro_SDPI} and in detail in Appendix~\ref{app:SDPI}, studying the $\Phi$-mutual information contraction along a Markov chain involves bounding each of the contraction coefficients along the trajectory. 
The contraction coefficient in terms of $\Phi$-divergence, which we bound for the ULA in Lemma~\ref{lem:ULAContractionCoefficientAndPSIEvolution}(a), are defined in~\eqref{eq:ContractionCoefficientMainText} and in Definition~\ref{def:ContractionCoefficient_PhiDivergence}.
Lemma~\ref{lem:ULAContractionCoefficientAndPSIEvolution}(a) bounds the contraction coefficient along ULA so long as the distribution along the chain satisfies a $\Phi$-Sobolev inequality. Ensuring that successive iterates along ULA satisfy a $\Phi$-Sobolev inequality is given in Lemma~\ref{lem:ULAContractionCoefficientAndPSIEvolution}(b).
We prove Lemma~\ref{lem:ULAContractionCoefficientAndPSIEvolution} in Appendix~\ref{app:PfOfULAContractionCoefficientAndPSIEvolution}.

\begin{lemma}\label{lem:ULAContractionCoefficientAndPSIEvolution}
Consider the ULA kernel $\bP$ defined in~\eqref{eq:ULA_TwoStepUpdate} with $\|F\|_\Lip = \gamma < 1$ and $\eta > 0$.
Let $\rho$ be a distribution that satisfies a $\Phi$-Sobolev inequality with optimal constant $\alpha_{\PSI}(\rho)$\,. Then we have:
\begin{itemize}[noitemsep,topsep=2pt]
    \item[(a)] If $F$ is also bijective, then the contraction coefficient satisfies
    \[
\varepsilon_{\sfD_{\Phi}}(\bP, \rho) \leq \frac{\gamma^2}{\gamma^2 + 2\eta \,\alpha_{\PSI}(\rho)}\,.
\]
\item[(b)] The $\Phi$-Sobolev inequality constants $\alpha_{\PSI}(\rho)$ and $\alpha_{\PSI}(\rho \bP)$ satisfy
\begin{equation}\label{eq:ULAPSIEvolution}
        \alpha_{\PSI}(\rho \bP) \geq \frac{\alpha_{\PSI}(\rho)}{\gamma^2 +  2\eta\,\alpha_{\PSI}(\rho)}\,.
    \end{equation}
\end{itemize}

\end{lemma}

Theorem~\ref{thm:ULA_Main} then follows by combining both parts of Lemma~\ref{lem:ULAContractionCoefficientAndPSIEvolution} across multiple steps of ULA.

\section{Convergence of $\Phi$-Mutual Information along Proximal Sampler}\label{sec:ProximalMain}

We show the following result on the exponential convergence of $\Phi$-mutual information along the Proximal Sampler~\eqref{eqs:ProximalSampler} for strongly log-concave target distributions, provided an iterate along the Proximal Sampler satisfies a $\Phi$-Sobolev inequality. 
We provide a proof sketch of Theorem~\ref{thm:ProximalMain} in Section~\ref{sec:MainSectionProximal_SDPIProofSketch}, and present the complete proof in Appendix~\ref{app:PfOfProximalMain}.

\begin{theorem}\label{thm:ProximalMain}
    Suppose $\nu^X \propto \exp{(-f)}$ is $\alpha$-strongly log-concave for some $\alpha >0$. 
    Let $X_k \sim \rho_k^X$ evolve following the Proximal Sampler~\eqref{eqs:ProximalSampler} with step-size $\eta >0$ from $X_0 \sim \rho_0^X$, and let the joint law of $(X_i, X_j)$ be $\rho^X_{i,j}$\,.
    If $\rho_\ell^X$ satisfies a $\Phi$-Sobolev inequality with optimal constant $\alpha_{\PSI}(\rho_\ell^X)$ for some $\ell \geq 1$.
    Then for all $k \geq \ell$, we have
    \begin{equation}\label{ineq:PS}
        \MI_{\Phi}(\rho^X_{0,k}) \leq \frac{\MI_{\Phi}(\rho^X_{0,\ell})}{(1 + \eta \min \{ \alpha, \alpha_{\PSI}(\rho^X_\ell) \})^{2(k-\ell)}}\,.
    \end{equation}
\end{theorem}

Theorem~\ref{thm:ProximalMain} implies the following iteration complexity for the Proximal Sampler, stated in Corollary~\ref{cor:PS}. We provide the proof of Corollary~\ref{cor:PS} in Appendix~\ref{app:PSIterationComplexity}.

\begin{corollary}\label{cor:PS}
    Under the same assumptions as Theorem~\ref{thm:ProximalMain} and given any error threshold $\epsilon >0$, the Proximal Sampler~\eqref{eqs:ProximalSampler} outputs a sample $X_k$ such that $\MI_{\Phi}(X_0; X_k) \leq \epsilon$ as long as 
    \[
k \geq \ell + \left[\frac12+ \frac{1}{2\eta \min \{ \alpha, \alpha_{\PSI}(\rho^X_\ell) \}} \right]\log\left( \epsilon^{-1}\MI_{\Phi}(\rho_{0,\ell}^X) \right).
\]
\end{corollary}

The condition in Theorem~\ref{thm:ProximalMain} that $\rho_\ell^X$ satisfy a $\Phi$-Sobolev inequality can be met by choosing a $\rho_0^X$ which satisfies a $\Phi$-Sobolev inequality, such as a strongly log-concave distribution. The fact that when $\rho_0^X$ satisfies a $\Phi$-Sobolev inequality, $\rho_\ell^X$ does as well (for all $\ell \geq 1$) is a consequence of Lemma~\ref{lem:proximal_ContractionCoefficientBoundAndPSIEvolution}(b). Additionally, Lemma~\ref{lem:PhiSISLC} states that strongly log-concave distributions satisfy $\Phi$-Sobolev inequalities.
Note that the rate of convergence of $\Phi$-mutual information matches that of $\Phi$-divergence; see~\citep[Theorem~2]{mitra2024fast}.
Further note that there is no smoothness assumption required in Theorem~\ref{thm:ProximalMain}, since the result is for the ideal Proximal Sampler which assumes access to a perfect RGO. 
Implementing the RGO typically requires smoothness assumptions; see Appendix~\ref{app:ProximalBackground} for a review. 
In Corollary~\ref{cor:ProximalWithRejectionSamplingRGO} in Appendix~\ref{app:ProximalBackground}, we show the expected oracle complexity when we combine the convergence guarantee from Theorem~\ref{thm:ProximalMain} with the standard implementation of RGO via rejection sampling.

For the special case of $\Phi(x) = x \log x$ and $\ell =1$, we are able to bound the initial mutual information $\MI(\rho_{0,1}^X)$ in terms of the variance of $\rho_0$\,. We state this in Appendix~\ref{app:RegularityPS} and present the corresponding iteration complexity in Corollary~\ref{cor:1StepAndInitialBoundProximal}.
The tightness of Theorem~\ref{thm:ProximalMain} for the special case of $\Phi(x) = x \log x$ can be seen by doing explicit calculations for the case when $\nu$ is Gaussian, and we present this in Appendix~\ref{app:OU_PS}.

\subsection{Proof Sketch of Theorem~\ref{thm:ProximalMain}}\label{sec:MainSectionProximal_SDPIProofSketch}

Recall the Proximal Sampler~\eqref{eqs:ProximalSampler}. We denote the Proximal Sampler by $\bP_\prox$, i.e. $\rho^X_k \, \bP_{\prox} = \rho^X_{k+1}$ where $\rho_k^X \coloneqq \law (X_k)$.
The proof of Theorem~\ref{thm:ProximalMain} follows the template mentioned in Section~\ref{sec:Intro_SDPI} (and mentioned in detail in Appendix~\ref{app:SDPI}) where the key step is in bounding the contraction coefficient along the trajectory of the Proximal Sampler. The bound on the contraction coefficient for the Proximal Sampler is presented in Lemma~\ref{lem:proximal_ContractionCoefficientBoundAndPSIEvolution}(a), which holds under the distribution satisfying a $\Phi$-Sobolev inequality. 
To apply Lemma~\ref{lem:proximal_ContractionCoefficientBoundAndPSIEvolution}(a) across multiple steps of the Proximal Sampler, we need to ensure that the $\Phi$-Sobolev inequality assumption is maintained along the trajectory of the chain. This is guaranteed by Lemma~\ref{lem:proximal_ContractionCoefficientBoundAndPSIEvolution}(b), which crucially requires the strong log-concavity of $\nu^X$.
We prove Lemma~\ref{lem:proximal_ContractionCoefficientBoundAndPSIEvolution} in Appendix~\ref{app:PfOfproximal_ContractionCoefficientBoundAndPSIEvolution}.

\begin{lemma}\label{lem:proximal_ContractionCoefficientBoundAndPSIEvolution}
    Suppose $\nu^X \propto \exp{(-f)}$ is $\alpha$-strongly log-concave for some $\alpha >0$, and consider the Proximal Sampler $\bP_{\prox}$~\eqref{eqs:ProximalSampler} with step-size $\eta > 0$ for sampling from $\nu^X$. Let $\rho$ be a distribution that satisfies a $\Phi$-Sobolev inequality with optimal constant $\alpha_{\PSI}(\rho)$\,. 
    Then we have the following:
    \begin{itemize}[noitemsep,topsep=2pt]
        \item[(a)] The contraction coefficient satisfies
    \[
    \varepsilon_{\sfD_{\Phi}}(\bP_{\prox}\,, \rho) \leq \frac{1}{1 + 2\eta \alpha_{\PSI} (\rho) + \eta^2 \alpha \alpha_{\PSI} (\rho) }\,.
    \]
    \item[(b)] The $\Phi$-Sobolev inequality constants $\alpha_{\PSI}(\rho)$ and $\alpha_{\PSI}(\rho \bP_\prox)$ satisfy
    \[
    \frac{1}{\alpha_{\PSI}(\rho \bP_{\prox})} \le \frac{1+\alpha_{\PSI}(\rho) \eta}{\alpha_{\PSI}(\rho) (1 + \alpha \eta)^2} + \frac{\eta}{1 + \alpha \eta}.
    \]
    \end{itemize}

\end{lemma}

Suitably combining both parts of Lemma~\ref{lem:proximal_ContractionCoefficientBoundAndPSIEvolution} across multiple steps of the Proximal Sampler proves Theorem~\ref{thm:ProximalMain}, which we present in Appendix~\ref{app:PfOfProximalMain}.

%%%%%%%%
\subsubsection{Proof Sketch of Lemma~\ref{lem:proximal_ContractionCoefficientBoundAndPSIEvolution}}

Recall that the Proximal Sampler~\eqref{eqs:ProximalSampler} is composed of a forward step and a backward step. 
We denote by $\bP_{\prox}^+$ and $\bP_{\prox}^-$ the Markov kernel corresponding to the forward and the backward step, respectively, so we can write the Proximal Sampler as the composition $\bP_{\prox} = \bP_{\prox}^+ \bP_{\prox}^-$.

We use the SDE interpretations of the forward step $\bP^+_{\prox}$ and the backward step $\bP^-_{\prox}$ to control the contraction coefficient and the evolution of the $\Phi$-Sobolev constant along each step.
We combine these estimates to to prove the estimates for the Proximal Sampler claimed in Lemma~\ref{lem:proximal_ContractionCoefficientBoundAndPSIEvolution}.

\paragraph{Forward Step:}
Suppose we start from $X_0 \sim \rho_0^X$. 
Along the forward step of the Proximal Sampler~\eqref{eqs:ProximalSampler}, $Y_0 \mid X_0 \sim \N(X_0, \eta I)$, so in particular, $\rho_0^Y = \rho_0^X * \N(0, \eta I)$.
Therefore, the action of the forward step $\bP^+_{\prox}$ is via a Gaussian convolution: 
$\rho \bP^+_{\prox} = \rho *\N(0, \eta I)$, which can be interpreted as the solution to the heat flow (generated by the Brownian motion SDE $\D X_t = \D W_t$) at time $\eta > 0$.

\paragraph{Backward Step:}
For the backward step $\bP^-_{\prox}$, it will be helpful to think of the corresponding SDE as the time reversal of the forward step SDE, i.e. the time reversal of the heat flow, which is known as the \emph{backward heat flow}; see~\citep{chen2022improved} and~\citep[Chapter~8.3]{chewi2023log}.

Fixing a step-size $\eta >0$, we have along the forward step ($\D X_t = \D W_t$) that if $X_0 \sim \nu^X$, then $X_{\eta} \sim \nu^Y$. The backward heat flow SDE is defined by
\begin{equation}\label{eq:BackwardHeatFlow}
    \D Y_t = \nabla \log{(\nu^X * \N_{\eta -t})}(Y_t) \D t + \D W_t\,.
\end{equation}
By construction, if we start the SDE~\eqref{eq:BackwardHeatFlow} from $Y_0 \sim \mu_0 = \nu^Y$, then for any $t \in [0,\eta]$, the distribution of $Y_t \sim \mu_t$ along~\eqref{eq:BackwardHeatFlow} is given by $\mu_t = \nu^X * \N_{\eta-t}$, and in particular, at time $t = \eta$, $Y_\eta \sim \mu_\eta = \nu^X$.
For the Proximal Sampler, we start the backward SDE~\eqref{eq:BackwardHeatFlow} from $Y_0 \sim \rho_0^Y$, to obtain $X_1 := Y_\eta \sim \rho_1^X$.

We note that computing the contraction coefficient for the backward step (Lemma~\ref{lem:ContractionCoefficientBound_BackHeatFlow}) and analyzing the evolution of the $\Phi$-Sobolev constant along the backward step (Lemma~\ref{lem:PhiSI_Evolution_BackwardHeatFlow}) is challenging due to the time-varying drift in the backward step SDE~\eqref{eq:BackwardHeatFlow}, and it is the key difficulty when studying the Proximal Sampler.
We provide the details of the computations above in Appendix~\ref{app:ProximalSampler}.

\section{Discussion}
We study the convergence of $\Phi$-mutual information along the Langevin dynamics~\eqref{eq:LangevinDynamics}, ULA~\eqref{eq:ULA}, and Proximal Sampler~\eqref{eqs:ProximalSampler} assuming the strong log-concavity of $\nu$. 
Our primary proof strategy is based on SDPIs and we explain how studying the contraction of information along a Markov chain requires controlling the SDPI contraction coefficients along the trajectory. 
We require the strong log-concavity of $\nu$ as this implies that the distributions along the trajectory satisfy $\Phi$-Sobolev inequalities, under which the contraction coefficients are strictly less than $1$.

We now mention some interesting future directions. The SDPI approach, the direct time derivative approach for continuous-time, as well as the regularity approach for the Langevin dynamics and ULA all require the strong log-concavity assumption on $\nu$. This is distinct from mixing time results where we have a rich understanding beyond strong log-concavity and under only isoperimetric assumptions on $\nu$ such as a log-Sobolev inequality, a Poincaré inequality, or in general a $\Phi$-Sobolev inequality.
Therefore, it would be interesting to explore if the convergence of $\Phi$-mutual information, perhaps even for specific $\Phi$, can be proven under weaker assumptions on $\nu$. 
Additionally, we focus here on the (overdamped) Langevin dynamics and its discretizations. Extending these results to other Markov chains such as the underdamped Langevin dynamics and Hamiltonian Monte Carlo would be interesting to explore.
Finally, just as the Langevin dynamics is the Wasserstein gradient flow for the relative entropy functional, one can also study the gradient flow of $\Phi$-mutual information directly and it would be interesting to explore algorithmic implications of such a flow.

\acks{The authors thank Vishwak Srinivasan for helpful comments, and to Sekhar Tatikonda and Sinho Chewi for feedback on an earlier version of the paper.
A.W. and S.M. are supported by NSF awards CCF-2403391 and CCF-2443097.
}

\newpage
\appendix

%%%%%%%%%%%%%%%%%%%%%%
\section{Additional Related Works}\label{sec:Intro_RelatedWork}

The Langevin Markov chains we consider in this paper are widely used for continuous-space sampling; see~\citep{chewi2023log} for an overview.
The mixing time of the Langevin dynamics in various divergences has been studied in many works, including~\citep{chafai2004entropies,  villani2009optimal, bolley2010phi, achleitner2015large, dolbeault2018varphi, cao2019exponential}.
We review the convergence guarantee of the $\Phi$-divergence along the Langevin dynamics under a $\Phi$-Sobolev inequality in Appendix~\ref{app:LangevinDynamicsBackground}.
The mixing time of the ULA has been studied in~\citep{roberts1996exponential, roberts1998optimal, dalalyan2017theoretical, dalalyan2017further, durmus2017nonasymptotic, cheng2018convergence, durmus2019analysis, VW19, chewietal2022lmcpoincare, AltschulerTalwar23, mitra2024fast}, and the mixing time of the proximal sampler has been studied in~\citep{lee2021structured, chen2022improved, liang2022proximal, liang2023a, yuan2023class,fan2023improved, kook2024inandout, mitra2024fast}.

The family of $\Phi$-divergences~\citep{polyanskiy2025information, sason2016f} is broad and includes many applications in addition to mixing time of Markov chains~\citep{chafai2004entropies, raginsky2016strong, mitra2024fast}. 
They have been used for hypothesis and distribution testing~\citep{pensia2024sample, gyorfi2002asymptotic, gretton2008nonparametric}, reinforcement learning~\citep{ho2022robust, panaganti2024model}, neuroscience~\citep{nemenman2004entropy, belitski2008local}, and differential privacy~\citep{asoodeh2021local, zamanlooy2024mathrm}, among other applications.
The induced $\Phi$-mutual informations have also been used in varied applications such as density estimation~\citep{goldfeld2019estimating, zhang2021convergence}, contrastive learning~\citep{lu24micl}, generalization~\citep{esposito2020robust}, and investment and portfolio theory~\citep{erkip1998efficiency, anantharam2013maximal, raginsky2016strong}.

SDPIs have been popular for proving mixing time of Markov chains~\citep{raginsky2016strong} and other general networks and processes~\citep{polyanskiy2017strong}. They have been frequent in the context of Langevin-type Markov chains as well, although much of the connection between SDPIs and mixing times there has been implicit.~\citet{VW19} use SDPIs to study the mixing time of ULA in Rényi divergence,~\citet{chen2022improved} use them for the proximal sampler,~\cite{yuan2023class} use them for the proximal sampler on graphs, and recently,~\citet{mitra2024fast} use them explicitly to study the mixing time of the ULA and proximal sampler in $\Phi$-divergence.

Identically distributed and dependent random variables are also studied more generally under \emph{mixing}~\citep{bradley1986basic, bradley_mixing}, where different definitions of dependency between sigma-algebras such as $\alpha$, $\beta$, and $\rho$-mixing~\citep[(1.1) to (1.5)]{bradley_mixing} are compared~\citep[(1.11) to (1.18)]{bradley_mixing} and studied.
In this paper, we focus specifically on data coming from Langevin-type Markov chains, and on $\Phi$-mutual information as our functional to measure dependency. 
Additionally, we do not assume the data to be stationary (i.e. identically distributed), which is helpful for modern Markov Chain Monte Carlo (MCMC) with short and unmixed chains~\citep{margossian2023many, margossian2024nested}.

The behaviors of mutual information along the evolution of a stochastic process have attracted interests in the information theory literature.
Along the Gaussian channel or the heat flow, the rate of change of mutual information is related to the minimum mean-square error (MMSE), known as the I-MMSE relationship~\citep{guo2005mutual}; this relationship can be generalized to the Poisson channel~\citep{ atar2010mutual} and to the family of \textit{Fokker-Planck} channels induced by a stochastic process driven by a Brownian motion~\citep{WJL17}, or the jump-diffusion process~\citep{fan2025differential}.
The \textit{convexity} of how the mutual information is decreasing has been studied along the heat flow~\citep{WJ18a}, the Ornstein-Uhlenbeck process~\citep{WJ18b}, and the general Fokker-Planck channel~\citep{zou2025convexity}.
In this work, we quantify the rate of decrease of the mutual information along the channel induced by the Langevin dynamics and its discrete-time implementations, the ULA and the proximal sampler.

%%%%%%%%%%%%%%%%%%%%%%
\section{Examples of $\Phi$-divergences}\label{app:ExamplesOfPhiDiv}

\begin{center}
\begin{table}[h!]
\begin{tabular}{cccc}
 \bm{$\Phi (x)$} & \bm{$\sfD_{\Phi}(\mu \dvert \nu)$} & \bm{$\sfD_{\Phi}(\mu \dvert \nu)$} \textbf{name}  & \textbf{\bm{$\Phi$}-Sobolev inequality} \\ 
 \\ \hline\\
 $x \log x$ & $\bigintsss \D \mu \log \frac{\D \mu}{\D \nu}$ &  KL divergence & log-Sobolev inequality \\
  \\
 $(x-1)^2$ & $\bigintsss \frac{(\D \mu - \D \nu)^2}{\D \nu}$ & chi-squared divergence & Poincaré inequality\\
 \\
 $\frac{1}{2}(\sqrt{x}-1)^2$ & $\frac{1}{2} \bigintsss (\sqrt{\D \mu} - \sqrt{\D \nu})^2$ & squared Hellinger distance & -- \\
 \\
 $\frac{1}{2}|x-1|$ & $\frac{1}{2}\bigintsss |\D \mu - \D \nu | $ & TV distance & -- \\
 \\
 $-\log x$ & $\bigintsss \D \nu \log \frac{\D \nu}{\D \mu}$ &  reverse KL divergence & -- \\
 \\
 $\frac{1}{x}-x$ & $2+\bigintsss \frac{(\D \mu - \D \nu)^2}{\D \mu}$ & reverse chi-squared divergence & --

\end{tabular}
\caption{Common $\Phi$ functions along with corresponding $\Phi$-divergences (Definition~\ref{def:PhiDivergence}) and $\Phi$-Sobolev inequalities (Definition~\ref{def:PhiSobolevInequality}).\label{table:PhiExamples}}
\end{table}
\end{center}

Table~\ref{table:PhiExamples} includes examples of common $\Phi$-divergences. Further examples include Jensen-Shannon divergence and Le Cam divergence and can be found in~\citet[Chapter~7]{polyanskiy2025information} and~\citet{sason2016f}.

\section{Properties of $\Phi$-Sobolev Inequalities}\label{app:PhiSobolevProperties}

Recall the $\Phi$-Sobolev inequality defined in Definition~\ref{def:PhiSobolevInequality}. 
Note that the inequality~\eqref{eq:PhiSI_DistributionBased} is equivalent to saying that for all smooth functions $g : \R^d \to \R_{\ge 0}$ with $\E_\nu[g] = 1$,
    \begin{equation*}\label{eq:PhiSI_FunctionBased}
        2 \alpha \, \Ent_{\Phi}^{\nu} (g) \leq \Dir_\Phi^\nu(g)\,,
    \end{equation*}
    where
    \[
    \Ent_{\Phi}^{\nu} (g) \coloneqq \E_{\nu} [\Phi (g)] - \Phi (\E_{\nu}[g]) \hspace{0.5cm}\text{and}\hspace{0.5cm} \Dir_\Phi^\nu(g) := \E_\nu\left[\|\nabla g\|^2 \Phi''(g)\right]\,.
    \]
    This can be seen by taking $g$ to be the density function of $\mu$ with respect to $\nu$.

Further recall the optimal $\Phi$-Sobolev constant defined in~\eqref{eq:PhiSI_Constant}.
Understanding how this constant evolves along the various operations that constitute the Markov chains we study such as pushforward and convolution is crucial in our approach.
The following properties from~\cite{chafai2004entropies} describe how the $\Phi$-Sobolev inequality constant evolves along these operations.

\begin{lemma}\cite[Remark~7]{chafai2004entropies}\label{lem:PhiSIchangeAlongLipschitzPushforward}
    Assume $\nu$ satisfies a $\Phi$-Sobolev inequality with optimal constant $\alpha_{\PSI}(\nu)$.
    Let $T : \R^d \to \R^d$ be a $\gamma$-Lipschitz map.
    Then the pushforward $ \tilde \nu = T_{\#}\nu$ satisfies a $\Phi$-Sobolev inequality with optimal constant 
    \begin{equation}\label{ineq:pushforward}
        \alpha_{\PSI}(\tilde \nu) \ge \frac{\alpha_{\PSI}(\nu)}{\gamma^2}\,.
    \end{equation}
\end{lemma}

The next lemma describes the change of the $\PSI$ constant after convolution.

\begin{lemma}\cite[Corollary~3.1]{chafai2004entropies}\label{lem:PhiSIchangeAlongConvolution}
    Assume $\mu$ and $\nu$ satisfy a $\Phi$-Sobolev inequality with optimal constants $\alpha_{\PSI}(\mu)$ and $\alpha_{\PSI}(\nu)$ respectively. 
    Then the convolution $\mu \ast \nu$ satisfies the $\Phi$-Sobolev inequality with constant
    \begin{equation}\label{ineq:convolution}
        \frac{1}{\alpha_{\PSI}(\mu * \nu)} \le \frac{1}{\alpha_{\PSI}(\mu)} + \frac{1}{\alpha_{\PSI}(\nu)}\,.
    \end{equation}
\end{lemma}

The following lemma tells us that when $\nu$ is $\alpha$-strongly log-concave, it also satisfies a $\Phi$-Sobolev inequality with the same constant.

\begin{lemma}\cite[Corollary~2.1]{chafai2004entropies}\label{lem:PhiSISLC}
    If $\nu$ is $\alpha$-strongly log-concave for some $\alpha > 0$, then $\nu$ satisfies a $\Phi$-Sobolev inequality with constant 
    $$\alpha_{\PSI}(\nu) \ge \alpha \,.$$
\end{lemma}

\section{Review of the Fast Mixing of Langevin Dynamics in $\Phi$-Divergence}\label{app:LangevinDynamicsBackground}

Recall the Langevin dynamics~\eqref{eq:LangevinDynamics} to sample from $\nu \propto \exp{(-f)}$ on $\R^d$
\[
\D X_t = -\nabla f(X_t) \D t + \sqrt{2} \D W_t\,,
\]
where $W_t$ is Brownian motion on $\R^d$. 
These dynamics can equivalently be viewed as the gradient flow for the KL divergence or relative entropy functional $\KL(\cdot \dvert \nu)$ in the space of distributions with the Wasserstein metric~\citep{JKO98, villani2021topics}. A log-Sobolev inequality corresponds to the gradient-domination condition for the relative entropy objective functional, under which there is rapid convergence of $\KL(\cdot \dvert \nu)$ along the dynamics. This affirms that the Langevin dynamics are well-suited for sampling from $\nu$.

The fast convergence of KL divergence under a log-Sobolev inequality assumption on $\nu$ can be extended to showing rapid convergence of $\Phi$-divergence under a $\Phi$-Sobolev inequality assumption~\citep{chafai2004entropies, achleitner2015large}. 
As the Poincaré inequality is the $\Phi$-Sobolev inequality for $\Phi(x) = (x-1)^2$, this also includes the convergence of chi-squared divergence under a Poincaré inequality as a special case.
The convergence of $\Phi$-divergence under a $\Phi$-Sobolev inequality follows easily from Lemma~\ref{lem:SimultaneousSDE} and we mention it below.

\begin{lemma}\label{lem:PhiConvergenceAlongLangevin}
 Suppose $X_t \sim \rho_t$ evolves along the Langevin dynamics~\eqref{eq:LangevinDynamics} to sample from $\nu \propto \exp{(-f)}$, and let $\nu$ satisfy a $\Phi$-Sobolev inequality with optimal constant $\alpha >0$. Then,
 \[
 \sfD_{\Phi}(\rho_t \dvert \nu) \leq e^{-2\alpha t}\sfD_{\Phi}(\rho_0 \dvert \nu).
 \]
\end{lemma}

\begin{proof}
    Applying Lemma~\ref{lem:SimultaneousSDE} (with $\mu_t = \rho_t$, $\nu_t = \nu$, $b_t(x) = -\nabla f(x)$ and $c=1$) along with the $\Phi$-Sobolev inequality of $\nu$ (Definition~\ref{def:PhiSobolevInequality}) yields,
    \begin{equation*}
        \frac{\D}{\D t}\sfD_\Phi (\mu_t \dvert \nu) 
        = - \FI_\Phi(\mu_t \dvert \nu) \leq -2\alpha \sfD_\Phi (\mu_t \dvert \nu).
    \end{equation*}
    Using Grönwall's lemma, we conclude the desired bound and thus complete the proof.   
\end{proof}

\section{Covariance and $\Phi$-Mutual Information}\label{app:CovPhiMI}

Here we show that the covariance between two random variables can be upper bounded by their $\Phi$-mutual information. This illustrates that $\Phi$-mutual information is a stronger notion of independence than covariance or correlation.

For random variables $X, Y \in \R^d$, let $\Cov(X,Y) = \E[(X - \E[X])(Y - \E[Y])^\top] \in \R^{d\times d}$ denote the covariance matrix of $X$ and $Y$.
Recall that the convex conjugate of a function $f:\R \to \R$ is $f^* (y) = \sup_{x \in \R} xy - f(x)$.

\begin{lemma}\label{lem:CovPhiMIBound}
   Let $(X, Y) \sim \rho^{XY}$ be a joint random variable. Let $\Phi_{\mathrm{ext}}$ be an extension of $\Phi$, defined as $\Phi_{\mathrm{ext}}(x) = \Phi(x)$ for $x \geq 0$, and $\Phi_{\mathrm{ext}}(x) = \infty$ for $x < 0$. Further denote $\Phi_{\mathrm{ext}}^*$ to be its convex conjugate. Assume $Y \sim \rho^Y$ satisfies the following:
   \begin{equation}\label{eq:ExtensionOfSubGaussianity}
   \exists\, 0<\xi<\infty\,, \text{such that } \forall \, \theta \in \R^d\,,\, \inf_{a \in \R} \E_{\rho^Y} \left[\Phi_{\mathrm{ext}}^*\left(\langle \theta, Y - \E[Y] \rangle - a \right)\right] + a \leq \frac{\|\theta\|^2 \xi^2}{2}\,.
   \end{equation}
   Then,
   \begin{equation}\label{ineq:COV}
       \|\Cov(X,Y)\|_\op \le \xi\, \sqrt{2 \, \|\Cov(X,X)\|_\op \, \MI_{\Phi}(\rho^{XY})}\,.
   \end{equation}
\end{lemma}

\begin{proof}
    Let $u, v \in \R^d$ with $\|u\| = \|v\| = 1$.
    We can bound
    \begin{align}
        u^\top \Cov(X,Y) v
        &= \E_{\rho^{XY}}\left[(u^\top X - \E_{\rho^X}[u^\top X])(v^\top Y - \E_{\rho^Y}[v^\top Y])\right] \notag \\
        &= \E_{\rho^X}\left[(u^\top X - \E_{\rho^X}[u^\top X])(v^\top \E_{\rho^{Y \mid X}}[Y] - \E_{\rho^Y}[v^\top Y])\right] \notag \\
        &\le \E_{\rho^X}\left[\left|u^\top X - \E_{\rho^X}[u^\top X]\right| \cdot \left|v^\top \E_{\rho^{Y \mid X}}[Y] - \E_{\rho^Y}[v^\top Y]\right|\right]. \label{Eq:CovPhiMI_Calc1}
    \end{align}
    The variational representation of $\Phi$-divergence~\cite[Theorem~7.26]{polyanskiy2025information} for $\pi_1 \ll \pi_2$ states that:
    \begin{equation}\label{eq:PhiDivergence_Variational}
    \sfD_{\Phi}(\pi_1 \dvert \pi_2) = \sup_{g \colon \R^d \to \R}\, \E_{Y \sim \pi_1}\left[g(Y)\right] - \psi^*_{\pi_2}(g)\,,
    \end{equation}
    where 
    \begin{equation}\label{eq:psi}
        \psi^*_{\pi_2}(g) \coloneqq \inf_{a \in \R} \E_{Y \sim \pi_2} \left[  \Phi_{\mathrm{ext}}^*(g(Y) - a) \right] + a\,.
    \end{equation}
    For any $\lambda > 0$\,, take 
    \[
    g(y) = \lambda\left(v^\top y - \E_{\rho^Y}[v^\top Y]\right), \quad \pi_1 = \rho^{Y \mid X}, \quad \pi_2 = \rho^Y
    \]
    for a fixed $X \in \R^d$\,, and substitute in~\eqref{eq:PhiDivergence_Variational} to obtain,
    \[
    \lambda\, v^\top \!\!\! \E_{\rho^{Y \mid X}}[Y] - \lambda\E_{\rho^Y}[v^\top Y]
        \,=\, \E_{\rho^{Y \mid X}}[g(Y)]
        \,\le\, \sfD_{\Phi}(\rho^{Y \mid X} \dvert \rho^Y) +  \psi_{\rho^Y}^*(g)\,.
    \]
    Therefore, we have that
    \begin{align*}
    v^\top \!\!\! \E_{\rho^{Y \mid X}}[Y] - \E_{\rho^Y}[v^\top Y]
        &\,\le\, \frac{1}{\lambda}\sfD_{\Phi}(\rho^{Y \mid X} \dvert \rho^Y) +  \frac{1}{\lambda}\psi_{\rho^Y}^*(g)\,,\\
        &\,=\, \frac{1}{\lambda}\sfD_{\Phi}(\rho^{Y \mid X} \dvert \rho^Y) + \frac{1}{\lambda} \left[ \inf_{a \in \R} \E_{\rho^Y} \left[ \Phi_{\mathrm{ext}}^*\left(\lambda(v^\top Y - \E_{\rho^Y}[v^\top Y]) -a \right)  \right] +a \right],
    \end{align*}
    where the identity is due to \eqref{eq:psi}.
    Using the assumption \eqref{eq:ExtensionOfSubGaussianity} on $\rho^Y$ along with the fact that $\|v\| = 1$, we have that,
    \[
    v^\top \!\!\! \E_{\rho^{Y \mid X}}[Y] - \E_{\rho^Y}[v^\top Y]
        \,\le\, \frac{1}{\lambda}\sfD_{\Phi}(\rho^{Y \mid X} \dvert \rho^Y) +  \frac{\lambda}{2}\xi^2\,.
    \]
    Choosing the optimal $\lambda = \frac{1}{\xi}\sqrt{2 \sfD_{\Phi}(\rho^{Y \mid X} \dvert \rho^Y)}$ gives the bound,
    \begin{align*}
        v^\top \E_{\rho^{Y \mid X}}[Y] - \E_{\rho^Y}[v^\top Y]
        &\le \xi \sqrt{2 \sfD_{\Phi}(\rho^{Y \mid X} \dvert \rho^Y)}\,.
    \end{align*}
    Since the right-hand side does not depend on $v$, the same argument applied to $-v$ yields the bound
    \begin{align*}
        \left|v^\top \E_{\rho^{Y \mid X}}[Y] - \E_{\rho^Y}[v^\top Y]\right|
        &\le \xi \sqrt{2 \sfD_{\Phi}(\rho^{Y \mid X} \dvert \rho^Y)}\,.
    \end{align*}
    We then plug this in to~\eqref{Eq:CovPhiMI_Calc1} and get
    \begin{align*}
        u^\top \Cov(X,Y) v
        &\le \E_{\rho^X}\left[\left|u^\top X - \E_{\rho^X}[u^\top X]\right| \cdot \xi \sqrt{2 \sfD_{\Phi}(\rho^{Y \mid X} \dvert \rho^Y)}\right] \\
        &\le \xi\sqrt{2} \, \left(\E_{\rho^X}\left[(u^\top X - \E_{\rho^X}[u^\top X])^2\right]\right)^{1/2} \, \left(\E_{\rho^X}[\sfD_{\Phi}(\rho^{Y \mid X} \dvert \rho^Y)]\right)^{1/2} \\
        &= \xi \sqrt{2} \, \sqrt{u^\top \Cov_{\rho^X}(X,X) u} \, \sqrt{\MI_{\Phi}(\rho^{XY})} \\
        &\leq \xi \sqrt{2} \, \sqrt{\|\Cov(X,X)\|_\op \, \MI_{\Phi}(\rho^{XY})}\,,
    \end{align*}
    where the second inequality follows by applying the Cauchy-Schwarz inequality, and the last inequality is because $\|\Cov(X,Y)\|_\op = \sup\{u^\top \Cov(X,Y) v \,\colon \|u\| = \|v\| = 1\}$\,.
    Finally, the conclusion \eqref{ineq:COV} of the lemma follows from this argument as well.
\end{proof}

The assumption on $\rho^Y$ in~\eqref{eq:ExtensionOfSubGaussianity} in Lemma~\ref{lem:CovPhiMIBound} is a generalization of a sub-Gaussianity assumption on $\rho^Y$, which~\eqref{eq:ExtensionOfSubGaussianity} simplifies to for $\Phi(x) = x \log x$. Indeed, for $\Phi(x) = x \log x$, $\Phi_{\mathrm{ext}}^*(y) = e^{y-1}$ and $\psi_{\pi_2}^*(g) = \log \E_{\pi_2} [e^{g(Y)}]$ (defined in~\eqref{eq:psi}), which means that $\rho^Y$ satisfying~\eqref{eq:ExtensionOfSubGaussianity} means that it is $\xi$ sub-Gaussian.
For $\Phi(x) = x \log x$,~\eqref{eq:PhiDivergence_Variational} corresponds to the Donsker-Varadhan variational formula for KL divergence.

Lemma~\ref{lem:CovPhiMIBound} strengthens~\citet[Lemma~1]{xu2017information} to (a) only requiring a marginal sub-Gaussian condition, and (b) extending the result to any $\Phi$-mutual information.

\section{Review of the Covariance Decay under Poincar\'e Inequality}\label{app:CovarianceDecay}

We review the classical fact that for any reversible Markov semigroup $\bP_t$ satisfying a Poincaré inequality, there is exponential convergence of covariance when measured against functions in $L^2_\nu$ (where $\nu$ is stationary for $\bP_t$); see also~\citep{madrasslade93, madras02markov}.

\begin{lemma}\label{lem:CovDecayUnderPoincare}
    Let $\bP_t$ be a reversible Markov semigroup with stationary distribution $\nu$ and associated stochastic process $X_t \sim \rho_t$. Suppose $\bP_t$ satisfies a Poincaré inequality with constant $\alpha >0$ and is at stationarity, i.e. $X_0 \sim  \nu$. Then for all functions $f \in L^2(\nu)$ and all $t \geq 0$
    \[
    \Cov (f(X_t), f(X_0)) \leq \exp (-\alpha t) \Var f(X_0)\,.
    \]
\end{lemma}

\begin{proof}
    Let $f$ be an arbitrary function in $L^2(\nu)$. Denote $\E_{\nu}[f] = \hat f$ and $g \coloneqq f - \hat f$. Then $g \in L^2(\nu)$ and $\E_{\nu}[g] = 0$. Also denote $\rho_{0,t} \coloneqq \law (X_0, X_t)$.
    Recall (see e.g.,~\citep[Theorem~1.2.21]{chewi2023log}) that a reversible semigroup $\bP_t$ satisfies a Poincaré inequality with constant $\alpha > 0$ if for all functions $g \in L^2(\nu)$ with $\E_{\nu}[g] = 0$,
\begin{equation}\label{eq:GeneralPoincareInequality}
\|\bP_t g\|^2_{L^2(\nu)} \leq \exp \left( -2\alpha t \right) \|g\|^2_{L^2(\nu)}.
\end{equation}
    Therefore, keeping in mind that the process is at stationarity, we have the following.
    \begin{align*}
        \Cov(f(X_t), f(X_0)) &= \E \left[ (f(X_t) - \hat f)(f(X_0) - \hat f)  \right]\\
        &= \E \left[ (f - \hat f)(X_t) \, (f-\hat f)(X_0)  \right]\\
        &= \E_{\rho_{0,t}} \left[ g(X_t) \, g(X_0) \right]\\
        &= \E_{\rho_0} \left[g(X_0) \E_{\rho_{t \mid 0}}[g(X_t)] \right]\\
        &= \langle g, \bP_t g \rangle_{\nu}\\
        &\leq \sqrt{\|\bP_t g\|^2_{L^2(\nu)} \|g\|^2_{L^2(\nu)}}\\
        &\leq \exp(-\alpha t)\|g\|^2_{L^2(\nu)}\\
        &= \exp(-\alpha t)\|f-\hat f\|^2_{L^2(\nu)} =  \exp(-\alpha t) \Var_{\nu} f = \exp (-\alpha t) \Var f(X_0).
    \end{align*}
    Here, the third equality is by definition of $g$, the fifth equality is by the definition of a Markov semigroup and the $L^2_{\nu}$ inner product, the first inequality is by Cauchy-Schwartz inequality, and the last inequality is due to Poincaré inequality~\eqref{eq:GeneralPoincareInequality}.

\end{proof}

\section{Further Background on Proximal Sampler and RGO Implementation via Rejection Sampling}\label{app:ProximalBackground}

The Proximal Sampler~\eqref{eqs:ProximalSampler} is a discrete-time continuous-space Gibbs sampling algorithm which has been studied for both unbounded space~\citep{lee2021structured, chen2022improved, mitra2024fast} and constrained space~\citep{kook2024inandout, kook2024sampling}. Our main result for the Proximal Sampler, Theorem~\ref{thm:ProximalMain}, considers the ideal Proximal Sampler which assumes access to an RGO which outputs a sample from~\eqref{eq:RGOdistribution} without any bias. This is assumed in many works~\citep{lee2021structured, chen2022improved, mitra2024fast}.
Improved analysis of the Proximal Sampler and better RGO implementations is an active area of research~\citep{liang2022proximal,liang2023a,liang2024proximal,fan2023improved,altschuler2024faster}. 

For clarity, we provide a brief review of a basic RGO implementation via rejection sampling in the smooth case, and state the corresponding oracle complexity (i.e. the expected number of calls to a first-order oracle of $f$) when using the Proximal Sampler with a rejection sampling-based RGO implementation in Corollary~\ref{cor:ProximalWithRejectionSamplingRGO}. RGO implementations in more general setups, where the potential function $f$ can be nonsmooth, weakly-smooth, and even nonconvex, have been extensively studied in~\citep{liang2022proximal,liang2023a,liang2024proximal,fan2023improved}. In these general settings, related proximal optimization problems are also efficiently solved within the RGO implementations.

\paragraph{Rejection Sampling}
Suppose $\pi \propto \exp{(-V)}$ on $\R^d$ is $\beta$-strongly log-concave and $M$-smooth. The rejection sampling method to sample from $\pi$ is the following:
\begin{enumerate}
    \item Compute the minimizer $x^*$ of $V$, so that for any $z \in \R^d$, $V(z) \geq V(x^*) + \frac{\beta}{2}\|z-x^*\|^2$.
    \item Draw $Z \sim \N(x^*, \frac{1}{\beta}I)$ and accept it with probability
    \[
    \exp \left( -V(Z) + V(x^*) + \frac{\beta}{2}\|Z-x^*\|^2 \right).
    \]
    Repeat this until acceptance.
\end{enumerate}

The output of this method is distributed according to $\pi$ and the expected number of iterations is $(\frac{M}{\beta})^{d/2}$ \citep[Theorem~7]{chewi2022query}.

\medskip

We can use rejection sampling to implement a RGO as follows. Recall the conditional distribution we seek to sample from~\eqref{eq:RGOdistribution}:
\[
\nu^{X \mid Y}(x \mid y) \propto_x \exp \left( -f(x) - \frac{\|x-y\|^2}{2\eta} \right).
\]
Define $g_y(x) \coloneqq  f(x) + \frac{\|x-y\|^2}{2\eta}$ so that for any fixed $y \in \R^d$, the target distribution for the RGO is $\tilde \nu_y (x) \propto \exp(-g_y(x))$.
Suppose the potential function $f$ is $L$-smooth and that $\eta < \frac{1}{L}$. In this case, $\tilde \nu_y$ is strongly log-concave with condition number $\frac{1+L\eta}{1-L\eta}$.

Using rejection sampling to sample from $\tilde \nu_y$ with $\eta \asymp \frac{1}{Ld}$ gives a valid implementation of the RGO under smoothness of $f$ with $\cO(1)$ many iterations in expectation. 
Specifically, $M = L + \frac{1}{\eta}$ and $\beta = -L + \frac{1}{\eta}$ and therefore, $\frac{M}{\beta} = \frac{1+L\eta}{1-L\eta}$. So if $\eta = \frac{1}{Ld}$, $(\frac{M}{\beta})^{d/2} = (1+\frac{2}{d-1})^{d/2} = \cO(1)$.

We therefore have the following corollary describing the oracle complexity of the Proximal Sampler with a rejection sampling-based RGO. For simplicity, in Corollary~\ref{cor:ProximalWithRejectionSamplingRGO} we assume $\rho_1^X$ satisfies a $\Phi$-Sobolev inequality with optimal constant $\alpha_{\PSI}(\rho_1^X)$ as opposed to some $\rho_\ell^X$ for $\ell \geq 1$ as we state in Theorem~\ref{thm:ProximalMain}. Changing this only corresponds to an additive constant $\ell$ term in the bound below.

\begin{corollary}\label{cor:ProximalWithRejectionSamplingRGO}
        Suppose $\nu^X \propto \exp{(-f)}$ is $\alpha$-strongly log-concave and $L$-smooth. Let the joint law of $(X_i, X_j)$ be $\rho_{i,j}^X$ and suppose $\rho_1^X$ satisfies a $\Phi$-Sobolev inequality with optimal constant $\alpha_{\PSI}(\rho_1^X)$. Then for any $\epsilon > 0$, the Proximal Sampler~\eqref{eqs:ProximalSampler} with $\eta \asymp \frac{1}{Ld}$ and with a rejection sampling-based RGO implementation (as described above) outputs $X_k \sim \rho_k^X$ with $\MI_{\Phi}(\rho_{0,k}^X) \leq \epsilon$ as long as $k \geq \frac{Ld}{2 \min\{\alpha, \alpha_{\PSI}(\rho_1^X)\}}\log \frac{\MI_{\Phi}(\rho_{0,1}^X)}{\epsilon} $\,. The expected number of oracle calls to $\nabla f$ is 
        \[
        \cO\left( \frac{Ld}{\min\{\alpha, \alpha_{\PSI}(\rho_1^X)\}}\log \frac{\MI_{\Phi}(\rho_{0,1}^X)}{\epsilon} \right).
        \]
\end{corollary}

\section{Rate of Change of Divergence Between Simultaneous Evolutions}

Here we describe the rate of change of $\Phi$-divergence along simultaneous evolutions of the same SDE. The following lemma is crucial in the analyses of all Langevin-type Markov chains considered in this paper. It is identical to \citet[Lemma~8]{mitra2024fast} and presented below for completeness. Similar results can be found in~\cite[Lemmas~12 and 15]{chen2022improved} and~\citet[Theorem~8.3.1]{chewi2023log}.

\begin{lemma}\label{lem:SimultaneousSDE}
    Suppose $X_t \sim \mu_t$ and $X_t \sim \nu_t$ with initial conditions $\mu_0$ and $\nu_0$ are two solutions of the following SDE:
    \begin{equation}\label{eq:GeneralSDE}
    \D X_t = b_t(X_t) \D t + \sqrt{2c} \D W_t\,,
    \end{equation}
    where $b_t : \R^d \to \R^d$ is a time-varying drift function, $c$ is a positive constant, and $W_t$ is the standard Brownian motion on $\R^d$.
    Then for all $t \ge 0$,
    \begin{equation*}
        \frac{\D}{\D t}\sfD_\Phi (\mu_t \dvert \nu_t) 
        = -c \, \FI_\Phi(\mu_t \dvert \nu_t).
    \end{equation*}
\end{lemma}

\begin{proof}
Begin by recalling that if $X_t \sim \rho_t$ where $\D X_t = b_t(X_t) \D t + \sqrt{2c} \D W_t$, then $\rho_t : \R^d \to \R$ satisfies the Fokker-Planck equation, given by:
\[
\partial_t \rho_t = -\nabla \cdot (\rho_t \, b_t) + c \Delta \rho_t\,.
\]
Also note that using the identity $\Delta \rho = \nabla \cdot (\rho \nabla \log \rho)$, the above can be written as: 
\begin{equation}\label{eq:EquivalentPDEForms}
\partial_t \rho_t = -\nabla \cdot (\rho_t \, b_t) + c \nabla \cdot (\rho_t \nabla \log \rho_t)\,.
\end{equation}

We identify $\mu_t$ and $\nu_t$ with their densities with respect to Lebesgue measure, and further denote their relative density as $h_t = \frac{\mu_t}{\nu_t}$\,. We also assume enough regularity to take the differential under the integral sign and use $\int f g$ as a shorthand for $\int f(x) g(x) \D x$. Throughout the proof, we use integration by parts in various steps, denoted by (IBP).  With all of this in mind, we have the following:
\begin{align*}
    \partial_t\, \sfD_\Phi (\mu_t \dvert \nu_t) &= \partial_t \int \nu_t\, \Phi (h_t) \\
    &= \int (\partial_t \nu_t) \Phi (h_t)  + \int \nu_t\, \big( \partial_t\, \Phi (h_t) \big) \\
    &= \int (\partial_t \nu_t)  \Phi(h_t)  + \int \nu_t \Phi'(h_t) \frac{\nu_t\, \partial_t \mu_t - \mu_t\, \partial_t \nu_t}{\nu_t^2}  \\
    &= \underbrace{\int (\partial_t \nu_t)  \Phi(h_t)}_{T_1}  + \underbrace{\int \Phi'(h_t) (\partial_t \mu_t)}_{T_2}  - \underbrace{\int \Phi'(h_t) \frac{\mu_t}{\nu_t} (\partial_t \nu_t)}_{T_3}
\end{align*}
We will now handle each of these terms separately. 
We have:
\begin{align*}
    T_1 &= \int (\partial_t \nu_t)  \Phi(h_t)\\
    &\stackrel{\eqref{eq:EquivalentPDEForms}}= \int \left( -\nabla \cdot (\nu_t \, b_t) + c \nabla \cdot (\nu_t \nabla \log \nu_t) \right) \Phi(h_t)\\
    &= -\int \nabla \cdot (\nu_t \, b_t) \Phi(h_t) + c \int \nabla \cdot (\nu_t \nabla \log \nu_t) \Phi(h_t)\\
    &\stackrel{\text{(IBP)}}= \int \langle \nu_t b_t , \nabla (\Phi (h_t)) \rangle  -c \int \langle \nu_t \nabla \log \nu_t , \nabla (\Phi (h_t)) \rangle\\
    &=  \int \langle \nu_t b_t , \Phi' (h_t) \nabla \frac{\mu_t}{\nu_t} \rangle  -c \int \langle \nu_t \nabla \log \nu_t , \Phi'(h_t) \nabla \frac{\mu_t}{\nu_t} \rangle
\end{align*}

We also have:
\begin{align*}
    T_2 &= \int \Phi'(h_t) (\partial_t \mu_t) \\
    &\stackrel{\eqref{eq:EquivalentPDEForms}}= \int \Phi'(h_t) \left( -\nabla \cdot (\mu_t \, b_t) + c \nabla \cdot (\mu_t \nabla \log \mu_t)  \right)\\
    &\stackrel{\text{(IBP)}}= \int \langle \mu_t \, b_t, \nabla (\Phi'(h_t)) \rangle - c \int \langle \mu_t \nabla \log \mu_t, \nabla (\Phi'(h_t)) \rangle\\
    &= \int \langle \mu_t \, b_t, \Phi''(h_t) \nabla \frac{\mu_t}{\nu_t} \rangle -c \int \langle \mu_t \nabla \log \mu_t, \Phi''(h_t) \nabla \frac{\mu_t}{\nu_t} \rangle 
\end{align*}

We also have:
\begin{align*}
    T_3 &= \int \Phi'(h_t) \frac{\mu_t}{\nu_t} (\partial_t \nu_t)\\
    &\stackrel{\eqref{eq:EquivalentPDEForms}}= \int \Phi'(h_t) \frac{\mu_t}{\nu_t} \left( -\nabla \cdot (\nu_t \, b_t) + c \nabla \cdot (\nu_t \nabla \log \nu_t) \right)\\
    &\stackrel{\text{(IBP)}}= \int \left \langle \nu_t \, b_t, \nabla \left(\Phi'(h_t) \frac{\mu_t}{\nu_t} \right)  \right \rangle - c \int \left \langle \nu_t \nabla \log \nu_t, \nabla \left(\Phi'(h_t) \frac{\mu_t}{\nu_t} \right) \right \rangle\\
    &=  \int \langle \mu_t \, b_t, \Phi''(h_t) \nabla \frac{\mu_t}{\nu_t} \rangle + \int \langle \nu_t b_t , \Phi' (h_t) \nabla \frac{\mu_t}{\nu_t} \rangle\\
    &\hspace{0.4cm} -c \int \left \langle \mu_t \nabla \log \nu_t, \Phi''(h_t) \nabla \frac{\mu_t}{\nu_t} \right \rangle  - c \int \left \langle \nu_t \nabla \log \nu_t, \Phi'(h_t) \nabla \frac{\mu_t}{\nu_t} \right \rangle 
\end{align*}

Therefore, combining the above, we see many terms cancel and we have the following:
\begin{align*}
    \partial_t\, \sfD_\Phi (\mu_t \dvert \nu_t) &= T_1 + T_2 - T_3\\
    &= c\int \left \langle \mu_t \nabla \log \nu_t, \Phi''(h_t) \nabla \frac{\mu_t}{\nu_t} \right \rangle  -c \int \left \langle \mu_t \nabla \log \mu_t, \Phi''(h_t) \nabla \frac{\mu_t}{\nu_t} \right \rangle  \\
    &= - c\int  \left \langle  \nabla \log \frac{\mu_t}{\nu_t},  \Phi''(h_t) \nabla \frac{\mu_t}{\nu_t}    \right \rangle \mu_t \\
    &= -c\, \E_{\mu_t} \left[ \left \langle  \nabla \log \frac{\mu_t}{\nu_t},  \Phi''(h_t) \nabla \frac{\mu_t}{\nu_t}    \right \rangle \right]\\
    &= -c\, \E_{\nu_t} \left[ \frac{\mu_t}{\nu_t} \left \langle  \nabla \log \frac{\mu_t}{\nu_t},  \Phi''(h_t) \nabla \frac{\mu_t}{\nu_t}    \right \rangle \right]  \\
    &= -c\, \E_{\nu_t} \left[ \left \langle  \nabla \frac{\mu_t}{\nu_t},  \Phi''(h_t) \nabla \frac{\mu_t}{\nu_t}    \right \rangle \right]  \\
    &= -c \, \FI_\Phi(\mu_t \dvert \nu_t)\,,
\end{align*}
which proves the desired statement.
\end{proof}

\section{From Mixing to Independence for Mutual Information}\label{app:IndependenceViaMixing}

We show that in the special case of $\Phi(x) = x \log x$, i.e. the standard mutual information, the contraction of mutual information can be bounded in terms of the mixing time in KL divergence.
Let $X_i \sim \rho_i$ for $i \geq 0$ be iterates along a Markov chain $\bP$ starting from $X_0 \sim \rho_0$ and denote the joint law of $(X_0, X_k)$ by $\rho_{0,k}$\,. If $\bP$ mixes in KL divergence from initial distributions taken to be point masses, i.e. $\delta_x$ for all $x \in \R^d$, then $\MI(\rho_{0,k})$ decreases at the same rate as the KL divergence to the stationary distribution.

\begin{lemma}\label{lem:IndependenceTimeleqMixingTime}
    Let $\bP$ be a Markov chain with stationary distribution $\nu$ and iterates $X_i \sim \rho_i$ for $i \geq 0$. Further denote the joint law of $(X_i, X_j)$ to be $\rho_{i.j}$\,. Given any error threshold $\epsilon \geq 0$, suppose there exists $k \in \bbN$ such that $\KL(\delta_x \bP^k \dvert \nu) \leq \epsilon$ for all $x \in \R^d$. Then, we have
    \[
    \MI(\rho_{0,k}) \leq \epsilon\,.
    \]
\end{lemma}

\begin{proof}
    Let $X_0 \sim \rho_0$. For any $i \geq 1$, we have $X_i \sim \rho_i =  \rho_0 \bP^i$. 
    It follows from Definition \ref{def:PhiMutualInformation} and direct calculation that
    \begin{align*}
        \MI(\rho_{0,k}) &\stackrel{\eqref{eq:MIPhi}}= \E_{x \sim \rho_0}[\KL(\rho_{k|0=x} \dvert \rho_k)] = \E_{x \sim \rho_0}[\KL(\delta_x \bP^k \dvert \rho_k)]\\
        &= \E_{x \sim \rho_0}\left[ \E_{\delta_x \bP^k} \left[ \log \frac{\delta_x \bP^k}{\rho_k}\right]\right] 
        = \E_{x \sim \rho_0}\left[ \E_{\delta_x \bP^k} \left[ \log \frac{\delta_x \bP^k}{\nu} - \log \frac{\rho_k}{\nu}\right]\right]\\
        &= \E_{x \sim \rho_0}[\KL(\delta_x \bP^k \dvert \nu)] - \E_{\rho_k} \left[  \log \frac{\rho_k}{\nu}\right] 
        = \E_{x \sim \rho_0}[\KL(\delta_x \bP^k \dvert \nu)] - \KL(\rho_k \dvert \nu) \\
        &\leq \E_{x \sim \rho_0}[\KL(\delta_x \bP^k \dvert \nu)]  \leq \epsilon\,,
    \end{align*}
    where the first inequality is by the non-negativity of KL divergence and the second inequality is due to the assumption that $\KL(\delta_x \bP^k \dvert \nu) \leq \epsilon$ for all $x \in \R^d$.
\end{proof}

We emphasize that the proof of Lemma~\ref{lem:IndependenceTimeleqMixingTime} does not extend to general $\Phi$-divergences and $\Phi$-mutual informations.
The key step in the proof of Lemma \ref{lem:IndependenceTimeleqMixingTime}, of expressing
\begin{equation}\label{eq:SeparateOutKL}
        \E_{x \sim \rho_0}[\KL(\delta_x \bP^k \dvert \rho_k)] = \E_{x \sim \rho_0}[\KL(\delta_x \bP^k \dvert \nu)] - \KL(\rho_k \dvert \nu)\,,
    \end{equation}
need not hold for general $\Phi$-divergences.
Indeed, \eqref{eq:SeparateOutKL} can alternatively be derived from the three-point identity of Bregman divergence.
It is known that KL divergence is the Bregman divergence of the entropy functional $\sfH(\rho)=-\E_{\rho} [\log \rho]$, that is $\KL (\rho \dvert \nu) = D_{-\sfH}(\rho; \nu)$.
So, KL divergence is special as it is both a $\Phi$-divergence and a Bregman divergence. 
In general, a $\Phi$-divergence need not be a Bregman divergence.

We would also like to mention that the requirement of having $\bP$ to mix from Dirac initializations is common for discrete-space chains~\citep{montenegro2006mathematical} but non-trivial in the setting of continuous-space Markov chains.
Mixing guarantees usually have dependence on the initial divergence to the target distribution, which becomes undefined in continuous-space (for example in KL divergence) when the initial distribution is not absolutely continuous with respect to the stationary distribution. Hence most mixing time guarantees for Langevin-type Markov chains implicitly require some regularity of the initial distribution. 
Challenges pertaining to the regularity of the initial distribution are prevalent in continuous-space sampling~\citep{boursier2023universal, koehler2024efficiently}, with corresponding results therefore restricted to TV distance or Wasserstein metric, which remain finite for Dirac distributions.

\section{Strong Data Processing Inequalities}\label{app:SDPI}

Recall the discussion on Strong Data Processing Inequalities (SDPIs) in Section~\ref{sec:Intro_SDPI}. Here, we provide a more comprehensive introduction to these inequalities. 
When studying SDPIs in $\Phi$-divergence, the key quantity to bound is the contraction coefficient in $\Phi$-divergence~\eqref{eq:ContractionCoefficientMainText}. This is formally defined as follows.

\begin{definition}\label{def:ContractionCoefficient_PhiDivergence}
    Let $\pi$ be a probability distribution and $\bP$ be a Markov kernel. Then the $\sfD_{\Phi}$-contraction coefficient $\varepsilon_{\sfD_{\Phi}}$ is defined as follows
    \begin{equation}\label{eq:PhiDivergenceContractionCoefficient}
        \varepsilon_{\sfD_{\Phi}} (\bP, \pi) \coloneqq \sup_{\mu\,:\, 0<\sfD_{\Phi}(\mu\dvert \pi) < \infty} \frac{\sfD_{\Phi}(\mu \bP \dvert \pi \bP)}{\sfD_{\Phi} (\mu \dvert \pi)}\,.
    \end{equation}
\end{definition}
The fact that~\eqref{eq:PhiDivergenceContractionCoefficient} is $\leq 1$ is guaranteed by the data processing inequality. We say that $(\bP, \pi)$ satisfies an SDPI in $\Phi$-divergence when $\varepsilon_{\sfD_{\Phi}} (\bP, \pi) < 1$.
Using Definition~\ref{def:ContractionCoefficient_PhiDivergence}, we state the definition of an SDPI in $\Phi$-divergence.

\begin{definition}\label{def:SDPI_PhiDivergence}
Let $\pi$ be a probability distribution and $\bP$ be a Markov kernel, and further define $\varepsilon_{\sfD_{\Phi}} (\bP, \pi)$ as in~\eqref{eq:PhiDivergenceContractionCoefficient}. Then we say that $(\bP, \pi)$ satisfies a strong data processing inequality in $\Phi$-divergence when $\varepsilon_{\sfD_{\Phi}} (\bP, \pi) < 1$\,. In particular, we have,
\begin{equation}\label{eq:SDPI_PhiDivergence}
    \sfD_\Phi(\mu \bP \dvert \pi \bP) \leq \varepsilon_{\sfD_{\Phi}} (\bP, \pi) \sfD_\Phi (\mu \dvert \pi)\,,
\end{equation}
where $\mu$ is any distribution such that $\sfD_{\Phi}(\mu \dvert \pi) < \infty$.
\end{definition}

Observe that SDPIs immediately yield mixing guarantees for the Markov chain. Suppose $\nu$ is stationary for $\bP$, i.e. $\nu \bP = \nu$, and $\mu$ is the initial distribution for the Markov chain. Then repeated application of~\eqref{eq:SDPI_PhiDivergence} yields that
\begin{equation}\label{eq:MixingViaSDPI}
    \sfD_{\Phi}(\mu \bP^k \dvert \nu) \leq \varepsilon_{\sfD_{\Phi}}^k (\bP, \nu)\,\sfD_{\Phi}(\mu \dvert \nu)\,.
\end{equation}
Therefore, so long as $\varepsilon_{\sfD_{\Phi}} (\bP, \nu) < 1$ (where $\nu$ is stationary for $\bP$), we can get quantitative mixing time bounds for $\bP$ in $\Phi$-divergence.
Bounding the contraction coefficient $\varepsilon_{\sfD_{\Phi}} (\bP, \nu)$ is the key challenge in SDPI-based mixing time analyses.

SDPIs can also be stated in terms of information. Suppose $U \to X \to Y$ forms a Markov chain, then the data processing inequality states that $\MI_{\Phi}(U; Y) \leq \MI_{\Phi}(U; X)$. 
SDPIs in $\Phi$-mutual information are stated in terms of the contraction coefficient in $\Phi$-mutual information. This is defined as follows.

\begin{definition}\label{def:ContractionCoefficient_PhiMutualInformation}
    Let $\pi$ be a probability distribution and $\bP$ be a Markov kernel. Furthermore, let $X \sim \pi$ and $Y \sim \pi \bP$\,. Then the $\MI_{\Phi}$-contraction coefficient $\varepsilon_{\MI_{\Phi}}$ is defined as follows
    \begin{equation}\label{eq:PhiMutualInformationContractionCoefficient}
        \varepsilon_{\MI_{\Phi}} (\bP, \pi) \coloneqq \sup_{\rho^{U \mid X}} \frac{\MI_{\Phi}(U; Y)}{\MI_{\Phi} (U; X)}\,,
    \end{equation}
    where $U \to X \to Y$ forms a Markov chain and the $\sup$ is over all conditional distributions of $U \mid X$, denoted as $\rho^{U \mid X}$. 
\end{definition}
Another way to think of the $\sup$ in~\eqref{eq:PhiMutualInformationContractionCoefficient} is over all Markov chains $ U \to X \to Y$ with fixed joint distribution $\rho^{XY}(x,y) = \pi(x) \bP (y \,|\, x)$.
Using Definition~\ref{def:ContractionCoefficient_PhiMutualInformation}, we state SDPIs in $\Phi$-mutual information. 

\begin{definition}\label{def:SDPI_PhiMutualInformation}
Let $\pi$ be a probability distribution and $\bP$ be a Markov kernel, and further define $\varepsilon_{\MI_{\Phi}} (\bP, \pi)$ as in~\eqref{eq:PhiMutualInformationContractionCoefficient}. Then we say that $(\bP, \pi)$ satisfies a strong data processing inequality in $\Phi$-mutual information when $\varepsilon_{\MI_{\Phi}} (\bP, \pi) < 1$\,. In particular, we have,
\begin{equation}\label{eq:SDPI_PhiMutualInformation}
    \MI_{\Phi}(U; Y) \leq \varepsilon_{\MI_{\Phi}} (\bP, \pi) \MI_{\Phi}(U; X)\,,
\end{equation}
where $U \to X \to Y$ is any Markov chain where $X \sim \pi$, and $Y \sim \pi \bP$.
\end{definition}

In Lemma~\ref{lem:RelatePhiDivergenceAndPhiMICoefficients}, we show that $\varepsilon_{\MI_{\Phi}}$ can always be upper bounded by $\varepsilon_{\sfD_{\Phi}}$. The converse also holds when $\Phi$ is bounded in a neighbourhood of $1$~\citep[Theorem~5.2]{raginsky2016strong}).
Hence, although we are interested in SDPIs in information, we will bound the contraction coefficients in divergence out of convenience.

\begin{lemma}\label{lem:RelatePhiDivergenceAndPhiMICoefficients}
    For any probability distribution $\pi$ and Markov kernel $\bP$, we have that
    \[
    \varepsilon_{\MI_{\Phi}} (\bP, \pi) \leq \varepsilon_{\sfD_{\Phi}} (\bP, \pi)\,.
    \]
\end{lemma}

\begin{proof}
    Consider a Markov chain $U \to X \to Y$ where $U \sim \rho^U$, $ X \sim \rho^X = \pi$ and $Y \sim \rho^Y = \pi \bP$.
    It follows from Definitions \ref{def:PhiMutualInformation} and \ref{def:ContractionCoefficient_PhiDivergence} that
    \begin{align*}
        \MI_{\Phi}(\rho^{UY}) &\stackrel{\eqref{eq:MIPhi}}= \E_{u \sim \rho^U} \left[\sfD_{\Phi}(\rho^{Y \mid U = u} \dvert \rho^Y)\right]\,,\\
        &= \E_{u \sim \rho^U} \left[  \sfD_{\Phi}(\rho^{X \mid U = u} \bP \dvert \rho^X \bP) \right]\,,\\
        &\stackrel{\eqref{eq:PhiDivergenceContractionCoefficient}}\leq \E_{u \sim \rho^U} \left[  \varepsilon_{\sfD_{\Phi}} (\bP, \pi) \sfD_{\Phi}(\rho^{X \mid U = u} \dvert \rho^X) \right]\,,\\
        &\stackrel{\eqref{eq:MIPhi}}= \varepsilon_{\sfD_{\Phi}}(\bP, \pi) \MI_{\Phi}(\rho^{UX})\,.
    \end{align*}
    Hence, the claim immediately follows from Definition \ref{def:ContractionCoefficient_PhiMutualInformation}.
\end{proof}

To study the contraction of $\Phi$-mutual information along multiple steps of a Markov chain, we have the following lemma.

\begin{lemma}\label{lem:IndependenceContractionBySDPI}
    Let $X_i \sim \rho_i$ be iterates along a Markov chain $\bP$ starting from $X_0 \sim \rho_0$. Then for any $\ell \geq 1$ and $k \geq \ell$
    \begin{equation*}\label{eq:IndependenceViaSDPI}
    \MI_{\Phi}(X_0; X_k) \leq  \prod_{i=\ell}^{k-1} \varepsilon_{\sfD_{\Phi}}(\bP, \rho_i) \, \MI_{\Phi}(X_0; X_\ell)\,.
\end{equation*}
\end{lemma}

\begin{proof}
    Applying~\eqref{eq:SDPI_PhiMutualInformation} $k-\ell$ times reveals that
    \begin{equation*}
    \MI_{\Phi}(X_0; X_k) \leq  \prod_{i=\ell}^{k-1} \varepsilon_{\MI_{\Phi}}(\bP, \rho_i) \, \MI_{\Phi}(X_0; X_\ell)\,.
\end{equation*}
From Lemma~\ref{lem:RelatePhiDivergenceAndPhiMICoefficients} we know that $\varepsilon_{\MI_{\Phi}}(\bP, \rho_i) \leq \varepsilon_{\sfD_{\Phi}}(\bP, \rho_i)$. Therefore, we get the desired claim.
\end{proof}

Therefore, we can see from Lemma~\ref{lem:IndependenceContractionBySDPI} that to compute the contraction of information, we need to compute the contraction coefficients along the trajectory of the Markov chain $\varepsilon_{\MI_{\Phi}}(\bP, \rho_i)$ as opposed to just the coefficient for the stationary distribution $\varepsilon_{\sfD_{\Phi}} (\bP, \nu)$ for mixing time~\eqref{eq:MixingViaSDPI}. This analysis along the trajectory makes studying the information functional more challenging.

\section{Regularity-Based Bounds for the Standard Mutual Information}

We now consider the special case of $\Phi(x) = x \log x$. We present an alternate set of results studying the convergence of mutual information by exploiting the regularity properties of the Markov chains. Regularity results for a Markov chain seek to bound the change in the initial data after the Markov chain is applied to it. For example for $x,y \in \R^d$, Markov kernel $\bP$, and $C > 0$, bounds of the form
\[
\KL(\delta_x \bP \dvert \delta_y \bP) \leq C \|x-y\|_2^2,
\]
arise frequently in the analysis of Hamiltonian Monte Carlo~\citep{bou2023mixing} and other Markov chains~\citep{bobkov2001hypercontractivity, altschuler2023shifted}.
Our goal is to study the convergence of the mutual information and we do so by relating it to the KL divergence via $\MI(\rho^{XY}) = \E_{\rho^X}[\KL(\rho^{Y \mid X} \dvert \rho^Y)]$ (Definition~\ref{def:PhiMutualInformation}) and using regularity bounds for these Markov chains in KL divergence.
Before proceeding, we define the Wasserstein-2 distance, which arises in the following analyses.

\begin{definition}\label{def:W2}
    The $\cW_2$ distance between probability distributions $\mu$ and $\nu$ is given by
    \[
    \cW_2(\mu, \nu) \coloneqq \inf_{\gamma \in \cC(\mu, \nu)} \Big( \E_{(x,y) \sim \gamma}\big[\|x-y\|_2^2\big]\Big)^{1/2},
    \]
    where $\cC(\mu, \nu)$ denotes the set of couplings of $\mu$ and $\nu$. 
\end{definition}

We study the Langevin dynamics, ULA, and Proximal Sampler in Appendices~\ref{app:RegularityLD},~\ref{app:RegularityULA}, and~\ref{app:RegularityPS} respectively. In Appendix~\ref{app:RegularityLD_AlternateTheorem} we present Theorem~\ref{thm:LangevinRegSLC} describing the regularity properties of the Langevin dynamics for strongly log-concave targets. We de not utilize Theorem~\ref{thm:LangevinRegSLC} when studying mutual information but include it as it might be of independent interest.

\subsection{Regularity-Based Bounds for Mutual Information along Langevin Dynamics}\label{app:RegularityLD}

\begin{theorem}
\label{thm:MI_LD_regularize}
Let $X_t \sim \rho_t$ evolve following the Langevin dynamics \eqref{eq:LangevinDynamics} from $X_0 \sim \rho_0$ and let the joint law of $(X_0, X_t)$ be $\rho_{0,t}$\,. The mutual information $\MI(\rho_{0,t})$ satisfies the following:

\begin{enumerate}
    \item[(a)] if $\nu$ is weakly log-concave, then for all $t>0$
        \[
        \MI(\rho_{0,t}) \le \frac{1}{2t} \Var(X_0).
        \]
    \item[(b)] if $\nu$ is $\alpha$-strongly log-concave for some $\alpha > 0$, then for all $t > 0$
    \begin{equation}\label{eq:RegularityLDBoundForSLC}
        \MI(\rho_{0,t}) \le \frac{\alpha}{e^{2\alpha t}-1} \Var(X_0).
    \end{equation}
\end{enumerate}
\end{theorem}

\begin{proof}
    We first consider case when $\nu$ is $\alpha$-strongly log-concave as the weakly log-concave case follows by taking the limit $\alpha \to 0$.
    Let $\bP_t$ be the Markov kernel associated with the Langevin dynamics. 
    Note that $\rho_{t \mid 0}(\cdot \mid x_0) = \delta_{x_0} \bP_t$ is the law at time $t$ if we start the Langevin dynamics from initial distribution $\delta_{x_0}$, for any fixed $x_0 \in \R^d$.
    It then follows from~\cite[Corollary~1]{altschuler2023shifted}\footnote{Also note that~\cite[Eq. (4.5)]{bobkov2001hypercontractivity} is equivalent to~\cite[Corollary~1]{altschuler2023shifted} via the duality between log-Harnack and reverse transport inequalities as explained in~\cite[Section~VI.B]{altschuler2023shifted}.} that
    \[
    \KL(\delta_{x_0} \bP_t \dvert \rho_0 \bP_t) \le \frac{\alpha}{2(e^{2\alpha t}-1)} \cW_2^2(\delta_{x_0}, \rho_0)
= \frac{\alpha}{2(e^{2\alpha t}-1)} \E_{X\sim\rho_0}\left[\|X-x_0\|^2\right].
\]
Using Definition~\ref{def:PhiMutualInformation} and the above inequality, we have
\[
    \MI(\rho_{0,t})
    = \E_{x_0 \sim \rho_0}[\KL(\delta_{x_0} \bP_t \dvert \rho_0 \bP_t)] 
    \le \frac{\alpha}{2(e^{2\alpha t}-1)} \E_{X,x_0\sim\rho_0}\left[\|X-x_0\|^2 \right]
    = \frac{\alpha}{e^{2\alpha t}-1} \Var(X_0),
\]
where in the above, $X, x_0 \sim \rho_0$ are independent, and we have used the variance formula $\Var (X_0) = \frac{1}{2} \E_{X,x_0\sim\rho_0}\left[\|X-x_0\|^2 \right]$.
Taking $\alpha \to 0$ yields the weakly log-concave result.

\end{proof}

Observe that Theorem~\ref{thm:MI_LD_regularize} does not require a $\Phi$-Sobolev inequality assumption on $\rho_s$ and also works for the weakly log-concave case. 
Theorem~\ref{thm:MI_LD_regularize}(b) with $t = s$ can be used to bound $\MI(\rho_{0,s})$ from Theorem~\ref{thm:LangevinDynamicsMain} (for the case of $\Phi(x) = x \log x$).

\subsection{Regularity-Based Bounds for Mutual Information along ULA}\label{app:RegularityULA}

\begin{theorem}\label{thm:MI_ULA_regularize}
Suppose $\nu$ is $\alpha$-strongly log-concave and $L$-smooth for some $0 < \alpha \leq L < \infty$. 
    Let $X_k \sim \rho_k$ evolve following ULA~\eqref{eq:ULA} with step-size $\eta \leq 1/L$ from $X_0 \sim \rho_0$.
Define $\gamma = 1-\alpha \eta \in (0,1)$ and let the joint law of $(X_0, X_k)$ be $\rho_{0,k}$\,.
Then for all $k \ge 1$, we have 
\begin{equation*}
    \MI(\rho_{0,k}) \le
    \frac{\alpha \, \gamma^{2k}}{1-\gamma^{2k}} \Var(X_0).
\end{equation*}
\end{theorem}

\begin{proof}
The proof is similar to that of Theorem \ref{thm:MI_LD_regularize}.
Let $\bP^k$ denote the kernel for $k$-step of ULA.
It follows from \cite[Theorem~6]{altschuler2023shifted} with $L=1-\eta\alpha$ that
\[
    \KL(\delta_{x_0} \bP^k \dvert \rho_0 \bP^k) \le \frac{1-L^2}{4\eta (L^{-2k} -1)}\cW_2^2(\delta_{x_0}, \rho_0)  \le \frac{\alpha}{2 [(1-\eta\alpha)^{-2k} -1]} \E_{X\sim\rho_0}\left[\|X-x_0\|^2\right]. 
\]
Using Definition~\ref{def:PhiMutualInformation} and the above inequality, we have
\begin{align*}
    \MI(\rho_{0,k}) &= 
     \E_{x_0\sim \rho_0}[\KL(\delta_{x_0} \bP^k \dvert \rho_0 \bP^k)]\\
    &\le \frac{\alpha}{2 [(1-\eta\alpha)^{-2k} -1]} \E_{X,x_0\sim\rho_0}\left[\|X-x_0\|^2 \right] \\
    &= \frac{\alpha}{ [(1-\eta\alpha)^{-2k} -1]} \Var(X_0)\,,
\end{align*}
where in the above, $X, x_0 \sim \rho_0$ are independent, and we use the variance formula $\Var (X_0) = \frac{1}{2} \E_{X,x_0\sim\rho_0}\left[\|X-x_0\|^2 \right]$.
\end{proof}

Theorem~\ref{thm:MI_ULA_regularize} with $k=\ell$ provides a bound on $\MI(\rho_{0,\ell})$ from Theorem~\ref{thm:ULA_Main} and also does not require $\rho_\ell$ to satisfy a $\Phi$-Sobolev inequality assumption.

\subsection{Regularity-Based One-step Bounds for the Proximal Sampler}\label{app:RegularityPS}

We wish to bound $\MI_{}(\rho^X_{0,1})$ appearing in Theorem~\ref{thm:ProximalMain} when $\ell = 1$.

\begin{lemma}\label{lem:1StepBoundProximal}
    Consider the Proximal Sampler starting from $X_0 \sim \rho^X_0$ with step-size $\eta >0$. Then we have that
    \begin{equation*}
        \MI(\rho^X_{0,1}) \leq \frac{1}{2\eta} \Var(X_0).
    \end{equation*}
\end{lemma}

\begin{proof}
    Note that $\MI_{}(\rho^X_{0,1}) = \MI(X_0, X_1) \leq \MI (X_0; Y_0)$ due to the data processing inequality. By the Proximal Sampler~\eqref{eqs:ProximalSampler}, we know that the law of $Y_0 \mid X_0$ is obtained via evolving $\rho^X_0$ along the heat flow for time $\eta$. Given this, it follows from Theorem~\ref{thm:MI_LD_regularize}(a) with $t=\eta$ that $\MI(X_0; Y_0) \leq \frac{1}{2\eta} \Var(X_0)$.
\end{proof}

This yields the following corollary.

\begin{corollary}\label{cor:1StepAndInitialBoundProximal}
    Consider the same conditions as Theorem~\ref{thm:ProximalMain} with $\ell = 1$, i.e. suppose $\rho_1^X$ satisfies a $\Phi$-Sobolev inequality with optimal constant $\alpha_{\PSI}(\rho_1)$. Then
    \begin{equation*}
        \MI_{}(\rho^X_{0,k}) \leq \frac{\Var(X_0)}{2\eta(1 + \eta \min \{ \alpha, \alpha_{\PSI}(\rho^X_1) \})^{2(k-1)}}\,.
    \end{equation*}
\end{corollary}

\subsection{Regularity of Langevin Dynamics}\label{app:RegularityLD_AlternateTheorem}

So far in this appendix we have studied bounds on the mutual information via the regularity properties of the Langevin dynamics (Appendix~\ref{app:RegularityLD}), the Unadjusted Langevin Algorithm (Appendix~\ref{app:RegularityULA}), and the Proximal Sampler (Appendix~\ref{app:RegularityPS}).
We now focus on the continuous-time Langevin dynamics.

Recall from the proof of Theorem~\ref{thm:MI_LD_regularize} (in Appendix~\ref{app:RegularityLD}) how the regularity properties of the Langevin dynamics (as studied in~\citep[Corollary~1]{altschuler2023shifted} for strongly log-concave targets and simultaneous Dirac initializations) result in guarantees for the contraction of mutual information. 
In this appendix we provide an alternate regularity theorem for the Langevin dynamics (Theorem~\ref{thm:LangevinRegSLC}) for strongly log-concave targets and smooth initializations. Although we do not utilize Theorem~\ref{thm:LangevinRegSLC} to study mutual information contraction, we include it here as it might be of independent interest.

Before describing Theorem~\ref{thm:LangevinRegSLC} we state results from~\citep{OV01} which we build upon and which are for the log-concave case.

\begin{theorem}[{\cite[Theorem \& Corollary~2]{OV01}}]\label{thm:LangevinRegLC}
    Let $\nu$ be a log-concave distribution and let $X_t \sim \rho_t$ evolve along the Langevin dynamics~\eqref{eq:LangevinDynamics} from any $X_0 \sim \rho_0$ such that $\rho_0$ is smooth and $\cW_2(\rho_0, \nu) < \infty$.
    Then, for any $t > 0$, we have
    \begin{align*}
        t^2 \, \FI (\rho_t \dvert \nu) + 2t\, \KL(\rho_t \dvert \nu) + \cW_2^2(\rho_t,\nu) \,\le\, \cW_2^2(\rho_0,\nu)\,.
    \end{align*}
    In particular, for any $t > 0$\,,
    \begin{align}\label{Eq:LangevinReg1}
        \KL (\rho_t \dvert \nu) \le \frac{\cW_2^2(\rho_0, \nu)}{4t}\,.
    \end{align}
\end{theorem}

Although the smoothness assumption on $\rho_0$ may be restrictive, this bound is useful in many settings where $\KL (\rho_0 \dvert \nu) = \infty$ but $\cW_2 (\rho_0, \nu)$ is finite, such as the case where $\rho_0$ is Cauchy, and $\nu$ is a Gaussian.
We extend the analysis of~\citep{OV01} to SLC target in Theorem \ref{thm:LangevinRegSLC}. 
Note that this provides an alternate proof to~\citep[Lemma 4.2]{bobkov2001hypercontractivity}.

\begin{theorem}\label{thm:LangevinRegSLC}
    Let $\nu$ be $\alpha$-SLC for some $\alpha > 0$ and let $X_t \sim \rho_t$ evolve along the Langevin dynamics~\eqref{eq:LangevinDynamics} from any $X_0 \sim \rho_0$ such that $\rho_0$ is smooth and $\cW_2(\rho_0, \nu) < \infty$.
    Then, for any $t > 0$, we have
    \begin{align} \label{Eq:LangevinReg2Lyapunov}
    \frac{(e^{\alpha t}-1)^2}{\alpha^2} \FI(\rho_t \dvert \nu) + \frac{2(e^{\alpha t}-1)}{\alpha} \KL (\rho_t \dvert \nu) + e^{\alpha t} \cW_2^2(\rho_t,\nu) \,\le\,  \cW_2^2(\rho_0,\nu) \,.
    \end{align}
    In particular, for any $t > 0$\,,
    \begin{align}\label{Eq:LangevinReg2}
        \KL (\rho_t \dvert \nu) \le \frac{\alpha}{2(e^{2\alpha t}-1)} \cW_2^2(\rho_0, \nu) \,.
    \end{align}
\end{theorem}

\begin{proof}
    Similar to the proof of Theorem \ref{thm:LangevinRegLC} given in~\citep{OV01}, we seek to find a Lyapunov functional of the form
    \begin{equation}\label{Eq:CalcLyapunov}
        \psi(t) = A_t \, \FI(\rho_t \dvert \nu) + B_t \, \KL (\rho_t \dvert \nu) + \cW_2^2(\rho_t, \nu).
    \end{equation}
    for some $A_t, B_t \ge 0$ that will be determined later to ensure that $\psi(t)$ is decreasing exponentially fast.
    From Lemma~\ref{lem:SimultaneousSDE} for the Langevin dynamics (i.e. with $\mu_t = \rho_t$, $\nu_t = \nu$, $b_t(x) = -\nabla f(x)$ and $c=1$), we get that
    \begin{align}\label{Eq:dtCalc1}
        \frac{\D}{\D t} \KL(\rho_t \dvert \nu) &= -\FI(\rho_t \dvert \nu)\,. 
    \end{align}
    Recalling \cite[Formula 15.7]{villani2009optimal} \cite[eq.(10)-(12)]{WJ18b}, we can characterize the rate of change of relative Fisher information,
    \begin{align}
        \frac{\D}{\D t} \FI(\rho_t \dvert \nu) 
        &= -2\mathcal{K}_\nu(\rho_t) - 2 \E_{\rho_t}\left[ \left\langle \nabla \log \frac{\rho_t}{\nu}, (\nabla^2 f)  \nabla \log \frac{\rho_t}{\nu} \right\rangle\right] \nn \\
        &\le -2\alpha \E_{\rho_t}\left[ \left\| \nabla \log \frac{\rho_t}{\nu} \right\|^2 \right]  = -2\alpha \FI(\rho_t \dvert \nu). \label{Eq:dtCalc2}
    \end{align}
    where $\cK_\nu(\rho) := \E_\rho\left[\left\|\nabla^2 \log \frac{\rho}{\nu}\right\|^2_{\HS}\right]$ is the second-order relative Fisher information
    and the inequality follows from the facts that $\cK_\nu(\rho)\ge 0$ and $\nu$ is $\alpha$-SLC, i.e., $\nabla^2 f \succeq \alpha I$.
    Also, recall the lemma from \citep{OV01}, which shows that
    \begin{align}
        \frac{\D}{\D t} \Bigg|^+ \cW_2^2(\rho_t, \nu) 
        &\le -2\E_{\rho_t}\left[\left\langle x-\nabla \varphi_t(x), \, \nabla \log \frac{\rho_t(x)}{\nu(x)} \right\rangle\right] \notag \\
        &\le -2\KL(\rho_t \dvert \nu) - \alpha \cW_2^2(\rho_t,\nu) \label{Eq:dtCalc3} 
    \end{align}
    where $(\D/\D t)^+$ is the upper derivative, and $\nabla \varphi_t$ is the (unique) gradient of the convex function that pushes forward $\rho_t$ to $\nu$, i.e., $\nabla \varphi_t \# \rho_t = \nu$.
    The second inequality above follows since $\nu$ is $\alpha$-SLC, which implies that relative entropy $\rho \mapsto \KL(\rho \dvert \nu)$ is $\alpha$-strongly convex, which further means that
    \begin{align}\label{Eq:dtCalc4}
        \KL (\nu \dvert \nu) - \KL(\rho_t \dvert \nu) - \E_{\rho_t}\left[\left\langle \nabla \varphi_t(x) - x, \, \nabla \log \frac{\rho_t(x)}{\nu(x)} \right\rangle\right] \ge \frac{\alpha}{2} \cW_2^2(\rho_t,\nu) \,.
    \end{align}
    \noindent
    Combining bounds \eqref{Eq:dtCalc1}, \eqref{Eq:dtCalc2}, and \eqref{Eq:dtCalc3}, along with the Lyapunov functional \eqref{Eq:CalcLyapunov}, we have
    \begin{align*}
        \frac{\D}{\D t} \Bigg|^+ \psi(t)
        &= (\dot A_t - B_t) \FI (\rho_t \dvert \nu) + A_t \frac{\D}{\D t} \FI (\rho_t \dvert \nu) + \dot B_t \KL(\rho_t \dvert \nu) + \frac{\D}{\D t} \Bigg|^+ \cW_2^2(\rho_t,\nu) \\
        &\le -\big(2\alpha A_t + B_t - \dot A_t\big) \, \FI (\rho_t \dvert \nu) - (2-\dot B_t) \, \KL(\rho_t \dvert \nu) - \alpha \cW_2^2(\rho_t, \nu)\,.
    \end{align*}
    We want to choose $A_t$ and $B_t$ so that the Lyapunov functional decays exponentially fast with rate $\alpha$ along the Langevin dynamics.
    To this end, we set $2 - \dot B_t = \alpha B_t$ with $B_0 = 0$, which implies that
    \begin{align*}
        B_t = \frac{2(1-e^{-\alpha t})}{\alpha} \,.
    \end{align*}
    We also set $2\alpha A_t + B_t - \dot A_t = \alpha A_t$ with $A_0 = 0$, which can be solved to yield
    \begin{align*}
        A_t = \frac{1}{\alpha^2} (e^{\alpha t} + e^{-\alpha t} - 2) = \frac{e^{\alpha t}(1-e^{-\alpha t})^2}{\alpha^2} \,.
    \end{align*}
    With these choices, we have $\frac{\D}{\D t}|^+ \psi(t) \le -\alpha \psi(t)$, so indeed,
    \begin{equation}\label{ineq:psi1} 
        \psi(t) \le e^{-\alpha t} \psi(0) = e^{-\alpha t} \cW_2^2(\rho_0,\nu) \,,
    \end{equation}
    which, after proper normalization, is the claim in \eqref{Eq:LangevinReg2Lyapunov}.

    Next, we prove \eqref{Eq:LangevinReg2}.
    Inequality \eqref{Eq:dtCalc4} together with the Cauchy-Schwarz inequality actually implies the HWI inequality \cite[Section 9.4]{villani2021topics},
    \begin{align*}
        \sqrt{\FI (\rho_t \dvert \nu)} \, \cW_2(\rho_t,\nu) \,\ge \KL(\rho_t \dvert \nu) + \frac{\alpha}{2} \cW_2^2(\rho_t,\nu)\,.
    \end{align*}
    Using the fact that $a^2+b^2\ge 2ab$ and the above relation, for any $C_t \ge 0$, we have
    \begin{align*}
        A_t \, \FI (\rho_t \dvert \nu) + C_t \, \cW_2(\rho_t,\nu)^2 &\ge 2\sqrt{A_t C_t} \, \sqrt{\FI(\rho_t \dvert \nu)} \, \cW_2(\rho_t,\nu) \\
        &\ge 2\sqrt{A_t C_t} \, \left(\KL(\rho_t \dvert \nu) + \frac{\alpha}{2} \cW_2^2(\rho_t,\nu)\right)\,.
    \end{align*}
    We can then decompose and bound the Lyapunov functional $\psi(t)$ as follows
    \begin{align}
        \psi(t) &= A_t \, \FI(\rho_t \dvert \nu) + C_t \, \cW_2^2(\rho_t,\nu) + B_t \, \KL(\rho_t \dvert \nu) + (1-C_t) \, \cW_2^2(\rho_t,\nu) \nn \\
        &\ge \left(B_t + 2\sqrt{A_t C_t} \right) \, \KL(\rho_t \dvert \nu) + \left(1-C_t + \alpha \sqrt{A_tC_t}\right) \, \cW_2^2(\rho_t,\nu) \,. \label{ineq:psi}
    \end{align}
    We now choose $C_t$ such that $1-C_t + \alpha \sqrt{A_tC_t} = 0$.
    Choosing the positive solution, we have
    \begin{equation}\label{eq:Ct}
        \sqrt{C_t} = \frac{\alpha \sqrt{A_t} + \sqrt{\alpha^2 A_t + 4}}{2} \,.
    \end{equation}
    This choice of $C_t$, the previous choices $B_t + \alpha A_t = \dot A_t = \frac{1}{\alpha} (e^{\alpha t} - e^{-\alpha t})$ and the subsequent calculation
    \[
        A_t(\alpha^2 A_t + 4) = \frac{e^{\alpha t} (1-e^{-\alpha t})^2}{\alpha^2} \left( e^{\alpha t} + e^{-\alpha t} + 2 \right) 
        = \frac{(e^{\alpha t} - e^{-\alpha t})^2}{\alpha^2} \,,
    \]
    yield that
    \[
        B_t + 2\sqrt{A_tC_t} \stackrel{\eqref{eq:Ct}}= B_t + \alpha A_t + \sqrt{A_t(\alpha^2 A_t + 4)}
        = \frac{2}{\alpha} (e^{\alpha t} - e^{-\alpha t}) \,.
    \]
    Thus, with this choice of $C_t$, we have
    \[
        \psi_t
        \stackrel{\eqref{ineq:psi}}\ge (B_t + 2\sqrt{A_t C_t}) \, \KL(\rho_t \dvert \nu)
        = \frac{2}{\alpha} (e^{\alpha t} - e^{-\alpha t}) \, \KL(\rho_t \dvert \nu) \,.
    \]
    Therefore, we conclude that
    \[
        \KL(\rho_t \dvert \nu) \stackrel{\eqref{ineq:psi1}}\le \frac{\alpha}{2(e^{\alpha t} - e^{-\alpha t})} \, e^{-\alpha t} \, \cW_2^2(\rho_0,\nu) 
        = \frac{\alpha}{2(e^{2\alpha t} - 1)} \, \cW_2^2(\rho_0,\nu) \,,
    \]
    as claimed in \eqref{Eq:LangevinReg2}.
\end{proof}

We note that~\cite[Corollary~1]{altschuler2023shifted} follows a different proof to obtain a stronger version of Theorem~\ref{thm:LangevinRegSLC}, which applies for simultaneous Langevin dynamics and under no assumption on $\rho_0$\,. Also note that the dual form of the log-Harnack inequality~\cite[Equation (4.5)]{bobkov2001hypercontractivity} yields a version of Theorem~\ref{thm:LangevinRegSLC} which is applicable for simultaneous Langevin dynamics and from Dirac initializations. The duality between Harnack inequalities and reverse transport inequalities (such as~\eqref{Eq:LangevinReg2}) is explained in~\cite[Section~VI.B]{altschuler2023shifted}.

\section{Proofs for Langevin Dynamics}

Here we present the proofs of and related to Theorem~\ref{thm:LangevinDynamicsMain}. In Appendix~\ref{app:PSIEvolutionAlongLangevinDynamics} we study the evolution of the $\Phi$-Sobolev constant along the Langevin dynamics. This approach is crucial to both of our proof strategies for Theorem~\ref{thm:LangevinDynamicsMain}. In Appendix~\ref{app:LangevinMainDirect} we present the direct time derivative-based approach for Theorem~\ref{thm:LangevinDynamicsMain} and in Appendix~\ref{app:LangevinMainSDPI} we discuss the SDPI approach.

\subsection{Evolution of $\Phi$-Sobolev Constant along Langevin Dynamics}\label{app:PSIEvolutionAlongLangevinDynamics}

\begin{lemma}\label{lem:PSIEvolutionAlongLangevinDynamics}
    Suppose $X_t \sim \rho_t$ evolves according to \eqref{eq:LangevinDynamics} where $\nabla^2 f \succeq \alpha I$ with $\alpha > 0$\,. If $\rho_s$ satisfies $\Phi$-Sobolev inequality with optimal constant $\alpha_{\PSI}(\rho_s)$ where $s \geq 0$\,. Then for $t \geq s$ we have that 
    \begin{equation}\label{eq:PSIEvolutionAlongLangevinDynamics}
    \frac{1}{\alpha_{\PSI}(\rho_t)} \leq \frac{e^{-2\alpha(t-s)}}{\alpha_{\PSI}(\rho_s)} + \frac{1-e^{-2\alpha(t-s)}}{\alpha}\,.
    \end{equation}
\end{lemma}

\begin{proof}
    Consider the forward discretization of \eqref{eq:LangevinDynamics} (i.e. the ULA~\eqref{eq:ULA}) with step-size $\eta > 0$\,:
    \[
    X_{k+1} = X_k - \eta \nabla f (X_k) + \sqrt{2\eta}Z_k\,.
    \]
    We will consider this discrete time update and then take the appropriate limit so that we recover the desired results for the continuous time dynamics. If $X_k \sim \rho_k$\,, then we have that along the ULA,
    \[
    \rho_{k+1} = (I - \eta \nabla f)_{\#}\rho_k * \N(0, 2\eta I)\,.
    \]
    Under $\nabla^2 f \succeq \alpha I$, we can see that the map $F(x) = x - \eta \nabla f(x)$ is $(1-\eta \alpha)$-Lipschitz. Using Lemmas~\ref{lem:PhiSIchangeAlongLipschitzPushforward} and~\ref{lem:PhiSIchangeAlongConvolution}, along with the shorthand $\alpha_i$ for $\alpha_{\PSI}(\rho_i)$ and $c_i$ for $1/\alpha_i$\,, we have that 
    \[
    c_{i+1} \stackrel{\eqref{ineq:convolution}}\le \frac{1}{\alpha_{\PSI}((I - \eta \nabla f)_{\#}\rho_i)} + 2\eta \stackrel{\eqref{ineq:pushforward}}\leq (1-\eta \alpha)^2 c_i + 2 \eta\,.
    \]
    Recursing this from $i=j$ to $i=k$\,, we get that 
    \begin{align*}
    c_{k+1} &\leq (1-\eta\alpha)^{2(k+1-j)}c_j + 2\eta [1 + (1-\eta \alpha)^2 + \dots + (1-\eta \alpha)^{2(k-j)}]\\
    &= (1-\eta\alpha)^{2(k+1-j)}c_j + 2\eta \left[ \frac{1-(1-\eta \alpha)^{2(k+1-j)}}{\eta \alpha(2-\eta \alpha)}  \right]\,.
    \end{align*}
    Taking $\eta k \to t$, $\eta j \to s$, and $\eta \to 0$\,, we get that 
    \[
    c_t \leq e^{-2\alpha(t-s)}c_s + \frac{1-e^{-2\alpha(t-s)}}{\alpha}\,,
    \]
    which yields the desired claim.    
\end{proof}

\subsection{Proofs for Direct Time Derivative Analysis}\label{app:LangevinMainDirect}

\subsubsection{Proof of Lemma~\ref{lem:PhiMutualDeBruijn}}\label{app:PfOfPhiMutualDeBruijn}

\begin{proof}
    Consider Lemma~\ref{lem:SimultaneousSDE} with $\mu_t = \tilde \rho_t$, $\nu_t = \rho_t$, $b_t = -\nabla f$, and $c=1$. Let $\tilde \rho_0 = \delta_{x_0}$ for some $x_0 \in \R^d$, so that $\tilde \rho_t = \rho_{t \mid 0}(\cdot \mid x_0)$\,. It follows from Lemma~\ref{lem:SimultaneousSDE} that
    \[
    \frac{\D}{\D t} \sfD_{\Phi}(\tilde \rho_t \dvert \rho_t) = - \FI_{\Phi}(\tilde \rho_t \dvert \rho_t)\,.
    \]
    It thus follows from Definition~\ref{def:PhiMutualInformation}~and~\eqref{eq:MutualPhiFisherInfo} that
    \begin{align*}
        \frac{\D}{\D t} \MI_\Phi(\rho_{0,t}) &= \frac{\D}{\D t} \E_{\rho_0} \left[ \sfD_{\Phi} (\rho_{t \mid 0} \dvert \rho_t) \right] = \E_{\rho_0} \left[ \frac{\D}{\D t} \sfD_{\Phi} (\rho_{t \mid 0} \dvert \rho_t) \right]\\
        &= - \E_{\rho_0} [\FI_{\Phi}( \rho_{t \mid 0} \dvert \rho_t)] \stackrel{\eqref{eq:MutualPhiFisherInfo}}= -\FI_\Phi^\M(\rho_{0,t}).
    \end{align*}
    We have thus completed the proof.
\end{proof}

\subsubsection{Proof of Lemma~\ref{lem:mutual_PhiSI}}\label{app:PfOfmutual_PhiSI}

\begin{proof}
    As $\rho^Y$ satisfies $\Phi$-Sobolev inequality, it follows from Definition \ref{def:PhiSobolevInequality} that the following holds for each $x \sim \rho^X$,
    \[
    2\alpha_{\PSI}(\rho^Y) \sfD_{\Phi}(\rho^{Y \mid X = x} \dvert \rho^Y) \stackrel{\eqref{eq:PhiSI_DistributionBased}}\leq \FI_{\Phi}(\rho^{Y \mid X = x} \dvert \rho^Y)\,.
    \]
    Taking expectation over $\rho^X$ and using Definition~\ref{def:PhiMutualInformation} and~\eqref{eq:MutualPhiFisherInfo} completes the proof.
\end{proof}

\subsubsection{Proof of Theorem~\ref{thm:LangevinDynamicsMain} (direct time derivative-based proof)}\label{app:DirectPfOfLangevinDynamicsMain}

\begin{proof}
    Fix an $s > 0$ such that $\rho_s$ satisfies a $\Phi$-Sobolev inequality with optimal constant $\alpha_{\PSI}(\rho_s)$. 
    Lemma~\ref{lem:PSIEvolutionAlongLangevinDynamics} tells us that for any $t \geq s$, we have
    \begin{equation}\label{eq:DirectPf_LD_1}
        \alpha_{\PSI}(\rho_t) \geq \frac{\alpha_{\PSI}(\rho_s) \alpha}{e^{-2\alpha (t-s)} (\alpha - \alpha_{\PSI}(\rho_s)) + \alpha_{\PSI}(\rho_s)}
        = \frac{1}{e^{-2\alpha (t-s)} (\alpha_{\PSI}(\rho_s)^{-1} - \alpha^{-1}) + \alpha^{-1}}.
    \end{equation}
    Now from Lemmas \ref{lem:PhiMutualDeBruijn} and \ref{lem:mutual_PhiSI}, we have that,
    \[
        \frac{\D}{\D t} \MI_{\Phi}(\rho_{0,t}) 
        = -\FI_{\Phi}^{\M}(\rho_{0,t}) 
        \le -2\alpha_{\PSI}(\rho_t) \, \MI_{\Phi}(\rho_{0,t})\,. 
    \]
    Integrating the above from $s$ to $t$, and using~\eqref{eq:DirectPf_LD_1} gives us
    \[
        \MI_{\Phi}(\rho_{0,t}) \le \exp\left(-2A_t\right) \MI_{\Phi}(\rho_{0,s})\,,
    \]
    where
    \[
    A_t \coloneqq  \int_s^t \frac{e^{2\alpha(r-s)}}{\alpha_{\PSI}(\rho_s)^{-1} - \alpha^{-1} + \alpha^{-1}e^{2\alpha(r-s)}} \D r\,.
    \]
    Upon simplifying, we get that
    \begin{align*}
        A_t = \frac{1}{2} \log \left(\frac{e^{-2\alpha (t-s)}(\alpha_{\PSI}(\rho_s)^{-1} - \alpha^{-1}) + \alpha^{-1}}{\alpha_{\PSI}(\rho_s)^{-1}}\right) + \alpha (t-s)\,.
    \end{align*}
    Therefore,
    \[
        \MI_{\Phi}(\rho_{0,t}) \le e^{-2A_t} \,\MI_{\Phi}(\rho_{0,s}) 
        = \frac{\alpha e^{-2\alpha (t-s)}}{\alpha_{\PSI}(\rho_s) (1-e^{-2\alpha (t-s)}) + \alpha e^{-2\alpha (t-s)}} \, \MI_{\Phi}(\rho_{0,s})\,.
    \]
    Observe that the denominator is a convex combination of $\alpha$ and $\alpha_{\PSI}(\rho_s)$\,, and so we get that
    \[
    \frac{\alpha }{\alpha_{\PSI}(\rho_s) (1-e^{-2\alpha (t-s)}) + \alpha e^{-2\alpha (t-s)}} \leq \max \left\{ 1, \frac{\alpha}{\alpha_{\PSI}(\rho_s)} \right\}\,,
    \]
    which proves the desired result.
\end{proof}

\subsection{Proofs for SDPI Analysis for Langevin Dynamics}\label{app:LangevinMainSDPI}

The key idea to apply the SDPI-based approach is to bound the contraction coefficient for the dynamics.
We do so in Lemma~\ref{lem:ContractionCoefficientForLangevinDynamics}.

We let $\bP_t$ denote the map which takes as input a distribution $\mu$ and outputs $\rho_t$ where $\rho_t$ evolves following~\eqref{eq:LangevinDynamics} from $\rho_0 = \mu$.

\subsubsection{Proof of Contraction Coefficient}

\begin{lemma}\label{lem:ContractionCoefficientForLangevinDynamics}
    Let $\bP_t$ denote the Markov kernel corresponding to the Langevin dynamics~\eqref{eq:LangevinDynamics} where $\nabla^2 f \succeq \alpha I$ with $\alpha > 0$\,. Then, if $\rho$ satisfies a $\Phi$-Sobolev inequality with optimal constant $\alpha_{\PSI}(\rho)$,
    \begin{equation*}
        \varepsilon_{\sfD_{\Phi}}(\bP_t, \rho) \leq \frac{\alpha e^{-2\alpha t}}{\alpha_{\PSI}(\rho) (1-e^{-2\alpha t}) + \alpha e^{-2\alpha t}}\,.
    \end{equation*}
\end{lemma}

\begin{proof}
    Let $\pi$ be an arbitrary distribution such that $\sfD_{\Phi}(\pi \dvert \rho) < \infty$\,. And let $\pi_t = \pi \bP_t$ and $\rho_t = \rho \bP_t$\,. Therefore, $\pi_0 = \pi$ and $\rho_0 = \rho$\,. 
    Recall from Lemma \ref{lem:PSIEvolutionAlongLangevinDynamics} that $\rho_t$ satisfies a $\Phi$-Sobolev inequality with optimal constant $\alpha_{\PSI}(\rho_t)$ as in \eqref{eq:PSIEvolutionAlongLangevinDynamics}.
    From Lemma~\ref{lem:SimultaneousSDE} with $\mu_t = \pi_t$, $\nu_t = \rho_t$, $b_t = -\nabla f$, and $c=1$, along with Definition~\ref{def:PhiSobolevInequality}, we have that
    \[
    \frac{\D}{\D t} \sfD_\Phi (\pi_t \dvert \rho_t) = - \FI_\Phi(\pi_t \dvert \rho_t)    \stackrel{\eqref{eq:PhiSI_DistributionBased}}\leq -2 \, \alpha_{\PSI}(\rho_t) \sfD_{\Phi}(\pi_t \dvert \rho_t)\,.
    \]
    Plugging \eqref{eq:PSIEvolutionAlongLangevinDynamics} into the above inequality, we have
    \begin{align*}
        \frac{\D}{\D t} \sfD_\Phi (\pi_t \dvert \rho_t) &\leq  \frac{-2\alpha_{\PSI}(\rho_s) \alpha}{e^{-2\alpha(t-s)} \alpha + [1-e^{-2\alpha(t-s)}] \alpha_{\PSI}(\rho_s)} \sfD_{\Phi}(\pi_t \dvert \rho_t) \\ 
        &= \frac{-2}{e^{-2\alpha s} (\alpha_{\PSI}(\rho_0)^{-1} - \alpha^{-1}) + \alpha^{-1}}  \sfD_{\Phi}(\pi_t \dvert \rho_t) 
    \end{align*}
    Applying Grönwall's inequality gives
    \[
    \sfD_{\Phi}(\pi_t \dvert \rho_t) \leq \exp \left(  -2  \int_0^t \frac{1}{e^{-2\alpha s} (\alpha_{\PSI}(\rho_0)^{-1} - \alpha^{-1}) + \alpha^{-1}} \D s \right) \sfD_{\Phi}(\pi_0 \dvert \rho_0)\,,
    \]
    which simplifies to,
    \[
    \sfD_\Phi (\pi_t \dvert \rho_t) \leq  \frac{\alpha e^{-2\alpha t}}{\alpha_{\PSI}(\rho_0) (1-e^{-2\alpha t}) + \alpha e^{-2\alpha t}}  \sfD_{\Phi}(\pi_0 \dvert \rho_0).
    \]
    Therefore, the desired bound immediately follows from the above inequality and Definition \ref{def:ContractionCoefficient_PhiDivergence}.
\end{proof}

\subsubsection{Proof of Theorem~\ref{thm:LangevinDynamicsMain} (SDPI-based proof)}\label{app:SDPIPfOfLangevinDynamicsMain}

\begin{proof}
    Let $\bP_t$ denote the Markov kernel for the Langevin dynamics~\eqref{eq:LangevinDynamics}. From Definition~\ref{def:SDPI_PhiMutualInformation} and Lemma~\ref{lem:RelatePhiDivergenceAndPhiMICoefficients}, we have that,
    \[
    \frac{\MI_{\Phi}(\rho_{0,t})}{\MI_{\Phi}(\rho_{0,s})} \stackrel{\eqref{eq:SDPI_PhiMutualInformation}}\leq \varepsilon_{\MI_{\Phi}}(\bP_{t-s}, \rho_s) \leq \varepsilon_{\sfD_{\Phi}}(\bP_{t-s}, \rho_s)\,.
    \]
    Now using Lemma~\ref{lem:ContractionCoefficientForLangevinDynamics} we can upper bound this by
    \[
    \frac{\alpha e^{-2\alpha (t-s)}}{\alpha_{\PSI}(\rho_s) (1-e^{-2\alpha (t-s)}) + \alpha e^{-2\alpha (t-s)}}\,.
    \]
    Observe that the denominator is a convex combination of $\alpha$ and $\alpha_{\PSI}(\rho_s)$\,, and so we get that
    \[
    \frac{\alpha }{\alpha_{\PSI}(\rho_s) (1-e^{-2\alpha (t-s)}) + \alpha e^{-2\alpha (t-s)}} \leq \max \left\{ 1, \frac{\alpha}{\alpha_{\PSI}(\rho_s)} \right\}\,,
    \]
    which proves the desired result.
\end{proof}

\section{Proofs for ULA}

\subsection{Proof of Lemma~\ref{lem:ULAContractionCoefficientAndPSIEvolution}}\label{app:PfOfULAContractionCoefficientAndPSIEvolution}

\begin{proof}
\textbf{Part (a):~}
Fix any $\mu$ such that $\sfD_{\Phi}(\mu \dvert \rho) < \infty$\,.
Denote 
\[
F_{\#}\rho * \N(0, tI) = \rho_t, \quad F_{\#}\mu * \N(0, tI) = \mu_t,
\]
where $F(x) = x - \eta \nabla f (x)$.
It follows from Lemma~\ref{lem:PhiSIchangeAlongLipschitzPushforward} that
\begin{equation}\label{ineq:push}
    \alpha_{\PSI}(\rho_0) \geq \frac{\alpha_{\PSI}(\rho)}{\gamma^2}
\end{equation}
where $\rho_0 = F_{\#}\rho$.
Furthermore, it follows from Lemma~\ref{lem:PhiSIchangeAlongConvolution} that
\begin{equation}\label{ineq:conv}
    \alpha_{\PSI}(\rho_t) \geq \frac{\alpha_{\PSI}(\rho_0)}{1 + \alpha_{\PSI}(\rho_0)t}\,.
\end{equation}
Now using Lemma~\ref{lem:SimultaneousSDE} with $\mu_t = \mu_t$, $\nu_t = \rho_t$, $b_t \equiv 0$, and $c = \frac{1}{2}$, along with Definition~\ref{def:PhiSobolevInequality}, we have the following
\[
\frac{\D}{\D t}\sfD_{\Phi}(\mu_t \dvert \rho_t) = -\frac{1}{2} \FI_{\Phi}(\mu_t \dvert \rho_t) \leq -\alpha_{\PSI}(\rho_t) \sfD_{\Phi}(\mu_t \dvert \rho_t).
\]
Integrating this from $t = 0$ to $t = 2\eta$\,, and using \eqref{ineq:conv} and \eqref{ineq:push}, we get that
\begin{equation}\label{ineq:contraction}
    \frac{\sfD_{\Phi}(\mu_{2\eta} \dvert \rho_{2\eta})}{\sfD_{\Phi}(\mu_0 \dvert \rho_0)} \leq \exp \left( -\int_0^{2\eta} \alpha_{\PSI}(\rho_t) \D t \right) \stackrel{\eqref{ineq:conv}}\le \frac{1}{1+2\eta\,\alpha_{\PSI}(\rho_0)} \stackrel{\eqref{ineq:push}}\leq \frac{\gamma^2}{\gamma^2 + 2\eta\,\alpha_{\PSI}(\rho)}\,.
\end{equation}
Note that $\sfD_{\Phi}(\mu \dvert \rho) = \sfD_{\Phi}(\mu_0 \dvert \rho_0)$ as $\Phi$-divergence is invariant to applying a bijective map to both arguments.
It thus follows that 
\[
\frac{\sfD_{\Phi}(\mu_{2\eta} \dvert \rho_{2\eta})}{\sfD_{\Phi}(\mu_0 \dvert \rho_0)} = \frac{\sfD_{\Phi}(\mu_0 * \N(0, 2\eta I) \dvert \rho_0 * \N(0, 2\eta I))}{\sfD_{\Phi}(\mu \dvert \rho)}
= \frac{\sfD_{\Phi}(\mu \bP \dvert \rho \bP)}{\sfD_{\Phi}(\mu \dvert \rho)}.
\]
Finally, the statement follows from the above observation, Definition \ref{def:ContractionCoefficient_PhiDivergence}, and \eqref{ineq:contraction}.

\medskip
\noindent
\textbf{Part (b):~}
Recall the ULA update $\rho \bP = F_{\#}\rho * \N(0, 2\eta I)$ mentioned in~\eqref{eq:ULA_TwoStepUpdate} and note that $\N(0, 2\eta I)$ is $\frac{1}{2\eta}$-strongly log-concave, and therefore satisfies $\Phi$-Sobolev inequality with the same constant (Lemma~\ref{lem:PhiSISLC}). The result then follows from Lemmas~\ref{lem:PhiSIchangeAlongLipschitzPushforward} and \ref{lem:PhiSIchangeAlongConvolution}.

\end{proof}

\subsection{Proof of Theorem~\ref{thm:ULA_Main}}\label{app:PfOfThmULAMain}

\begin{proof}
    Recall the ULA update $\rho \bP = F_{\#}\rho * \N(0, 2\eta I)$ mentioned in~\eqref{eq:ULA_TwoStepUpdate}. Denote the Lipschitz constant of $F$ by $\gamma >0$.
    We begin by ensuring that $F$ satisfies the conditions required by Lemma~\ref{lem:ULAContractionCoefficientAndPSIEvolution}.
    The $\alpha$-strong log-concavity and $L$-smoothness of $\nu$ implies $(1-\eta L)I \preceq \nabla F \preceq (1-\eta \alpha) I$ where note that $\nabla F (x) = I - \eta \nabla^2 f(x)$\,. 
    Therefore, $F$ is $\gamma = (1-\eta\alpha)$-Lipschitz.
    For Lemma~\ref{lem:ULAContractionCoefficientAndPSIEvolution}(a), it remains to check that $F$ is bijective. The continuity of $F$ ensures that it is surjective. For injectivity, note that,
    $
        F(x) = F (y) \implies x-y = \eta [\nabla f (x) - \nabla f (y)]\,.
    $
    Taking norm on both sides and using that $\nabla f$ is $L$-Lipschitz, along with $\eta \leq 1/L$, implies that $F$ is injective. Therefore $F$ meets all of the requirements set forth by Lemma~\ref{lem:ULAContractionCoefficientAndPSIEvolution}.

    Fix an $\ell \geq 1$ such that $\rho_\ell$ satisfies a $\Phi$-Sobolev inequality. This is guaranteed by assumption.
    Then Lemma~\ref{lem:IndependenceContractionBySDPI} states that
    \begin{equation}\label{eq:ProductOfContractionCoefficients}
        \MI_{\Phi}(\rho_{0,k}) \leq  \prod_{i=\ell}^{k-1} \varepsilon_i \, \MI_{\Phi}(\rho_{0,\ell}), \quad \varepsilon_i = \varepsilon_{\sfD_{\Phi}}(\bP, \rho_i).
    \end{equation}
    As $\rho_\ell$ satisfies a $\Phi$-Sobolev inequality, Lemma~\ref{lem:ULAContractionCoefficientAndPSIEvolution}(b) guarantees that $\rho_j$ satisfies a $\Phi$-Sobolev inequality too for all $j \geq \ell +1$\,.
    Denote $\alpha_{\PSI}(\rho_i)$ by $\alpha_i$, then the claim of Lemma~\ref{lem:ULAContractionCoefficientAndPSIEvolution}(b) can be rewritten as 
    \begin{equation}\label{ineq:DiscreteTimeIdentityGaussianPSIEvolutionRewritten}
        1 + \frac{ 2 \alpha_i \eta}{\gamma^2} \stackrel{\eqref{eq:ULAPSIEvolution}}\geq \frac{\alpha_i}{\gamma^2 \alpha_{i+1}}\,.
    \end{equation}
    It thus follows from Lemma~\ref{lem:ULAContractionCoefficientAndPSIEvolution}(a) that
    \[
        \varepsilon_i \leq \left(  1 + \frac{2 \alpha_i \eta}{\gamma^2} \right)^{-1} \stackrel{\eqref{ineq:DiscreteTimeIdentityGaussianPSIEvolutionRewritten}}\leq \frac{\gamma^2 \alpha_{i+1}}{\alpha_i}\,.
    \]
    Plugging this bound into \eqref{eq:ProductOfContractionCoefficients}, we obtain
    \begin{equation}\label{eq:Ek_BoundInDiscreteTimeIdentityGaussianMainTheoremPf}
    \MI_{\Phi}(\rho_{0,k}) \leq \frac{\gamma^{2(k-\ell)}\alpha_k}{\alpha_\ell} \MI_{\Phi}(\rho_{0,\ell})\,.
    \end{equation}
    
    Our goal now is to derive a simple upper bound on $\alpha_k/\alpha_\ell$\,. To that end, further denote $c_i = 1/\alpha_i$\,. Then, \eqref{ineq:DiscreteTimeIdentityGaussianPSIEvolutionRewritten} is equivalent to $c_{i+1} \leq \gamma^2 c_i + 2\eta$\,.
    Subtracting $2\eta/(1-\gamma^2)$ from both sides and applying the resulting inequality repeatedly, we have
    \[
    c_k - \frac{2\eta}{1-\gamma^2} \leq \gamma^{2(k-\ell)} \left(c_\ell - \frac{2\eta}{1-\gamma^2}\right)\,.
    \]
    Recalling $\alpha_k = 1/c_k$ yields
    \begin{equation}\label{eq:ULA_LSI_recursion}
    \frac{1}{\alpha_k} = \frac{1-\gamma^{2(k-\ell)}}{\alpha^*} + \frac{\gamma^{2(k-\ell)}}{\alpha_\ell} \ge \min\left\{\frac{1}{\alpha^*}, \frac{1}{\alpha_\ell}\right\} \ge \min\left\{\frac{1}{\alpha}, \frac{1}{\alpha_\ell}\right\},
    \end{equation}
    where $\alpha^* = \alpha ( 1- \eta \alpha/2)\leq \alpha$.
    Therefore, we obtain
    \[
    \frac{\alpha_k}{\alpha_\ell} \le \max\left\{1, \frac{\alpha}{\alpha_\ell}\right\},
    \]
    and hence the claim of the theorem immediately follows from \eqref{eq:Ek_BoundInDiscreteTimeIdentityGaussianMainTheoremPf}.
\end{proof}

\subsection{Proof of Corollary~\ref{cor:ULA}}\label{app:ULAIterationComplexity}

\begin{proof}
    Note that $\tau \le e^{\tau-1}$ for every $\tau\in \R$. Using Theorem \ref{thm:ULA_Main} and this observation with $\tau=1-\alpha \eta$,
    we have
    \[
    \frac{\MI_{\Phi}(\rho_{0,k})}{\max \left\{1, \frac{\alpha}{\alpha_{\PSI}(\rho_\ell)} \right\} \MI_{\Phi}(\rho_{0,\ell})} \stackrel{\eqref{eq:ULA_MainTheorem}}\leq (1-\alpha\eta)^{2(k-\ell)} \le e^{-2\alpha \eta(k-\ell)}.
    \]
    In view of the above inequality, to bound $\MI_{\Phi}(X_0; X_k) \leq \epsilon$, it suffices to bound
    \[
    e^{-2\alpha \eta(k-\ell)} \le \frac{\epsilon}{\max \left\{1, \frac{\alpha}{\alpha_{\PSI}(\rho_\ell)} \right\} \MI_{\Phi}(\rho_{0,\ell})},
    \]
    which gives the desired complexity bound on $k$.
\end{proof}

\section{Proofs for Proximal Sampler}\label{app:ProximalSampler}

In this section, our goal is to prove Theorem~\ref{thm:ProximalMain}.
As mentioned in Section~\ref{sec:MainSectionProximal_SDPIProofSketch}, the SDPI-based proof for the Proximal Sampler proceeds by studying the contraction coefficient and evolution of the $\Phi$-Sobolev constant along the forward step and backward step separately, and then combining them to obtain the contraction coefficient and $\Phi$-Sobolev constant evolution for the entire proximal sampler (Lemma~\ref{lem:proximal_ContractionCoefficientBoundAndPSIEvolution}).

We denote the Proximal Sampler~\eqref{eqs:ProximalSampler} as $\bP_{\prox} = \bP_{\prox}^+ \bP_{\prox}^-$ where $\bP_{\prox}^+$ and $\bP_{\prox}^-$ correspond to the forward and backward steps, respectively.
As each step is an update of probability distributions on $\R^d$, this composition is justified.
Regarding notation, we denote $\rho_k^X \coloneqq \law (X_k)$, $\rho_k^Y \coloneqq \law (Y_k)$, and therefore $\rho^X_k \, \bP_{\prox} = \rho^X_{k+1}$\,, $\rho^X_k\,  \bP_{\prox}^+ = \rho^Y_{k}$\,, and $\rho_k^Y \bP_\prox^- = \rho_{k+1}^X$.

\medskip
\noindent
We repeat the SDE interpretations of both the forward and backward steps mentioned in Section~\ref{sec:ProximalMain}.

\paragraph{Forward Step:}
Suppose we start from $X_0 \sim \rho_0^X$. 
Along the forward step of the Proximal Sampler~\eqref{eqs:ProximalSampler}, $Y_0 \mid X_0 \sim \N(X_0, \eta I)$, so in particular, $\rho_0^Y = \rho_0^X * \N(0, \eta I)$.
Therefore, the action of the forward step $\bP^+_{\prox}$ is via a Gaussian convolution: 
$\rho \bP^+_{\prox} = \rho *\N(0, \eta I)$, which can be interpreted as the solution to the heat flow (generated by the Brownian motion SDE $\D X_t = \D W_t$) at time $\eta > 0$.

\paragraph{Backward Step:}
For the backward step $\bP^-_{\prox}$, it will be helpful to think of the corresponding SDE as the time reversal of the forward step SDE, i.e. the time reversal of the heat flow, which is known as the \emph{backward heat flow}; see~\citep{chen2022improved} and~\citep[Chapter~8.3]{chewi2023log}.

Fixing a step-size $\eta >0$, we have along the forward step ($\D X_t = \D W_t$) that if $X_0 \sim \nu^X$, then $X_{\eta} \sim \nu^Y$. The backward heat flow SDE~\eqref{eq:BackwardHeatFlow} is defined by
\begin{equation*}
    \D Y_t = \nabla \log{(\nu^X * \N_{\eta -t})}(Y_t) \D t + \D W_t\,.
\end{equation*}
By construction, if we start the SDE~\eqref{eq:BackwardHeatFlow} from $Y_0 \sim \mu_0 = \nu^Y$, then for any $t \in [0,\eta]$, the distribution of $Y_t \sim \mu_t$ along~\eqref{eq:BackwardHeatFlow} is given by $\mu_t = \nu^X * \N_{\eta-t}$, and in particular, at time $t = \eta$, $Y_\eta \sim \mu_\eta = \nu^X$.
For the Proximal Sampler, we start the backward SDE~\eqref{eq:BackwardHeatFlow} from $Y_0 \sim \rho_0^Y$, to obtain $X_1 := Y_\eta \sim \rho_1^X$.

\medskip
We present the proofs corresponding to the forward step in Appendix~\ref{app:ProximalForwardStep}, the backward step in Appendix~\ref{app:ProximalBackwardStep}, and the complete Proximal Sampler in Appendix~\ref{app:ProximalCompleteSampler}.

\subsection{Proofs for the Forward Step of the Proximal Sampler}\label{app:ProximalForwardStep}

\subsubsection{Contraction Coefficient for Forward Step}

\begin{lemma}\label{lem:ContractionCoefficientBound_FwdHeatFlow}
    Let $\bP_t$ denote the Markov kernel or heat semigroup corresponding to evolution along $\D X_t = \D W_t$ for time $t$\,. If $\rho$ satisfies a $\Phi$-Sobolev inequality with optimal constant~$\alpha_{\PSI}(\rho)$\,, then
    \[
    \varepsilon_{\sfD_{\Phi}}(\bP_\eta, \rho) \leq \frac{1}{1 + \eta \alpha_{\PSI}(\rho)}\,.
    \]
\end{lemma}

\begin{proof}
    Let $\pi$ be an arbitrary distribution such that $\sfD_{\Phi}(\pi \dvert \rho) < \infty$\,. And let $\pi_t = \pi \bP_t$ and $\rho_t = \rho \bP_t$\,. Therefore, $\pi_0 = \pi$ and $\rho_0 = \rho$\,. It follows from Definition~\ref{def:PhiSobolevInequality} and Lemma~\ref{lem:SimultaneousSDE} with $\mu_t = \pi_t$, $\nu_t = \rho_t$, $b_t \equiv 0$, and $c = \frac{1}{2}$ that
    \begin{equation}\label{ineq:LSI}
        \frac{\D}{\D t} \sfD_\Phi (\pi_t \dvert \rho_t) \stackrel{\eqref{eq:GeneralSDE}}= -\frac{1}{2} \, \FI_\Phi(\pi_t \dvert \rho_t)
    \leq - \, \alpha_{\PSI}(\rho_t) \sfD_{\Phi}(\pi_t \dvert \rho_t)\,.
    \end{equation}
    Since the forward step is the heat flow, i.e., $\rho_\eta = \rho \bP^+_{\prox} = \rho *\N(0, \eta I)$, it follows from Lemma \ref{lem:PhiSIchangeAlongConvolution} that
    \[
    \frac{1}{\alpha_{\PSI}(\rho_t)} \leq \frac{1}{\alpha_{\PSI}(\rho_0)} + t\,.
    \]
    Plugging the above inequality into \eqref{ineq:LSI} yields that
    \[
    \frac{\D}{\D t} \sfD_\Phi (\pi_t \dvert \rho_t) 
    \leq \frac{- \alpha_{\PSI}(\rho_0)}{1+t \alpha_{\PSI}(\rho_0)} \sfD_{\Phi}(\pi_t \dvert \rho_t).
    \]
    Applying Grönwall's inequality, we have
    \[
    \sfD_{\Phi}(\pi_\eta \dvert \rho_\eta) \leq \exp \left(  - \alpha_{\PSI}(\rho_0) \int_0^{\eta} \frac{1}{1 + t \alpha_{\PSI}(\rho_0)} \D t \right) \sfD_{\Phi}(\pi_0 \dvert \rho_0),
    \]
    which simplifies to
    \[
    \sfD_\Phi (\pi_\eta \dvert \rho_\eta) \leq \frac{\sfD_{\Phi}(\pi_0 \dvert \rho_0)}{1 +  \eta \alpha_{\PSI}(\rho_0)}.
    \]
    The conclusion of the lemma immediately follows from the above inequality and Definition \ref{def:ContractionCoefficient_PhiDivergence}.
\end{proof}

\subsection{Proofs for the Backward Step of the Proximal Sampler}\label{app:ProximalBackwardStep}

\subsubsection{Contraction Coefficient for Backward Step}

\begin{lemma}\label{lem:ContractionCoefficientBound_BackHeatFlow}
    Let $\nu_0$ be $\alpha$-SLC for some $\alpha > 0$ and define $\nu_t = \nu_0 * \N(0, tI)$. Fix some positive constant $T > 0$. For $t \in [0, T]$\,, let $\bP_t^T$ denote evolution along the following SDE for time~$t$
    \[
    \D X_t = \nabla \log \nu_{T-t}(X_t) \D t + \D W_t\,.
    \]
    If $\rho$ satisfies a $\Phi$-Sobolev inequality with optimal constant $\alpha_{\PSI}(\rho)$\,,
    then,
    \begin{equation}\label{ineq:minus}
        \varepsilon_{\sfD_{\Phi}}(\bP_t^T, \rho) \leq \frac{1 + \alpha T - \alpha t}{(1+\alpha T)[1 + \alpha_{\PSI}(\rho) t] - \alpha t}\,.
    \end{equation}
\end{lemma}

\begin{proof}
    Let $\pi$ be an arbitrary distribution such that $\sfD_{\Phi}(\pi \dvert \rho) < \infty$\,. And let $\pi_t = \pi \bP_t^T$ and $\rho_t = \rho \bP_t^T$\,. Therefore, $\pi_0 = \pi$ and $\rho_0 = \rho$\,. From Lemma~\ref{lem:SimultaneousSDE} with $\mu_t = \pi_t$, $\nu_t = \rho_t$, $b_t = \nabla \log{\nu_{T-t}}$, and $c = \frac{1}{2}$, and Definition~\ref{def:PhiSobolevInequality}, we have that
    \begin{equation}\label{ineq:ddt}
        \frac{\D}{\D t} \sfD_\Phi (\pi_t \dvert \rho_t) = -\frac{1}{2} \, \FI_\Phi(\pi_t \dvert \rho_t)
    \stackrel{\eqref{eq:PhiSI_DistributionBased}}\leq - \, \alpha_{\PSI}(\rho_t) \sfD_{\Phi}(\pi_t \dvert \rho_t)\,.    
    \end{equation}
    Denote $\alpha_{\PSI}(\rho_t)$ as $\alpha_t$ for $t \in [0,T]$. Then Lemma~\ref{lem:PhiSI_Evolution_BackwardHeatFlow} tells us,
    \[
    \alpha_t \geq \frac{\alpha_0 (1 + \alpha T)^2}{(1 + \alpha (T-t))^2 + \alpha_0 t [1 + \alpha (T-t)](1+\alpha T)}\,.
    \]
    To apply Grönwall's inequality, we need to evaluate the following
    \begin{align*}
    A_t &= \int_0^t \frac{\alpha_0 (1 + \alpha T)^2}{(1 + \alpha (T-s))^2 + \alpha_0 s (1 + \alpha (T-s))(1+\alpha T)} \D s\\
    &= \bigintsss_0^t \frac{\alpha_0 (1 + \alpha T)^2}{\alpha (\alpha - \alpha_0 (1+\alpha T)) \left( s - \frac{1 + \alpha T}{\alpha - \alpha_0(1+\alpha T)}  \right) \left( s - \frac{1 + \alpha T}{\alpha}   \right)} \D s\\
    &= \bigintsss_0^t \frac{1}{ s - \frac{1 + \alpha T}{\alpha - \alpha_0(1+\alpha T)} }  \D s - \bigintsss_0^t \frac{1}{ s - \frac{1 + \alpha T}{\alpha} } \D s \\
    &= \log \frac{1 + \alpha T - t (\alpha - \alpha_0(1 + \alpha T))}{1 + \alpha T} - \log \frac{1 + \alpha T - \alpha t}{1 + \alpha T} \\
    &= \log  \frac{(1+\alpha T)(1 + \alpha_0 t) - \alpha t}{1 + \alpha T - \alpha t}.
    \end{align*}
    Applying Grönwall's inequality to \eqref{ineq:ddt} gives $\sfD_{\Phi}(\pi_t \dvert \rho_t) \leq e^{-A_t} \sfD_{\Phi}(\pi_0 \dvert \rho_0)$\,, i.e.,
    \[
    \sfD_{\Phi}(\pi_t \dvert \rho_t) \leq \frac{1 + \alpha T - \alpha t}{(1+\alpha T)(1 + \alpha_0 t) - \alpha t} \sfD_{\Phi}(\pi_0 \dvert \rho_0)\,.
    \]
    The result follows by using Definition \ref{def:ContractionCoefficient_PhiDivergence} and noting that the choice of $\pi$ is arbitrary.
\end{proof}

\subsubsection{Evolution of $\Phi$-Sobolev constant along Backward Step}

\begin{lemma}\label{lem:PhiSI_Evolution_BackwardHeatFlow}
    Let $\nu_0$ be $\alpha$-SLC for some $\alpha > 0$,
    and define $\nu_t = \nu_0 * \N(0, tI)$. Fix some $T > 0$. For $t \in [0, T]$\,, consider $X_t \sim \rho_t$ which evolves according to
    \[
    \D X_t = \nabla \log \nu_{T-t}(X_t) \D t + \D W_t\,,
    \]
     from $X_0 \sim \rho_0$ where $\rho_0$ satisfies a $\Phi$-Sobolev inequality with optimal constant $\alpha_{\PSI}(\rho_0)$\,.
    Then we have that,
    \begin{equation}\label{ineq:back}
        \frac{1}{\alpha_{\PSI}(\rho_t)} \leq \frac{1}{\alpha_{\PSI}(\rho_0)} \left(  1- \frac{\alpha t}{1+\alpha T}  \right)^2 + \frac{t(1 + \alpha (T-t))}{1 + \alpha T}\,.
    \end{equation}
\end{lemma}

\begin{proof}
    As $\nu_0$ is $\alpha$-SLC and $\nu_t = \nu_0 * \N(0, tI)$, 
    it follows from \citep[Theorem 3.7(b)]{saumard2014log} that $\nu_t$ is $\alpha_t$-SLC where $\alpha_t = \alpha/(1 + t\alpha)$\,.
    Writing $\nu_t \propto \exp{(-f_t)}$, the backward heat flow SDE can be rewritten as
\[
\D X_t = -\nabla f_{T-t}(X_t) \D t + \D W_t\,.
\] 
Now consider a discretization of this SDE with a sufficiently small step-size $\eta  > 0$ to get
\[
X_{k+1} = X_k - \eta \nabla f_{T - \eta k}(X_k) + \sqrt{\eta} Z_k\,,
\]
where $Z_k \sim \N(0, I)$\,.
We therefore know that $f_{T-\eta k}$ is $\beta_k$-strongly convex where 
\begin{equation}\label{eq:beta}
    \beta_k = \frac{\alpha}{1 + \alpha(T - \eta k)}.
\end{equation}
Letting $X_k \sim \rho_k$, we therefore have that 
\begin{equation}\label{eq:evolution}
    \rho_{k+1} = (I - \eta \nabla f_{T - \eta k})_{\#} \rho_k * \N(0, \eta I),
\end{equation}
and the mapping $I - \eta \nabla f_{T - \eta k}$ is $(1-\eta \beta_k)$-Lipschitz continuous.
Denote $c_k = 1/\alpha_{\PSI}(\rho_k)$\,. 
In view of \eqref{eq:evolution},
using Lemmas~\ref{lem:PhiSIchangeAlongLipschitzPushforward} and~\ref{lem:PhiSIchangeAlongConvolution}, we have the following recursion
    \[
    c_{k+1} \stackrel{\eqref{ineq:convolution}}\le \frac{1}{\alpha_{\PSI}((I - \eta \nabla f_{T-\eta k})_{\#}\rho_k)} + \eta \stackrel{\eqref{ineq:pushforward}}\leq (1-\eta \beta_k)^2 c_k + \eta\,.
    \]
Upon further iteration this yields 
\[
c_k \leq \prod_{i=0}^k (1-\eta \beta_i)^2 c_0 + \eta \left[ 1 + \sum_{j=1}^{k-1} \prod_{i=j}^{k-1} (1-\eta \beta_i)^2 \right]\,.
\]
Plugging in the expression for $\beta_i$ in \eqref{eq:beta} and simplifying, we get
\[
c_k \leq \left( 1 - \frac{\alpha \eta k}{1 + \alpha T} \right)^2 c_0 + \eta \left[  1 + \sum_{j=1}^{k-1} \left(  \frac{1 + \alpha (T - \eta k)}{1 + \alpha (T-\eta j)}  \right)^2   \right]\,.
\]
Now taking the limit $\eta \to 0$\,, $\eta k \to t$\,, and $\eta j \to s$\,, we get 
\[
c_t \leq \left( 1 - \frac{\alpha t}{1 + \alpha T} \right)^2 c_0 + (1 + \alpha (T-t))^2 \int_0^t \frac{\D s}{(1 + \alpha(T-s))^2}\,,
\]
which simplifies to
\[
c_t \leq \left( 1 - \frac{\alpha t}{1 + \alpha T} \right)^2 c_0 + \frac{t(1 + \alpha (T-t))}{1 + \alpha T}\,.
\]
The desired claim \eqref{ineq:back} now follows from the fact that $c_t = 1/\alpha_{\PSI}(\rho_t)$.
\end{proof}

\subsection{Proofs for the Complete Proximal Sampler}\label{app:ProximalCompleteSampler}

\subsubsection{Proof of Lemma~\ref{lem:proximal_ContractionCoefficientBoundAndPSIEvolution}}\label{app:PfOfproximal_ContractionCoefficientBoundAndPSIEvolution}

\begin{proof}
\textbf{Part (a):~}
It follows from Definition \ref{def:ContractionCoefficient_PhiDivergence} and the fact that $\bP_{\prox} = \bP^+_{\prox}\bP^-_{\prox}$ that
    \begin{align}
        \varepsilon_{\sfD_{\Phi}}(\bP_{\prox}\,, \rho) &\stackrel{\eqref{eq:PhiDivergenceContractionCoefficient}}\coloneqq \sup_{\pi\,:\, 0<\sfD_{\Phi}(\pi\dvert \rho) < \infty} \frac{\sfD_{\Phi}(\pi\, \bP_\prox \dvert \rho \, \bP_\prox)}{\sfD_{\Phi} (\pi \dvert \rho)} \nn \\
        &= \sup_{\pi\,:\, 0<\sfD_{\Phi}(\pi\dvert \rho) < \infty} \frac{\sfD_{\Phi}(\pi\, \bP_{\prox}^+  \bP_{\prox}^- \dvert \rho \, \bP_{\prox}^+  \bP_{\prox}^-)}{\sfD_{\Phi}(\pi\, \bP_\prox^+ \dvert \rho \, \bP_\prox^+)} \times  \frac{\sfD_{\Phi}(\pi\, \bP_\prox^+ \dvert \rho \, \bP_\prox^+)}{\sfD_{\Phi} (\pi \dvert \rho)} \nn \\
        &\stackrel{\eqref{eq:PhiDivergenceContractionCoefficient}}\leq \varepsilon_{\sfD_{\Phi}}(\bP_{\prox}^+\,, \rho) \,\varepsilon_{\sfD_{\Phi}}(\bP_{\prox}^-\,, \rho\,\bP_{\prox}^+), \label{ineq:decompose}
    \end{align}
    where the inequality is also due to Definition \ref{def:ContractionCoefficient_PhiDivergence}.
    Observe that $\bP_{\prox}^+$ is the forward heat flow operation for time $\eta > 0$, and hence $\rho \bP_{\prox}^+ = \rho * \N(0, \eta I)$. It thus follows from Lemma \ref{lem:PhiSIchangeAlongConvolution} that
    \begin{equation}\label{ineq:forward1}
        \alpha_{\PSI} (\rho \,\bP_{\prox}^+) \stackrel{\eqref{ineq:convolution}}\geq \frac{\alpha_{\PSI} (\rho)}{1 + \eta \alpha_{\PSI} (\rho)},
    \end{equation}
    and from Lemma \ref{lem:ContractionCoefficientBound_FwdHeatFlow} that
    \begin{equation}\label{ineq:plus}
        \varepsilon_{\sfD_{\Phi}}(\bP_{\prox}^+, \rho) \leq \frac{1}{1 + \eta \alpha_{\PSI}(\rho)}.
    \end{equation}
    In view of Lemma \ref{lem:ContractionCoefficientBound_BackHeatFlow}, $\bP_{\prox}^-$ is the backward heat flow operation with $\nu_0 = \nu^X$\,, $T = \eta$, and for time $\eta$\,. 
    Hence, using \eqref{ineq:decompose}, \eqref{ineq:forward1}, \eqref{ineq:plus}, and Lemma~\ref{lem:ContractionCoefficientBound_BackHeatFlow} with $T=t=\eta$, $\bP_t^T=\bP_{\prox}^-$, and $\rho=\rho \bP_{\prox}^+$, we obtain
    \begin{align*}
        \varepsilon_{\sfD_{\Phi}}(\bP_{\prox}\,, \rho) 
    &\stackrel{\eqref{ineq:decompose},\eqref{ineq:plus}}\leq \frac{1}{1 + \eta \alpha_{\PSI}(\rho)} \,\varepsilon_{\sfD_{\Phi}}(\bP_{\prox}^-\,, \rho\,\bP_{\prox}^+) \\
    &\stackrel{\eqref{ineq:minus}}\leq \frac{1}{1 + \eta \alpha_{\PSI}(\rho)} \times \frac{1}{(1 + \alpha \eta)(1 + \alpha_{\PSI}(\rho \bP_{\prox}^+)\eta)- \alpha \eta} \\
    &\stackrel{\eqref{ineq:forward1}}\leq \frac{1}{1 + 2\eta \alpha_{\PSI} (\rho) + \eta^2 \alpha \alpha_{\PSI} (\rho) }.
    \end{align*}
    This completes part (a).

\medskip
\noindent
\textbf{Part (b):~}
Observe that $\bP_{\prox}^+$ is the forward heat flow operation for time $\eta > 0$, and hence $\rho \bP_{\prox}^+ = \rho * \N(0, \eta I)$. It thus follows from Lemma \ref{lem:PhiSIchangeAlongConvolution} that
    \begin{equation}\label{ineq:forward}
        \frac{1}{\alpha_{\PSI}(\rho \bP_{\prox}^+)} \stackrel{\eqref{ineq:convolution}}\leq \frac{1}{\alpha_{\PSI}(\rho)} + \eta\,.
    \end{equation}
    In view of Lemma \ref{lem:PhiSI_Evolution_BackwardHeatFlow}, $\bP_{\prox}^-$ is the backward heat flow operation with $\nu_0 = \nu^X$, $T = \eta$, and for time $\eta$\,. 
    Hence, using \eqref{ineq:forward} and Lemma~\ref{lem:PhiSI_Evolution_BackwardHeatFlow} with $T=t=\eta$ and $\rho_0=\rho \bP_{\prox}^+$, we obtain
    \begin{align*}
        \frac{1}{\alpha_{\PSI}(\rho\, \bP_{\prox})} &\stackrel{\eqref{ineq:back}}\leq \frac{1}{\alpha_{\PSI}(\rho \bP_{\prox}^+)} \frac{1}{(1 + \alpha \eta)^2} + \frac{\eta}{1 + \alpha \eta} \\
        &\stackrel{\eqref{ineq:forward}}\leq \frac{1+\alpha_{\PSI}(\rho) \eta}{\alpha_{\PSI}(\rho)} \frac{1}{(1 + \alpha \eta)^2} + \frac{\eta}{1 + \alpha \eta}.
    \end{align*}
    Therefore, the desired claim is proved.
\end{proof}

\subsubsection{Proof of Theorem~\ref{thm:ProximalMain}}\label{app:PfOfProximalMain}

\begin{proof}
    First, it follows from Lemma \ref{lem:IndependenceContractionBySDPI} for the Proximal Sampler that for any $\ell \geq 1$ and $k \geq \ell$
    \begin{equation}\label{ineq:MI}
        \MI_{\Phi}(\rho^X_{0,k}) \leq  \prod_{i=\ell}^{k-1} \varepsilon_i \, \MI_{\Phi}(\rho^X_{0,\ell}), \quad \varepsilon_i = \varepsilon_{\sfD_{\Phi}}(\bP_{\prox}, \rho_i^X).
    \end{equation}
    Throughout the proof, let $\alpha_i$ denote $\alpha_{\PSI}(\rho^X_i)$.
    Recall from Lemma~\ref{lem:proximal_ContractionCoefficientBoundAndPSIEvolution}(b) that
\begin{equation}\label{ineq:alpha}
    \frac{1}{\alpha_{i+1}} \leq \frac{1+\alpha_i \eta}{\alpha_i(1 + \alpha \eta)^2} + \frac{\eta}{1 + \alpha \eta}\,.
\end{equation}

Fix $\ell$ such that $\rho_\ell$ satisfies a $\Phi$-Sobolev inequality. This holds by assumption.
    
    \paragraph{\underline{Case 1: $\alpha \leq \alpha_\ell$}}\mbox{}\newline
Using $\alpha_\ell \geq \alpha$ and \eqref{ineq:alpha} with $i=\ell$, we have
\[
\frac{1}{\alpha_{\ell +1}} \stackrel{\eqref{ineq:alpha}}\leq \frac{1+\alpha_\ell \eta}{\alpha_\ell(1 + \alpha \eta)^2} + \frac{\eta}{1 + \alpha \eta} \leq \frac{1}{\alpha(1 + \alpha \eta)} + \frac{\eta}{1 + \alpha \eta} = \frac{1}{\alpha},
\]
which implies that $\alpha_{\ell +1 } \geq \alpha$\,. Repeating this argument on \eqref{ineq:alpha}, we have that $\alpha_i \geq \alpha$ for all $i \geq \ell$\,. The contraction coefficient bound in Lemma~\ref{lem:proximal_ContractionCoefficientBoundAndPSIEvolution}(a) therefore simplifies to 
\[
\varepsilon_i \leq \frac{1}{1 + 2\eta \alpha_i + \eta^2 \alpha \alpha_i } \leq \frac{1}{ (1 + \eta \alpha)^2 }\,,
\]
for all $i \geq \ell$\,. 
Plugging this inequality into \eqref{ineq:MI}, we obtain
\begin{equation}\label{ineq:case1}
    \MI_{\Phi}(\rho^X_{0,k}) \leq \frac{\MI_{\Phi}(\rho^X_{0,\ell})}{(1 + \eta \alpha)^{2(k-\ell)}}\,.
\end{equation}

\paragraph{\underline{Case 2: $\alpha > \alpha_\ell$}}\mbox{}\newline
Using $\alpha > \alpha_\ell$ and \eqref{ineq:alpha}, we have
\begin{equation}\label{ineq:recur}
    \frac{1}{\alpha_{i+1}} \leq \frac{1+\alpha_i \eta}{\alpha_i(1 + \alpha_\ell \eta)^2} + \frac{\eta}{1 + \alpha_\ell \eta}\,,
\end{equation}
for all $i \geq \ell$\,.
We next prove by induction that $\alpha_i \geq \alpha_\ell$ for all $i \geq \ell$. This is of course true when $i=\ell$. Assume for some $i\ge \ell$ that $\alpha_i \geq \alpha_\ell$. Then, it follows from \eqref{ineq:recur} that
\[
\frac{1}{\alpha_{i+1}} \leq \frac{1+\alpha_\ell \eta}{\alpha_\ell(1 + \alpha_\ell \eta)^2} + \frac{\eta}{1 + \alpha_\ell \eta} =\frac{1}{\alpha_\ell},
\]
and hence that the claim is true for the case $i+1$.
Together with this conclusion, Lemma~\ref{lem:proximal_ContractionCoefficientBoundAndPSIEvolution}(a) simplifies to 
\[
\varepsilon_i \leq \frac{1}{1 + 2\eta \alpha_i + \eta^2 \alpha \alpha_i } \leq \frac{1}{ (1 + \eta \alpha_\ell)^2 }\,,
\] 
for all $i \geq \ell$\,. 
Plugging this inequality into \eqref{ineq:MI}, we obtain 
\begin{equation}\label{ineq:case2}
    \MI_{\Phi}(\rho^X_{0,k}) \leq \frac{\MI_{\Phi}(\rho^X_{0,\ell})}{(1 + \eta \alpha_\ell)^{2(k-\ell)}}\,.
\end{equation}
Finally, the theorem follows by combining \eqref{ineq:case1} and \eqref{ineq:case2}.
\end{proof}

\subsubsection{Proof of Corollary~\ref{cor:PS}}\label{app:PSIterationComplexity}

\begin{proof}
Note that $\tau \le e^{\tau-1}$ for every $\tau\in \R$. Using Theorem \ref{thm:ProximalMain} and this observation with
\[
\tau = \frac{1}{1+\eta \tilde \alpha}, \quad \tilde \alpha = \min \{ \alpha, \alpha_{\PSI}(\rho^X_\ell) \},
\]
we have
    \[
    \frac{\MI_{\Phi}(\rho^X_{0,k})}{\MI_{\Phi}(\rho^X_{0,\ell})} \stackrel{\eqref{ineq:PS}}\leq \frac{1}{\left(1 + \eta \tilde \alpha\right)^{2(k-\ell)} }\le \exp\left(-\frac{2 \eta \tilde \alpha (k-\ell)}{1 + \eta \tilde \alpha }\right).
    \]
    In view of the above inequality, to bound $\MI_{\Phi}(X_0; X_k) \leq \epsilon$, it suffices to bound
    \[
    \exp\left(-\frac{2 \eta \tilde \alpha (k-\ell)}{1 + \eta \tilde \alpha }\right) \le \frac{\epsilon}{\MI_{\Phi}(\rho^X_{0,\ell})},
    \]
    which gives the desired complexity bound on $k$.
\end{proof}

\section{Examples for the Ornstein-Uhlenbeck Process and the Heat Flow}\label{app:OU_HeatFlow}

In this section, we fix $\Phi(x) = x \log x$ and discuss the convergence of mutual information for the Ornstein-Uhlenbeck (OU) process and the heat flow.
The entropy functional of a distribution $\rho$ is defined to be $\sfH(\rho) = -\E_{\rho}[\log{\rho}]$. Recall that the standard mutual information is
\begin{equation}\label{eq:MI}
    \MI(\rho^{XY}) = \KL\left(\rho^{XY} \dvert \rho^X \otimes \rho^Y \right) = \sfH(\rho^Y) - \E_{x \sim \rho^X}[\sfH(\rho_{Y \mid X}(\,\cdot \mid x))]\,.
\end{equation}

\subsection{Convergence Rates in Continuous Time along Langevin Dynamics}\label{app:Tightness_LD}

We provide two propositions showing that the convergence rates of mutual information along OU process and heat flow are $\Theta\left(e^{-2\alpha t}\right)$ and $\Theta\left(1/t\right)$, respectively. This shows the tightness of Theorems \ref{thm:LangevinDynamicsMain} and \ref{thm:MI_LD_regularize}. Before proceeding, we mention the following fact which will be useful in the analysis.

\begin{fact}\label{fact:epi}
    \citep[Theorem~17.7.3]{cover1999elements}
    The {\em entropy power} of a probability distribution $\rho$ on $\R^d$ is
\[
\Lambda(\rho) = \frac{1}{2\pi e} e^{\frac{2}{d} \sfH(\rho)}.
\]
The {\em entropy power inequality} states for independent random variables with distributions $\rho$ and $\nu$,
\[
\Lambda(\rho \ast \nu) \ge \Lambda(\rho) + \Lambda(\nu).
\]
\end{fact}

\subsubsection{Mutual Information along the OU Process} \label{app:OU}

The OU process is the Langevin dynamics for a Gaussian target distribution $\nu = \N(c, \Sigma)$ for some $c \in \R^d$ and $\Sigma \succ 0$.
Here, for simplicity, we consider the case where $c = 0$ and $\Sigma = \frac{1}{\alpha} I$ for some $\alpha > 0$.
In this case, \eqref{eq:LangevinDynamics} becomes
\begin{equation}\label{eq:OU}
    \D X_t = -\alpha X_t \, \D t + \sqrt{2} \D W_t.
\end{equation}
The solution to the OU process is 
\begin{equation}\label{sol:OU}
    X_t = e^{-\alpha t} X_0 + \sqrt{\tau_\alpha(t)} Z,
\end{equation}
where $Z \sim \N(0,I)$ and 
\begin{equation}\label{eq:tau}
    \tau_\alpha(t) = \frac{1-e^{-2\alpha t}}{\alpha}.
\end{equation}

We have the following theorem that describes the rate of convergence of mutual information along the OU process. Note when $\rho_0$ is a Gaussian, $\rho_t$ is a Gaussian for all $t$, and we can compute the mutual information explicitly; however, our statement holds for more general initial distributions $\rho_0$.

\begin{proposition}\label{thm:OU}
Let $X_t \sim \rho_t$ evolve along the OU process \eqref{eq:OU} from $X_0 \sim \rho_0$, where $\rho_0$ satisfies $\Cov(\rho_0) \preceq J \, I$, and let $\rho_{0,t}$ be the joint law of $(X_0,X_t)$.
Then, we have
\begin{equation}\label{ineq:upper}
    \MI(\rho_{0,t}) \le \frac{d}{2} \log\left(1 + \frac{\alpha J}{e^{2\alpha t}-1}\right) \le \frac{\alpha d J}{2(e^{2\alpha t}-1)}\,,
\end{equation}
and,
\begin{equation}\label{ineq:lower}
    \MI(\rho_{0,t}) \ge \frac{d}{2} \log \left( \frac{\alpha e^{\frac{2}{d} \sfH(\rho_0)}}{2\pi e (e^{2\alpha t}-1)} + 1 \right)\,.
\end{equation}
Therefore, along the OU process, as $t \to \infty$, we have
\begin{equation}\label{eq:mi-rate}
    \MI(\rho_{0,t}) = \Theta\left(e^{-2\alpha t}\right).
\end{equation}
\end{proposition}

\begin{proof}
It follows from \eqref{sol:OU} that
\[
\rho_{t \mid 0} = \N(e^{-\alpha t} X_0, \tau_\alpha(t) I), \quad \sfH(\rho_{t \mid 0 }) = \frac{d}{2} \log \left(2\pi e \tau_\alpha(t) \right),
\]
and
\begin{equation}\label{eq:exp-ent}
    \E_{\rho_0}[\sfH(\rho_{t \mid 0})]= \frac{d}{2} \log \left( 2\pi e \tau_\alpha(t) \right).
\end{equation}
We next bound $\sfH(\rho_t)$ where,
$$\rho_t(y) = \int_{\R^d} \rho_0(x) \, \mathbb{P}_{\N(e^{-\alpha t} x, \tau_\alpha(t) I)}(y) \, \D x\,,$$
and $\mathbb{P}_{\N(c, \Sigma)}$ is the density of $\N(c, \Sigma)$\,.
We first derive an upper bound on $\sfH(\rho_t)$.
Using the fact that for a fixed covariance, the Gaussian distribution maximizes entropy, and the observation that 
\[
\Cov(\rho_t) = e^{-2\alpha t} \Cov(\rho_0) + \tau_\alpha(t) I,
\]
we have
\begin{align}
\sfH(\rho_t) &\le \sfH(\N(0,\Cov(\rho_t)))
= \frac{d}{2} \log (2\pi e) + \frac{1}{2} \log \det\left(e^{-2\alpha t} \Cov(\rho_0) + \tau_\alpha(t) I\right) \nn \\
&\le \frac{d}{2} \log (2\pi e) + \frac{1}{2} \log \det\left( 
e^{-2\alpha t} J I + \tau_\alpha(t) I\right)
= \frac{d}{2} \log \left(2\pi e (e^{-2\alpha t} J + \tau_\alpha(t)) \right), \label{ineq:ent-ub}
\end{align}
where the second inequality follows from the assumption that $\Cov(\rho_0) \preceq J \, I$.
Next, we derive a lower bound on $\sfH(\rho_t)$.
Considering the solution \eqref{sol:OU} and letting $\rho$ denote the distribution of $e^{-\alpha t} X_0$ and $\nu = \N(0, \tau_\alpha(t) I)$, then we have $\rho_t = \rho \ast \nu$.
By definition, we have
\[
\sfH(\rho) = \sfH(\law{(e^{-\alpha t} X_0)}) = \sfH(\rho_0) - d \alpha t.
\]
It thus follows from the entropy power inequality (i.e., Fact \ref{fact:epi}) that
\begin{equation}\label{ineq:ent-lb}
    \sfH(\rho_t) \ge \frac{d}{2} \log \left( e^{\frac{2}{d} \sfH(\law{(e^{-\alpha t} X_0)})} + e^{\frac{2}{d} \sfH(\N(0,\tau_\alpha(t) I))} \right) 
    = \frac{d}{2} \log \left( e^{\frac{2}{d} \sfH(\rho_0) - 2 \alpha t} + 2\pi e \tau_\alpha(t) \right).
\end{equation}
Now, \eqref{ineq:upper} immediately follows \eqref{eq:MI}, \eqref{eq:tau}, \eqref{eq:exp-ent}, and \eqref{ineq:ent-ub},
\begin{align*}
    \MI(\rho_{0,t}) 
    \stackrel{\eqref{eq:MI}}= \sfH(\rho_t) - \E_{\rho_0}[\sfH(\rho_{t \mid 0})] 
    \stackrel{\eqref{eq:exp-ent}, \eqref{ineq:ent-ub}}\le \frac{d}{2} \log \left( \frac{e^{-2\alpha t} J}{\tau_\alpha(t)}  + 1\right)  
    \stackrel{\eqref{eq:tau}}= \frac{d}{2} \log\left(\frac{\alpha J}{e^{2\alpha t}-1} + 1\right) \le \frac{\alpha d J}{2(e^{2\alpha t}-1)},
\end{align*}
where the last inequality is due to the fact that $\log(1+x) \le x$.
Using \eqref{eq:MI}, \eqref{eq:tau}, \eqref{eq:exp-ent}, and \eqref{ineq:ent-lb}, we have
\begin{align*}
    \MI(\rho_{0,t}) \stackrel{\eqref{eq:MI}}= \sfH(\rho_t) - \E_{\rho_0}[\sfH(\rho_{t \mid 0})] \stackrel{\eqref{eq:exp-ent}, \eqref{ineq:ent-lb}}\ge \frac{d}{2} \log \left( \frac{e^{\frac{2}{d} \sfH(\rho_0) - 2 \alpha t}}{2\pi e \tau_\alpha(t)} + 1 \right)
    \stackrel{\eqref{eq:tau}}= \frac{d}{2} \log \left( \frac{\alpha e^{\frac{2}{d} \sfH(\rho_0)}}{2\pi e (e^{2\alpha t}-1)} + 1 \right).
\end{align*}
Hence, \eqref{ineq:lower} is proved.
It follows from the fact that $\log(1+x)\ge x/2$ for $x\in [0,1]$ and \eqref{ineq:lower} that as $t \to \infty$,
\[
\MI(\rho_{0,t}) \ge  \frac{\alpha d e^{\frac{2}{d} \sfH(\rho_0)}}{8\pi e (e^{2\alpha t}-1)}.
\]
Finally, \eqref{eq:mi-rate} follows from the above conclusion and \eqref{ineq:upper}.
\end{proof}

\begin{remark}
    Using Stam's inequality \citep{stam1959some}\,,
    \[
    \Lambda(\rho) \FI(\rho) \ge d,
    \] 
    and the Blachman-Stam inequality \citep{blachman1965convolution}\,,
    \[
\frac{1}{\FI(\rho*\nu)} \ge \frac{1}{\FI(\rho)} + \frac{1}{\FI(\nu)},
\]
instead of using the entropy power inequality (i.e., Fact \ref{fact:epi}),
    we can derive an alternative lower bound,
    \[
    \MI(\rho_{0,t}) 
    \ge \frac{d}{2} \log \left( \frac{\alpha d}{(e^{2\alpha t}-1)\FI(\rho_0)} + 1 \right)\,.
    \]
\end{remark}

\subsubsection{Mutual Information along the Heat Flow}

Recall the target distribution for the OU process discussed in Appendix \ref{app:OU} is $\nu = \N (0, \frac{1}{\alpha}I)$. As $\alpha \to 0$, \eqref{eq:OU} simply reduces to
\begin{equation}\label{eq:heat}
    \D X_t = \sqrt{2} \D W_t\,,
\end{equation}
 which is the heat flow. It is the SDE corresponding to the heat equation: $\partial_t \rho_t = \Delta \rho_t$.

 We have the following proposition that describes the rate of convergence of mutual information along the heat flow.
Its proof follows similarly from that of Proposition \ref{thm:OU} and hence is omitted.

 \begin{proposition}\label{thm:HeatFlow}
Let $X_t \sim \rho_t$ evolve along the heat flow \eqref{eq:heat} from $X_0 \sim \rho_0$ from some $\rho_0$ and let $\rho_{0,t}$ be the joint law of $(X_0,X_t)$.
Then,
\[
\MI(\rho_{0,t}) \le \frac{1}{2} \sum_{i=1}^d \log \left(\frac{\lambda_i(\Cov(\rho_0))}{2t} + 1\right),
\]
where $\lambda_i(\cdot)$ is the $i$-th largest eigenvalue of a matrix, and
\[
\MI(\rho_{0,t}) \ge \frac{d}{2} \log \left( \frac{e^{\frac{2}{d} \sfH(\rho_0)} }{4\pi e t} + 1\right).
\]
Therefore, along the heat flow, as $t \to \infty$, we have
\[
\MI(\rho_{0,t}) = \Theta\left(\frac{1}{t}\right).
\]
\end{proposition}

\subsection{Convergence Rates in Discrete Time along ULA}\label{app:OU_ULA}

Here we discuss ULA for the OU process, indicating the tightness of Theorems \ref{thm:ULA_Main} and \ref{thm:MI_ULA_regularize}.

For simplicity, as in Appendix~\ref{app:OU},  we again consider the target to be $\nu = \N (0, \frac{1}{\alpha}I)$\,. In this case, the ULA update~\eqref{eq:ULA} becomes:
\begin{equation}
    X_{k+1} = (1-\eta \alpha)X_k + \sqrt{2\eta} Z_k\,,\label{eq:OU_ULA}
\end{equation}
and the solution to \eqref{eq:OU_ULA} is:
\begin{equation}\label{eq:OU_ULA_solution}
    X_k = (1-\eta \alpha)^k X_0 + \sqrt{\frac{2(1-(1-\eta \alpha)^{2k})}{\alpha (2-\eta \alpha)}}Z\,,
\end{equation}
where $Z \sim \N(0, I)$\,. To show a tightness result, we simply take $\rho_0$ to be $\N(0, I)$ in the proposition below.

\begin{proposition}\label{thm:OU_ULA}
    Let $X_k$ evolve along \eqref{eq:OU_ULA} from $X_0 \sim \rho_0 = \N (0, I)$. Then 
    \begin{align*}
    \MI(X_0; X_k) = \MI(\rho_{0,k}) = \frac{d}{2} \log \left[ 1+  \frac{(1-\eta\alpha)^{2k}(2\alpha-\eta\alpha^2)}{2(1-(1-\eta \alpha)^{2k})}   \right]\,.
\end{align*}
Therefore, as $k \to \infty$, we have that $\MI(\rho_{0,k}) = \Theta(d\alpha(1-\eta \alpha)^{2k})$.
\end{proposition}

\begin{proof}
It follows from \eqref{eq:OU_ULA_solution} that
\begin{equation}\label{eq:OU_ULA_rhokmid0}
    \rho_{k \mid 0} = \N \left(  (1-\eta \alpha)^k X_0, \frac{2(1-(1-\eta \alpha)^{2k})}{\alpha (2-\eta \alpha)} I    \right),
\end{equation}
which, together with the fact that $X_0 \sim \rho_0 = \N (0, I)$, implies that $X_k \sim \rho_k$ where
\begin{equation}\label{eq:OU_ULA_rhok}
    \rho_k = \N \left( 0, \frac{2+(1-\eta\alpha)^{2k}(2\alpha-\eta\alpha^2-2)}{\alpha (2-\eta \alpha)}I      \right)\,.
\end{equation}
Using \eqref{eq:MI} and the formula for entropy of a Gaussian on \eqref{eq:OU_ULA_rhok} and \eqref{eq:OU_ULA_rhokmid0}, we therefore get that:
\begin{align*}
    \MI(\rho_{0,k}) &\stackrel{\eqref{eq:MI}}= \sfH(\rho_k) - \E_{\rho_0} [\sfH (\rho_{k \mid 0})] = \frac{d}{2} \log \left[  \frac{2+(1-\eta\alpha)^{2k}(2\alpha-\eta\alpha^2-2)}{2(1-(1-\eta \alpha)^{2k})}   \right]\\
    &= \frac{d}{2} \log \left[ 1+  \frac{(1-\eta\alpha)^{2k}(2\alpha-\eta\alpha^2)}{2(1-(1-\eta \alpha)^{2k})}   \right]\,,
\end{align*}
which, together with the fact that $\log(1+x)\ge x/2$ for $x\in [0,1]$, proves the desired claim.
\end{proof}

This shows the tightness of Theorems \ref{thm:ULA_Main} and \ref{thm:MI_ULA_regularize}.

\subsection{Convergence Rates in Discrete Time along the Proximal Sampler}\label{app:OU_PS}

Let the target distribution be $\nu^X = \N(0, \frac{1}{\alpha}I)$ for $\alpha > 0$ and $\rho_0^X = \N(m_0, c_0^2 I)$ for $m_0 \in \R^d$ and $c_0 > 0$. 
 In this case, the conditional distribution $\nu^{X \mid Y}$~\eqref{eq:RGOdistribution} is given by
 \[
 \nu^{X \mid Y}( \,\cdot\, | \,y) = \N \left( \frac{y}{1+\alpha \eta}, \frac{\eta}{1+\alpha \eta}I \right).
 \]
Explicit computation for the Proximal Sampler for the Gaussian case~\citep[Section~4.4]{chen2022improved} reveals that we have $\rho_k^X = \N(m_k, c_k^2 I)$ for $k \geq 0$ where
\[
m_{k+1} = \frac{m_k}{1+\alpha \eta} \hspace{1cm}\text{and}\hspace{1cm}c_{k+1}^2 - \frac{1}{\alpha} = \frac{1}{(1+\alpha\eta)^2}\Big( c_{k}^2 - \frac{1}{\alpha} \Big)\,.
\]
Hence for all $k \geq 0$
\begin{equation}\label{eq:PSIteratesAlongOU}
m_{k} = \frac{m_0}{(1+\alpha \eta)^k} \hspace{1cm}\text{and}\hspace{1cm}c_{k}^2 - \frac{1}{\alpha} = \frac{1}{(1+\alpha\eta)^{2k}}\Big( c_{0}^2 - \frac{1}{\alpha} \Big)\,.
\end{equation}
We use this explicit Gaussian solution to compute $\MI(X_0; X_k)$ along the Proximal Sampler.

\begin{proposition}
    Let $X_k \sim \rho_k^X$ be iterates along the Proximal Sampler~\eqref{eqs:ProximalSampler} with $\nu^X = \N(0, \frac{1}{\alpha}I)$ for $\alpha >0$ and $\rho_0^X = \N(0,I)$. Then
    \[
    \MI(X_0; X_k) = \frac{d}{2} \log \left[ 1 + \frac{\alpha}{(1+\alpha\eta)^{2k}-1} \right].
    \]
    Therefore, as $k \to \infty$, we have that $\MI(\rho_{0,k}) = \Theta(d\alpha(1+\eta \alpha)^{-2k})$.
\end{proposition}

\begin{proof}
    Recall from \eqref{eq:MI} that $\MI(\rho_{0,k}) = \sfH(\rho_k) - \E_{\rho_0} [\sfH (\rho_{k \mid 0})]$. 
    It follows from~\eqref{eq:PSIteratesAlongOU} that $\rho_k = \N(m_k, c_k^2 I)$ with
    \begin{equation}\label{eq:density}
        m_{k} = 0 \hspace{1cm}\text{and}\hspace{1cm}c_{k}^2 - \frac{1}{\alpha} = \frac{1}{(1+\alpha\eta)^{2k}}\Big( 1 - \frac{1}{\alpha} \Big).
    \end{equation}
    It remains to evaluate $\rho_{k|0}$ in order to compute $\E_{x \sim \rho_0} [\sfH (\rho_{k \mid 0 = x})]$. Consider a fixed $x \in \R^d$ and suppose we wish to compute $\rho_{k \mid 0 = x}$. Then we can view this distribution as the solution along the Proximal Sampler when started from $\N(x, c^2I)$ as $c^2 \to 0$. Using~\eqref{eq:PSIteratesAlongOU} again, we obtain
    \begin{equation}\label{eq:rho0k}
        \rho_{k\mid 0 = x} = \N \Big( \frac{x}{(1+\alpha \eta)^k}, \frac{1}{\alpha}\big( 1-\frac{1}{(1+\alpha\eta)^{2k}}  \big)I \Big)\,.
    \end{equation}
    Using \eqref{eq:MI} and the formula for entropy of a Gaussian on \eqref{eq:density} and \eqref{eq:rho0k}, we therefore get that 
    \begin{align*}
        \MI(X_0; X_k) &\stackrel{\eqref{eq:MI}}=\sfH(\rho_k) - \E_{\rho_0} [\sfH (\rho_{k \mid 0})]\stackrel{\eqref{eq:density}, \eqref{eq:rho0k}}= \frac{d}{2}\log \left[  \frac{(1+\alpha\eta)^{2k} + \alpha-1}{(1+\alpha\eta)^{2k}-1}   \right]\\
        &= \frac{d}{2} \log \left[ 1 + \frac{\alpha}{(1+\alpha\eta)^{2k}-1} \right].
    \end{align*}
    This proves the desired claim.
\end{proof}

This shows the tightness of Theorem \ref{thm:ProximalMain}.

\addcontentsline{toc}{section}{References}
\bibliography{mi}

\begin{thebibliography}{84}
\providecommand{\natexlab}[1]{#1}
\providecommand{\url}[1]{\texttt{#1}}
\expandafter\ifx\csname urlstyle\endcsname\relax
  \providecommand{\doi}[1]{doi: #1}\else
  \providecommand{\doi}{doi: \begingroup \urlstyle{rm}\Url}\fi

\bibitem[Achleitner et~al.(2015)Achleitner, Arnold, and
  Stürzer]{achleitner2015large}
Franz Achleitner, Anton Arnold, and Dominik Stürzer.
\newblock {Large-time behavior in non-symmetric Fokker-Planck equations}.
\newblock \emph{Rivista di Matematica della Università di Parma}, 6\penalty0
  (1):\penalty0 1--68, 2015.

\bibitem[Altschuler and Talwar(2023)]{AltschulerTalwar23}
Jason Altschuler and Kunal Talwar.
\newblock {Resolving the Mixing Time of the Langevin Algorithm to its
  Stationary Distribution for Log-Concave Sampling}.
\newblock In \emph{Proceedings of Thirty Sixth Conference on Learning Theory},
  Proceedings of Machine Learning Research. PMLR, 2023.

\bibitem[Altschuler and Chewi(2024{\natexlab{a}})]{altschuler2023shifted}
Jason~M Altschuler and Sinho Chewi.
\newblock {Shifted composition I: Harnack and reverse transport inequalities}.
\newblock \emph{IEEE Transactions on Information Theory}, 2024{\natexlab{a}}.

\bibitem[Altschuler and Chewi(2024{\natexlab{b}})]{altschuler2024faster}
Jason~M Altschuler and Sinho Chewi.
\newblock Faster high-accuracy log-concave sampling via algorithmic warm
  starts.
\newblock \emph{Journal of the ACM}, 71\penalty0 (3):\penalty0 1--55,
  2024{\natexlab{b}}.

\bibitem[Anantharam et~al.(2013)Anantharam, Gohari, Kamath, and
  Nair]{anantharam2013maximal}
Venkat Anantharam, Amin Gohari, Sudeep Kamath, and Chandra Nair.
\newblock {On maximal correlation, hypercontractivity, and the data processing
  inequality studied by Erkip and Cover}.
\newblock \emph{arXiv preprint arXiv:1304.6133}, 2013.

\bibitem[Asoodeh et~al.(2021)Asoodeh, Aliakbarpour, and
  Calmon]{asoodeh2021local}
Shahab Asoodeh, Maryam Aliakbarpour, and Flavio~P Calmon.
\newblock {Local Differential Privacy Is Equivalent to Contraction of an
  $f$-Divergence}.
\newblock In \emph{2021 IEEE International Symposium on Information Theory
  (ISIT)}, pages 545--550. IEEE, 2021.

\bibitem[Atar and Weissman(2012)]{atar2010mutual}
Rami Atar and Tsachy Weissman.
\newblock {Mutual information, relative entropy, and estimation in the Poisson
  channel}.
\newblock \emph{IEEE Transactions on Information theory}, 58\penalty0
  (3):\penalty0 1302--1318, 2012.

\bibitem[Bakry et~al.(2014)Bakry, Gentil, and Ledoux]{BGL14}
Dominique Bakry, Ivan Gentil, and Michel Ledoux.
\newblock \emph{Analysis and geometry of Markov diffusion operators}, volume
  103.
\newblock Springer, 2014.

\bibitem[Belitski et~al.(2008)Belitski, Gretton, Magri, Murayama, Montemurro,
  Logothetis, and Panzeri]{belitski2008local}
A~Belitski, A~Gretton, C~Magri, Y~Murayama, M~Montemurro, N~Logothetis, and
  S~Panzeri.
\newblock Local field potentials and spiking activity in primary visual cortex
  convey independent information about natural stimuli.
\newblock \emph{Journal of Neuroscience}, 28\penalty0 (22):\penalty0
  5696--5709, 2008.

\bibitem[Blachman(1965)]{blachman1965convolution}
Nelson Blachman.
\newblock The convolution inequality for entropy powers.
\newblock \emph{IEEE Transactions on Information theory}, 11\penalty0
  (2):\penalty0 267--271, 1965.

\bibitem[Bobkov et~al.(2001)Bobkov, Gentil, and
  Ledoux]{bobkov2001hypercontractivity}
Sergey~G Bobkov, Ivan Gentil, and Michel Ledoux.
\newblock Hypercontractivity of {Hamilton--Jacobi} equations.
\newblock \emph{Journal de Math{\'e}matiques Pures et Appliqu{\'e}es},
  80\penalty0 (7):\penalty0 669--696, 2001.

\bibitem[Bolley and Gentil(2010)]{bolley2010phi}
François Bolley and Ivan Gentil.
\newblock Phi-entropy inequalities for diffusion semigroups.
\newblock \emph{{Journal de Mathématiques Pures et Appliquées}}, 93\penalty0
  (5):\penalty0 449--473, 2010.

\bibitem[Bou-Rabee and Eberle(2023)]{bou2023mixing}
Nawaf Bou-Rabee and Andreas Eberle.
\newblock {Mixing time guarantees for unadjusted Hamiltonian Monte Carlo}.
\newblock \emph{Bernoulli}, 29\penalty0 (1):\penalty0 75--104, 2023.

\bibitem[Boursier et~al.(2023)Boursier, Chafaï, and
  Labbé]{boursier2023universal}
Jeanne Boursier, Djalil Chafaï, and Cyril Labbé.
\newblock {Universal cutoff for Dyson Ornstein Uhlenbeck process}.
\newblock \emph{Probability Theory and Related Fields}, 185\penalty0
  (1):\penalty0 449--512, 2023.

\bibitem[Bradley(1986)]{bradley1986basic}
Richard~C Bradley.
\newblock Basic properties of strong mixing conditions.
\newblock \emph{Dependence in Probability and Statistics: A Survey of Recent
  Results}, pages 165--192, 1986.

\bibitem[Bradley(2005)]{bradley_mixing}
Richard~C. Bradley.
\newblock {Basic Properties of Strong Mixing Conditions. A Survey and Some Open
  Questions}.
\newblock \emph{Probability Surveys}, 2:\penalty0 107 -- 144, 2005.

\bibitem[Cao et~al.(2019)Cao, Lu, and Lu]{cao2019exponential}
Yu~Cao, Jianfeng Lu, and Yulong Lu.
\newblock {Exponential decay of Rényi divergence under Fokker--Planck
  equations}.
\newblock \emph{Journal of Statistical Physics}, 176:\penalty0 1172--1184,
  2019.

\bibitem[Chafa{\"\i}(2004)]{chafai2004entropies}
Djalil Chafa{\"\i}.
\newblock {Entropies, convexity, and functional inequalities, On
  $\Phi$-entropies and $\Phi$-Sobolev inequalities}.
\newblock \emph{Journal of Mathematics of Kyoto University}, 2004.

\bibitem[Chen et~al.(2022)Chen, Chewi, Salim, and Wibisono]{chen2022improved}
Yongxin Chen, Sinho Chewi, Adil Salim, and Andre Wibisono.
\newblock Improved analysis for a proximal algorithm for sampling.
\newblock In \emph{Conference on Learning Theory}, pages 2984--3014. PMLR,
  2022.

\bibitem[Cheng and Bartlett(2018)]{cheng2018convergence}
Xiang Cheng and Peter Bartlett.
\newblock {Convergence of Langevin MCMC in KL-divergence}.
\newblock In \emph{Algorithmic Learning Theory}, pages 186--211. PMLR, 2018.

\bibitem[Chewi(2024)]{chewi2023log}
Sinho Chewi.
\newblock Log-concave sampling.
\newblock 2024.
\newblock Draft available at: \url{https://chewisinho.github.io}.

\bibitem[Chewi et~al.(2022{\natexlab{a}})Chewi, Erdogdu, Li, Shen, and
  Zhang]{chewietal2022lmcpoincare}
Sinho Chewi, Murat~A. Erdogdu, Mufan~(Bill) Li, Ruoqi Shen, and Matthew Zhang.
\newblock Analysis of {L}angevin {M}onte {C}arlo from {P}oincar\'e to
  log-{S}obolev.
\newblock In \emph{Proceedings of Thirty Fifth Conference on Learning Theory},
  Proceedings of Machine Learning Research, pages 1--2. PMLR,
  2022{\natexlab{a}}.

\bibitem[Chewi et~al.(2022{\natexlab{b}})Chewi, Gerber, Lu, Le~Gouic, and
  Rigollet]{chewi2022query}
Sinho Chewi, Patrik~R Gerber, Chen Lu, Thibaut Le~Gouic, and Philippe Rigollet.
\newblock The query complexity of sampling from strongly log-concave
  distributions in one dimension.
\newblock In \emph{Conference on Learning Theory}, pages 2041--2059. PMLR,
  2022{\natexlab{b}}.

\bibitem[Cover(1999)]{cover1999elements}
Thomas~M Cover.
\newblock \emph{Elements of information theory}.
\newblock John Wiley \& Sons, 1999.

\bibitem[Dalalyan(2017{\natexlab{a}})]{dalalyan2017further}
Arnak Dalalyan.
\newblock {Further and stronger analogy between sampling and optimization:
  Langevin Monte Carlo and gradient descent}.
\newblock In \emph{Conference on Learning Theory}, pages 678--689. PMLR,
  2017{\natexlab{a}}.

\bibitem[Dalalyan(2017{\natexlab{b}})]{dalalyan2017theoretical}
Arnak~S Dalalyan.
\newblock Theoretical guarantees for approximate sampling from smooth and
  log-concave densities.
\newblock \emph{Journal of the Royal Statistical Society Series B: Statistical
  Methodology}, 79\penalty0 (3):\penalty0 651--676, 2017{\natexlab{b}}.

\bibitem[Dolbeault and Li(2018)]{dolbeault2018varphi}
Jean Dolbeault and Xingyu Li.
\newblock {$\varphi$-entropies: Convexity, coercivity and hypocoercivity for
  Fokker--Planck and kinetic Fokker--Planck equations}.
\newblock \emph{Mathematical Models and Methods in Applied Sciences},
  28\penalty0 (13):\penalty0 2637--2666, 2018.

\bibitem[Durmus and Moulines(2017)]{durmus2017nonasymptotic}
Alain Durmus and {\'E}ric Moulines.
\newblock Nonasymptotic convergence analysis for the {U}nadjusted {L}angevin
  {A}lgorithm.
\newblock \emph{The Annals of Applied Probability}, 27\penalty0 (3):\penalty0
  1551, 2017.

\bibitem[Durmus et~al.(2019)Durmus, Majewski, and
  Miasojedow]{durmus2019analysis}
Alain Durmus, Szymon Majewski, and B{\l}a{\.z}ej Miasojedow.
\newblock Analysis of {Langevin Monte Carlo} via convex optimization.
\newblock \emph{Journal of Machine Learning Research}, 20\penalty0
  (73):\penalty0 1--46, 2019.

\bibitem[Erkip and Cover(1998)]{erkip1998efficiency}
Elza Erkip and Thomas~M Cover.
\newblock The efficiency of investment information.
\newblock \emph{IEEE Transactions on information theory}, 44\penalty0
  (3):\penalty0 1026--1040, 1998.

\bibitem[Esposito et~al.(2020)Esposito, Gastpar, and Issa]{esposito2020robust}
Amedeo~Roberto Esposito, Michael Gastpar, and Ibrahim Issa.
\newblock {Robust Generalization via $f$-Mutual Information}.
\newblock In \emph{2020 IEEE International Symposium on Information Theory
  (ISIT)}, pages 2723--2728. IEEE, 2020.

\bibitem[Fan et~al.(2023)Fan, Yuan, and Chen]{fan2023improved}
Jiaojiao Fan, Bo~Yuan, and Yongxin Chen.
\newblock Improved dimension dependence of a proximal algorithm for sampling.
\newblock In \emph{The Thirty Sixth Annual Conference on Learning Theory},
  pages 1473--1521. PMLR, 2023.

\bibitem[Fan et~al.(2025)Fan, Zou, and Wang]{fan2025differential}
Luyao Fan, Jiayang Zou, and Jia Wang.
\newblock Differential properties of information in jump-diffusion channels.
\newblock \emph{arXiv preprint arXiv:2501.05708}, 2025.

\bibitem[Gelman et~al.(1995)Gelman, Carlin, Stern, and
  Rubin]{gelman1995bayesian}
Andrew Gelman, John~B Carlin, Hal~S Stern, and Donald~B Rubin.
\newblock \emph{Bayesian data analysis}.
\newblock Chapman \& Hall/CRC, 1995.

\bibitem[Geyer(2011)]{geyer2011introduction}
Charles~J Geyer.
\newblock {Introduction to Markov chain Monte Carlo}.
\newblock \emph{Handbook of markov chain monte carlo}, 20116022\penalty0
  (45):\penalty0 22, 2011.

\bibitem[Gilks et~al.(1995{\natexlab{a}})Gilks, Richardson, and
  Spiegelhalter]{gilks1995markov}
Walter~R Gilks, Sylvia Richardson, and David Spiegelhalter.
\newblock \emph{{Markov chain Monte Carlo in practice}}.
\newblock CRC press, 1995{\natexlab{a}}.

\bibitem[Gilks et~al.(1995{\natexlab{b}})Gilks, Richardson, and
  Spiegelhalter]{gilks1995introducing}
Walter~R Gilks, Sylvia Richardson, and David~J Spiegelhalter.
\newblock {Introducing Markov chain Monte Carlo}.
\newblock \emph{Markov chain Monte Carlo in practice}, 1, 1995{\natexlab{b}}.

\bibitem[Goldfeld et~al.(2019)Goldfeld, Van Den~Berg, Greenewald, Melnyk,
  Nguyen, Kingsbury, and Polyanskiy]{goldfeld2019estimating}
Ziv Goldfeld, Ewout Van Den~Berg, Kristjan Greenewald, Igor Melnyk, Nam Nguyen,
  Brian Kingsbury, and Yury Polyanskiy.
\newblock {Estimating Information Flow in Deep Neural Networks}.
\newblock In \emph{International Conference on Machine Learning}, pages
  2299--2308. PMLR, 2019.

\bibitem[Gretton and Gy{\"o}rfi(2008)]{gretton2008nonparametric}
Arthur Gretton and L{\'a}szl{\'o} Gy{\"o}rfi.
\newblock Nonparametric independence tests: Space partitioning and kernel
  approaches.
\newblock In \emph{Algorithmic Learning Theory: 19th International Conference,
  ALT 2008, Budapest, Hungary, October 13-16, 2008. Proceedings 19}, pages
  183--198. Springer, 2008.

\bibitem[Guo et~al.(2005)Guo, Shamai, and Verd{\'u}]{guo2005mutual}
Dongning Guo, Shlomo Shamai, and Sergio Verd{\'u}.
\newblock Mutual information and minimum mean-square error in {G}aussian
  channels.
\newblock \emph{IEEE Transactions on Information Theory}, 51\penalty0
  (4):\penalty0 1261--1282, 2005.

\bibitem[Gy{\"o}rfi and Vajda(2002)]{gyorfi2002asymptotic}
L{\'a}szl{\'o} Gy{\"o}rfi and I~Vajda.
\newblock Asymptotic distributions for goodness-of-fit statistics in a sequence
  of multinomial models.
\newblock \emph{Statistics \& probability letters}, 56\penalty0 (1):\penalty0
  57--67, 2002.

\bibitem[Ho et~al.(2022)Ho, Petrik, and Wiesemann]{ho2022robust}
Chin~Pang Ho, Marek Petrik, and Wolfram Wiesemann.
\newblock {Robust $\Phi$-Divergence MDPs}.
\newblock \emph{Advances in Neural Information Processing Systems},
  35:\penalty0 32680--32693, 2022.

\bibitem[Johannes and Polson(2003)]{johannes2003mcmc}
Michael Johannes and Nicholas Polson.
\newblock {MCMC Methods for Continuous-Time Financial Econometric}.
\newblock \emph{Handbook of Financial Econometrics}, 2, 2003.

\bibitem[Jordan et~al.(1998)Jordan, Kinderlehrer, and Otto]{JKO98}
Richard Jordan, David Kinderlehrer, and Felix Otto.
\newblock The variational formulation of the {F}okker--{P}lanck equation.
\newblock \emph{SIAM journal on mathematical analysis}, 29\penalty0
  (1):\penalty0 1--17, 1998.

\bibitem[Kass et~al.(1998)Kass, Carlin, Gelman, and Neal]{kass1998markov}
Robert~E Kass, Bradley~P Carlin, Andrew Gelman, and Radford~M Neal.
\newblock {Markov chain Monte Carlo in practice: a roundtable discussion}.
\newblock \emph{The American Statistician}, 52\penalty0 (2):\penalty0 93--100,
  1998.

\bibitem[Koehler et~al.(2024)Koehler, Lee, and Vuong]{koehler2024efficiently}
Frederic Koehler, Holden Lee, and Thuy-Duong Vuong.
\newblock Efficiently learning and sampling multimodal distributions with
  data-based initialization.
\newblock \emph{arXiv preprint arXiv:2411.09117}, 2024.

\bibitem[Kook and Vempala(2024)]{kook2024sampling}
Yunbum Kook and Santosh~S Vempala.
\newblock Sampling and integration of logconcave functions by algorithmic
  diffusion.
\newblock \emph{arXiv preprint arXiv:2411.13462}, 2024.

\bibitem[Kook et~al.(2024)Kook, Vempala, and Zhang]{kook2024inandout}
Yunbum Kook, Santosh~S. Vempala, and Matthew~S. Zhang.
\newblock {In-and-Out: Algorithmic Diffusion for Sampling Convex Bodies}.
\newblock \emph{Advances in Neural Information Processing Systems}, 2024.

\bibitem[Lee et~al.(2021)Lee, Shen, and Tian]{lee2021structured}
Yin~Tat Lee, Ruoqi Shen, and Kevin Tian.
\newblock {Structured logconcave sampling with a restricted Gaussian oracle}.
\newblock In \emph{Conference on Learning Theory}, pages 2993--3050. PMLR,
  2021.

\bibitem[Levin et~al.(2017)Levin, Peres, and Wilmer]{levin2017markov}
David~A Levin, Yuval Peres, and Elizabeth Wilmer.
\newblock \emph{Markov chains and mixing times}, volume 107.
\newblock American Mathematical Society, 2017.

\bibitem[Liang and Chen(2022)]{liang2022proximal}
Jiaming Liang and Yongxin Chen.
\newblock A proximal algorithm for sampling from non-smooth potentials.
\newblock In \emph{2022 Winter Simulation Conference (WSC)}, pages 3229--3240.
  IEEE, 2022.

\bibitem[Liang and Chen(2023)]{liang2023a}
Jiaming Liang and Yongxin Chen.
\newblock A proximal algorithm for sampling.
\newblock \emph{Transactions on Machine Learning Research}, 2023.

\bibitem[Liang and Chen(2024)]{liang2024proximal}
Jiaming Liang and Yongxin Chen.
\newblock Proximal oracles for optimization and sampling.
\newblock \emph{arXiv preprint arXiv:2404.02239}, 2024.

\bibitem[Lu et~al.(2024)Lu, Zhang, Sun, Guo, and Yu]{lu24micl}
Yiwei Lu, Guojun Zhang, Sun Sun, Hongyu Guo, and Yaoliang Yu.
\newblock {$f$-MICL: Understanding and Generalizing InfoNCE-based Contrastive
  Learning}.
\newblock \emph{Transactions on Machine Learning Research}, 2024.

\bibitem[Madras and Randall(2002)]{madras02markov}
Neal Madras and Dana Randall.
\newblock {Markov Chain Decomposition for Convergence Rate Analysis}.
\newblock \emph{The Annals of Applied Probability}, 12\penalty0 (2), 2002.

\bibitem[Madras and Slade(1993)]{madrasslade93}
Neal Madras and Gordon Slade.
\newblock \emph{The Self-Avoiding Walk}.
\newblock Birkh\"auser Boston, 1993.

\bibitem[Margossian and Gelman(2023)]{margossian2023many}
Charles~C Margossian and Andrew Gelman.
\newblock {For how many iterations should we run Markov chain Monte Carlo?}
\newblock \emph{arXiv preprint arXiv:2311.02726}, 2023.

\bibitem[Margossian et~al.(2024)Margossian, Hoffman, Sountsov, Riou-Durand,
  Vehtari, and Gelman]{margossian2024nested}
Charles~C Margossian, Matthew~D Hoffman, Pavel Sountsov, Lionel Riou-Durand,
  Aki Vehtari, and Andrew Gelman.
\newblock {Nested $\hat R$: Assessing the convergence of Markov chain Monte
  Carlo when running many short chains}.
\newblock \emph{Bayesian Analysis}, 1\penalty0 (1):\penalty0 1--28, 2024.

\bibitem[Mitra and Wibisono(2025)]{mitra2024fast}
Siddharth Mitra and Andre Wibisono.
\newblock {Fast Convergence of $\Phi$-Divergence Along the Unadjusted Langevin
  Algorithm and Proximal Sampler}.
\newblock In \emph{36th International Conference on Algorithmic Learning
  Theory}, 2025.

\bibitem[Montenegro and Tetali(2006)]{montenegro2006mathematical}
Ravi Montenegro and Prasad Tetali.
\newblock {Mathematical aspects of mixing times in Markov chains}.
\newblock \emph{Foundations and Trends{\textregistered} in Theoretical Computer
  Science}, 1\penalty0 (3):\penalty0 237--354, 2006.

\bibitem[Nemenman et~al.(2004)Nemenman, Bialek, and de~Ruyter~van
  Steveninck]{nemenman2004entropy}
Ilya Nemenman, William Bialek, and Rob de~Ruyter~van Steveninck.
\newblock Entropy and information in neural spike trains: Progress on the
  sampling problem.
\newblock \emph{Physical Review E—Statistical, Nonlinear, and Soft Matter
  Physics}, 69\penalty0 (5):\penalty0 056111, 2004.

\bibitem[Otto and Villani(2001)]{OV01}
Felix Otto and C{\'e}dric Villani.
\newblock {Comment on: ``Hypercontractivity of Hamilton--Jacobi equations''',
  by S. Bobkov, I. Gentil and M. Ledoux}.
\newblock \emph{Journal de Math{\'e}matiques Pures et Appliqu{\'e}es},
  80\penalty0 (7):\penalty0 697--700, 2001.

\bibitem[Panaganti et~al.(2024)Panaganti, Wierman, and
  Mazumdar]{panaganti2024model}
Kishan Panaganti, Adam Wierman, and Eric Mazumdar.
\newblock {Model-Free Robust $\Phi$-Divergence Reinforcement Learning Using
  Both Offline and Online Data}.
\newblock \emph{arXiv preprint arXiv:2405.05468}, 2024.

\bibitem[Pensia et~al.(2024)Pensia, Jog, and Loh]{pensia2024sample}
Ankit Pensia, Varun Jog, and Po-Ling Loh.
\newblock The sample complexity of simple binary hypothesis testing.
\newblock In \emph{Proceedings of Thirty Seventh Conference on Learning
  Theory}, Proceedings of Machine Learning Research. PMLR, 30 Jun--03 Jul 2024.

\bibitem[Polyanskiy and Wu(2017)]{polyanskiy2017strong}
Yury Polyanskiy and Yihong Wu.
\newblock {Strong data-processing inequalities for channels and Bayesian
  networks}.
\newblock In \emph{Convexity and Concentration}, pages 211--249. Springer,
  2017.

\bibitem[Polyanskiy and Wu(2025)]{polyanskiy2025information}
Yury Polyanskiy and Yihong Wu.
\newblock \emph{Information theory: From coding to learning}.
\newblock Cambridge university press, 2025.

\bibitem[Raginsky(2016)]{raginsky2016strong}
Maxim Raginsky.
\newblock {Strong data processing inequalities and $\Phi$-{S}obolev
  inequalities for discrete channels}.
\newblock \emph{IEEE Transactions on Information Theory}, 62\penalty0
  (6):\penalty0 3355--3389, 2016.

\bibitem[Roberts and Rosenthal(1998)]{roberts1998optimal}
Gareth~O Roberts and Jeffrey~S Rosenthal.
\newblock {Optimal scaling of discrete approximations to Langevin diffusions}.
\newblock \emph{Journal of the Royal Statistical Society: Series B (Statistical
  Methodology)}, 60\penalty0 (1):\penalty0 255--268, 1998.

\bibitem[Roberts and Tweedie(1996)]{roberts1996exponential}
Gareth~O Roberts and Richard~L Tweedie.
\newblock Exponential convergence of {L}angevin distributions and their
  discrete approximations.
\newblock \emph{Bernoulli}, 2\penalty0 (4):\penalty0 341--363, 1996.

\bibitem[Sason and Verd{\'u}(2016)]{sason2016f}
Igal Sason and Sergio Verd{\'u}.
\newblock $ f $-divergence inequalities.
\newblock \emph{IEEE Transactions on Information Theory}, 62\penalty0
  (11):\penalty0 5973--6006, 2016.

\bibitem[Saumard and Wellner(2014)]{saumard2014log}
Adrien Saumard and Jon~A Wellner.
\newblock {Log-concavity and strong log-concavity: A review}.
\newblock \emph{Statistics surveys}, 8:\penalty0 45, 2014.

\bibitem[Stam(1959)]{stam1959some}
Aart~J Stam.
\newblock {Some inequalities satisfied by the quantities of information of
  Fisher and Shannon}.
\newblock \emph{Information and Control}, 2\penalty0 (2):\penalty0 101--112,
  1959.

\bibitem[Vempala and Wibisono(2019)]{VW19}
Santosh Vempala and Andre Wibisono.
\newblock Rapid convergence of the {U}nadjusted {L}angevin {A}lgorithm:
  {I}soperimetry suffices.
\newblock In \emph{Advances in Neural Information Processing Systems},
  volume~32. Curran Associates, Inc., 2019.

\bibitem[Villani(2009)]{villani2009optimal}
C{\'e}dric Villani.
\newblock \emph{Optimal Transport: Old and New}.
\newblock Springer, 2009.

\bibitem[Villani(2021)]{villani2021topics}
C{\'e}dric Villani.
\newblock \emph{Topics in optimal transportation}, volume~58.
\newblock American Mathematical Soc., 2021.

\bibitem[Von~Toussaint(2011)]{von2011bayesian}
Udo Von~Toussaint.
\newblock Bayesian inference in physics.
\newblock \emph{Reviews of Modern Physics}, 83\penalty0 (3):\penalty0 943,
  2011.

\bibitem[Wibisono and Jog(2018{\natexlab{a}})]{WJ18a}
Andre Wibisono and Varun Jog.
\newblock Convexity of mutual information along the heat flow.
\newblock In \emph{2018 IEEE International Symposium on Information Theory
  (ISIT)}, pages 1615--1619. IEEE, 2018{\natexlab{a}}.

\bibitem[Wibisono and Jog(2018{\natexlab{b}})]{WJ18b}
Andre Wibisono and Varun Jog.
\newblock Convexity of mutual information along the {O}rnstein--{U}hlenbeck
  flow.
\newblock In \emph{2018 International Symposium on Information Theory and Its
  Applications (ISITA)}, pages 55--59. IEEE, 2018{\natexlab{b}}.

\bibitem[Wibisono et~al.(2017)Wibisono, Jog, and Loh]{WJL17}
Andre Wibisono, Varun Jog, and Po-Ling Loh.
\newblock Information and estimation in {F}okker--{P}lanck channels.
\newblock In \emph{2017 IEEE International Symposium on Information Theory
  (ISIT)}, pages 2673--2677. IEEE, 2017.

\bibitem[Xu and Raginsky(2017)]{xu2017information}
Aolin Xu and Maxim Raginsky.
\newblock Information-theoretic analysis of generalization capability of
  learning algorithms.
\newblock \emph{Advances in Neural Information Processing Systems}, 30, 2017.

\bibitem[Yuan et~al.(2023)Yuan, Fan, Liang, Wibisono, and Chen]{yuan2023class}
Bo~Yuan, Jiaojiao Fan, Jiaming Liang, Andre Wibisono, and Yongxin Chen.
\newblock {On a Class of Gibbs Sampling over Networks}.
\newblock In \emph{Proceedings of Thirty Sixth Conference on Learning Theory},
  Proceedings of Machine Learning Research, 2023.

\bibitem[Zamanlooy et~al.(2024)Zamanlooy, Asoodeh, Diaz, and
  Calmon]{zamanlooy2024mathrm}
Behnoosh Zamanlooy, Shahab Asoodeh, Mario Diaz, and Flavio~P Calmon.
\newblock {$\mathrm{E}_{\gamma}$-Mixing Time}.
\newblock In \emph{2024 IEEE International Symposium on Information Theory
  (ISIT)}, pages 3474--3479. IEEE, 2024.

\bibitem[Zhang et~al.(2021)Zhang, Cheng, and Reeves]{zhang2021convergence}
Yixing Zhang, Xiuyuan Cheng, and Galen Reeves.
\newblock {Convergence of Gaussian-smoothed optimal transport distance with
  sub-gamma distributions and dependent samples}.
\newblock In \emph{International Conference on Artificial Intelligence and
  Statistics}, pages 2422--2430. PMLR, 2021.

\bibitem[Zou et~al.(2025)Zou, Fan, Gao, and Wang]{zou2025convexity}
Jiayang Zou, Luyao Fan, Jiayang Gao, and Jia Wang.
\newblock Convexity of mutual information along the {F}okker-{P}lanck flow.
\newblock \emph{arXiv preprint arXiv:2501.05094}, 2025.

\end{thebibliography}

\end{document}